\documentclass{aims}
\usepackage{amsmath}
  \usepackage{paralist,subfig,caption,pbox,float,array}
  \usepackage{graphics} 

\usepackage{graphicx}  

 \usepackage[colorlinks=true]{hyperref}
\hypersetup{urlcolor=blue, citecolor=red}

\def\currentvolume{X}
 \def\currentissue{X}
  \def\currentyear{201X}
   \def\currentmonth{XX}
    \def\ppages{X--XX}
     \def\DOI{10.3934/xx.xx.xx.xx}
     
\newtheorem{theorem}{Theorem}[section]

\newtheorem{lemma}[theorem]{Lemma}

\theoremstyle{definition}

\title[Symmetric Periodic Orbits]{Symmetric Periodic Orbits in Three Sub-Problems of the $N-$body Problem}

\author[Nai-Chia Chen]{}


 \email{chen1945@umn.edu}

\thanks{This research is supported by NSF grant DMS-1208908.}

\begin{document}
\maketitle

\centerline{\scshape Nai-Chia Chen }
\medskip
{\footnotesize
 \centerline{School of Mathematics}
   \centerline{University of Minnesota}
   \centerline{Minneapolis, MN 55455, USA}
} 

\bigskip

 \centerline{(Communicated by the associate editor name)}

\begin{abstract}
We study three sub-problems of the $N$-body problem that have two degrees of freedom, namely the $n-$pyramidal problem, the planar double-polygon problem, and the spatial double-polygon problem. We prove the existence of several families of symmetric periodic orbits, including ``Schubart-like" orbits and brake orbits, by using topological shooting arguments.
\end{abstract}

\section{Introduction}

The Newtonian $n$-body problem studies the motion of $n$ point masses moving in the Euclidean space, under the influence of their mutual gravitational attraction. The motion is determined by the system of differential equations:
\[ \ddot{\mathbf{x}}_{i}=\sum_{j\neq i}^{n}m_{j}\frac{\mathbf{x}_{j}-\mathbf{x}_{i}}{|\mathbf{x}_{j}-\mathbf{x}_{i}|^{3}},\ \ \mathbf{x}_{i}\in\mathbf{R}^3,\]
where $\mathbf{x}_{i}$ and $m_{i}$ represent the position and the mass of the i-th mass respectively.

One of the difficulties in studying the $n-$body problem is due to the large number of variables. Therefore, sub-problems of the $n-$body problem that have lower degrees of freedom, which are usually obtained by adding constraints on the configurations, have received considerable attention. Two very popular examples are the collinear three-body problem and the isosceles three-body problem. 

On one hand, a well-known periodic orbit in the collinear three-body problem is the so called Schubart orbit, formed by two equal masses $m_{1}=m_{2}$ and another mass $m_{3}$, bouncing in between and having binary collisions with $m_{1}$, $m_{2}$ alternatively. This orbit was numerically found by Schubart~\cite{schubart1956}. Moeckel~\cite{moeckel2008} and Venturelli~\cite{venturelli2008} separately proved its existence by using topological shooting arguments and variational methods. On the other hand, in the isosceles three-body problem, Broucke found a symmetric periodic orbit ~\cite{broucke1979}, called Broucke orbit or a ``Schubart-like" orbit, formed by two equal masses $m_{1}=m_{2}$ whose positions are symmetric with respect to a fixed axes, along with a third mass $m_{3}$ that is moving up and down on the axes. See table~\ref{orbits_table}(a) for the Broucke orbit. Broucke orbit has been proved to exist by Shibayama~\cite{shibayama2011} and Mart\'{i}nez~\cite{martinez2012} separately. A ``Schubart-like"(Broucke) orbit is similar to Schubart orbit in the sense that in one period, they both have two singularities due to binary collisions, and that when a binary collision occurs, the third mass reaches its maximum distance to the origin.

The motivation for this paper comes from the work of Mart\'{i}nez~\cite{martinez2012,martinez2013} and Chen~\cite{chen2013}.  Mart\'{i}nez~\cite{martinez2012} studies certain Hamiltonian systems with two degrees of freedom and provides sufficient conditions for the existence of doubly symmetric ``Schubart-like" periodic orbits (DSSP orbits) in these systems. These sufficient conditions are applied to three examples, namely the $n-$pyramidal problem, the $2n$-planar problem, and the double-polygon problem. However, the double-polygon problem fails one of these sufficient conditions. Recently, Mart\'{i}nez~\cite{martinez2013} extended her previous results. While the previously proved Schubart-like orbits have only one singularity in a half period (called 0-DSSP orbits), the orbits in the new paper have a sufficiently large number of singularities in a half period (called k-DSSP orbits), yet the existence proof requires one hypothesis that is verified only for several values of $n$. On the other hand, another six families of periodic orbits in the isosceles three-body problem are proved to be existed by using topological shooting arguments~\cite{chen2013}. These are the so called brake orbits; that is, these orbits have zero initial velocity. 

In this paper, we apply the framework from~\cite{chen2013} to the following three problems: the $n-$pyramidal problem, the planar double-polygon problem, and the spatial double-polygon problem. See table~\ref{summary}. First, the $n-$pyramidal problem is a spatial problem which consists of $n$ masses whose configuration forms a planar regular $n-$gon, along with an additional mass lying on the vertical axes crossing the center of the $n-$gon. We remark that when $n=2$, the $n-$pyramidal problem is identical to the planar isosceles three-body problem. Second, in the planar double-polygon problem, the configuration of the $2n$ bodies forms two regular $n-$gons in the plane. Third, in the spatial double-polygon problem, the configuration of the $2n$ bodies forms an anti-prism. We remark that the spatial double-polygon problem is a new example that has not appeared in ~\cite{martinez2012,martinez2013}, and it is a special case of the dihedral n-body problem. The dihedral n-body problem is proposed by Ferrario and Portaluri in~\cite{ferrario}, where they find all central configurations and compute the dimensions of the stable/unstable manifolds, while the existence of other periodic orbits, except for the relative equilibria, has not been studied.

We prove the existence of several families of periodic orbits in the $n-$pyramidal problem and the spatial double-polygon problem. Representative orbits in the isosceles problem can be found in table~\ref{orbits_table}. To picture orbits of the same types in other problems, for the $n-$pyramidal problem, one may imagine replacing the two symmetric bodies in the isosceles problem with $n$ bodies lying on the vertices of a regular $n-$gon; for the spatial double-polygon problem, one furthermore replaces the third body with another $n$ bodies lying on the vertices of another regular $n-$gon. As for the planar double-polygon problem, unfortunately, as in~\cite{martinez2012}, there are difficulties to apply our arguments. Nonetheless, we complete Mart\'{i}nez's existence proof for 0-DSSP orbits, in which one of her three sufficient conditions fails, by showing that two of those conditions are enough to ensure the existence of 0-DSSP orbits. Existence of general k-DSSP orbits for the planar double-polygon problem are not proved here and will require further work. We summarize our conclusions in table~\ref{summary}.

Compared to Mart\'{i}nez's work~\cite{martinez2012,martinez2013}, we have found new families of periodic orbits, and our proofs are significantly simplified; we provide sufficient conditions that are looser and rigorously verified. Moreover, while the orbits in \cite{martinez2012,martinez2013} must have either one or a sufficiently large number of singularities, we prove the existence of periodic  orbits for any positive number of singularities in a half period.

The paper is organized as follows. Section 2 introduces two coordinate systems, both of which will be used in our proofs. Section 3 provides sufficient conditions for the existence of periodic brake orbits and Schubart-like orbit; it turns out that our sufficient conditions will be boiled down to the behaviours of two orbits. Section 4 proves theorems about the two orbits just mentioned. Finally, in Section 5, we apply our theorems to the three problems.

\section{Two Coordinate Systems}
Following the setting in Mart\'{i}nez's papers~\cite{martinez2012,martinez2013}, we consider Lagrangian systems with two degrees of freedom of the following form:
\begin{equation}
L(\mathbf{q},\dot{\mathbf{q}})=\frac{1}{2}\dot{\mathbf{q}}^{T}A\dot{\mathbf{q}}+\mathcal{U}(\mathbf{q}),
\end{equation}\label{lagrangian}
where $\mathbf{q}$ lies in an open set of $\mathbf{R}^{2}$, $A=diag(a_{1},a_{2})$ is a constant diagonal matrix with $a_{1}>0, a_{2}>0$, and $\mathcal{U}(\mathbf{q})$ satisfies certain assumptions that will be stated shortly.

Similar to~\cite{chen2013}, we will define two coordinate systems; one will be called Devaney's coordinates, and the other will be called the new coordinates.

\subsection{Devaney and Mart\'{i}nez's Coordinates~\cite{martinez2012,martinez2013}}
Mart\'{i}nez generalizes Devaney's coordinate system (which uses McGehee-type coordinates~\cite{mcgehee}) for the isosceles problem to the Lagrangian system ~(\ref{lagrangian}). In this subsection, we summarize the relevant results and the coordinate system used in ~\cite{martinez2012,martinez2013}. 

First, the size variable $r\geq 0$ is defined by $r^2=\mathbf{q}^{T}A\mathbf{q}$. The shape variable $\phi$ is defined by $\bar{\mathbf{q}}=(\cos\phi,\sin\phi)$, where $\bar{\mathbf{q}}=\frac{1}{r}\sqrt{A}\mathbf{q}$ and $\Vert\bar{\mathbf{q}}\Vert^{2}=\bar{\mathbf{q}}^{T}\bar{\mathbf{q}}=1$. 

Before introducing more variables, we state the assumptions that the potential function $\mathcal{U}(\mathbf{q})$ must satisfy as follows:
\begin{enumerate}
\item[A.1] $U(\mathbf{q})$ is a homogeneous function of degree -1 such that $\mathcal{U}(\mathbf{q})=V(\phi)/r$, where
\[V(\phi)=\frac{\beta_{1}}{\sin(\theta_{b}-\theta)}+\frac{\beta_{2}}{\sin(\theta-\theta_{a})}+\widehat{V}(\phi),\]
with $\beta_{1}>0, \beta_{2}\geq 0$ constants, where $\beta_{2}=0$ if and only if
 $\phi_{b}-\phi_{a}=\pi$, and $\widehat{V}(\phi)>0$ is a smooth (at least $\mathcal{C}^3$) bounded function in $\phi_{a},\phi_{b}$. Additionally, in the case $\phi_{b}-\phi_{a}<\pi$, we define $f(\phi)=\sin(\phi-\phi_{a})\sin(\phi_{b}-\phi)$; in the case $\phi_{b}-\phi_{a}=\pi$, we define $f(\phi)=\sin(\phi_{b}-\phi)$, so that $f(\phi)V(\phi)$ is bounded.
Furthermore, the critical values of $V(\phi)$ is non-degenerate, that is, if $V'(\phi_{\ast})=0$, then $V''(\phi_{\ast})\neq 0$.
\item[A.2] $V(\phi)$ is symmetrical with respect to $\phi_{m}:=(\phi_{a}+\phi_{b})/2$.
\item[A.3] $V(\phi)$ has exactly three critical points in $(\phi_{a},\phi_{b})$. They are $\phi_{L}<\phi_{m}<\phi_{R}$.
\end{enumerate}
\textbf{Remark on A.1} The potential function $\mathcal{U}(\mathbf{q})$ has singularities at $r=0$ and at $\phi=\phi_{a},\phi_{b}$. Physically, $r=0$ represents total collision of all masses, and $\phi=\phi_{a}$ and $\phi_{b}$ represent the partial collisions of certain masses, which will be referred to as $a-$collisions and $b-$collisions respectively. \\
\textbf{Remark on A.3} Mart\'{i}nez~\cite{martinez2012} studies both the cases when $V(\phi)$ has either one or three critical points. The case when $V(\phi)$ has one critical point has been completely treated; therefore, we will omit this case.

In Devaney's coordinates, the Lagrangian becomes
\[L(r,\dot{r},\phi,\dot{\phi})=\frac{1}{2}\dot{r}^2+\frac{1}{2}r^2\dot{\phi}^2+\frac{1}{r}V(\phi).\]

Furthermore, we define the size velocity $v$ by $v=r^{\frac{1}{2}}\dot{r}$ and the shape velocities $w$ by $w=\dot{\phi}r^{3/2}\sqrt{\frac{f(\phi)}{V(\phi)}}$.  

Let $t$ be the original time scale and define
\[W(\phi)=f(\phi)V(\phi), \ \ F(\phi)=\frac{f(\phi)}{\sqrt{W(\phi)}}.\]
After a change of time scale by $\frac{dt}{ds}=r^{\frac{3}{2}}F(\phi)$, the equations of the system become 
\begin{equation}
\begin{aligned}
\frac{dr}{ds}=&r v F(\phi) \\
\frac{dv}{ds}=&F(\phi)(2hr-\frac{v^2}{2})+\sqrt{W(\phi)} \\
\frac{d\phi}{ds}=&w \\
\frac{dw}{ds}=& -\frac{vw}{2}F(\phi)+\frac{W'(\phi)}{W(\phi)}(f(\phi)-\frac{w^2}{2})+f'(\phi)(1+\frac{f(\phi)}{W(\phi)}(2hr-v^2)),
\end{aligned}\label{diffeqn_devaney}
\end{equation} 
where $h$ is the energy of the system, and the energy equation is
\[\frac{w^2}{2f(\phi)}-1=\frac{f(\phi)}{W(\phi)}(rh-\frac{v^2}{2}).\]

The \textit{energy manifold} (with $h=-1$) is defined as the set
\begin{equation}
\{(r,v,\phi,w): r\geq 0, \phi_{a}\leq\phi\leq\phi_{b}, \frac{w^2}{2f(\phi)}-1=\frac{f(\phi)}{W(\phi)}(-r-\frac{v^2}{2})\},
\end{equation}
and the \textit{collision manifold} is defined as the set 
\begin{equation}
\{(r,v,\phi,w): r=0, \phi_{a}\leq\phi\leq\phi_{b}, \frac{w^2}{2f(\phi)}+\frac{f(\phi)}{W(\phi)}\frac{v^2}{2}=1\}.
\end{equation}

The collision manifold is a two-dimensional invariant manifold and is topologically a sphere with four holes. See figure~\ref{collision_manifold_devaney}. We note that the system of differential equations~(\ref{diffeqn_devaney}) no longer has singularities; the singularities due to partial collisions have been regularized, and the total collision singularity is replaced by the the collision manifold, an invariant set for the flow.
 
\begin{figure}[ht]
\centering
\includegraphics[width=0.5\textwidth]{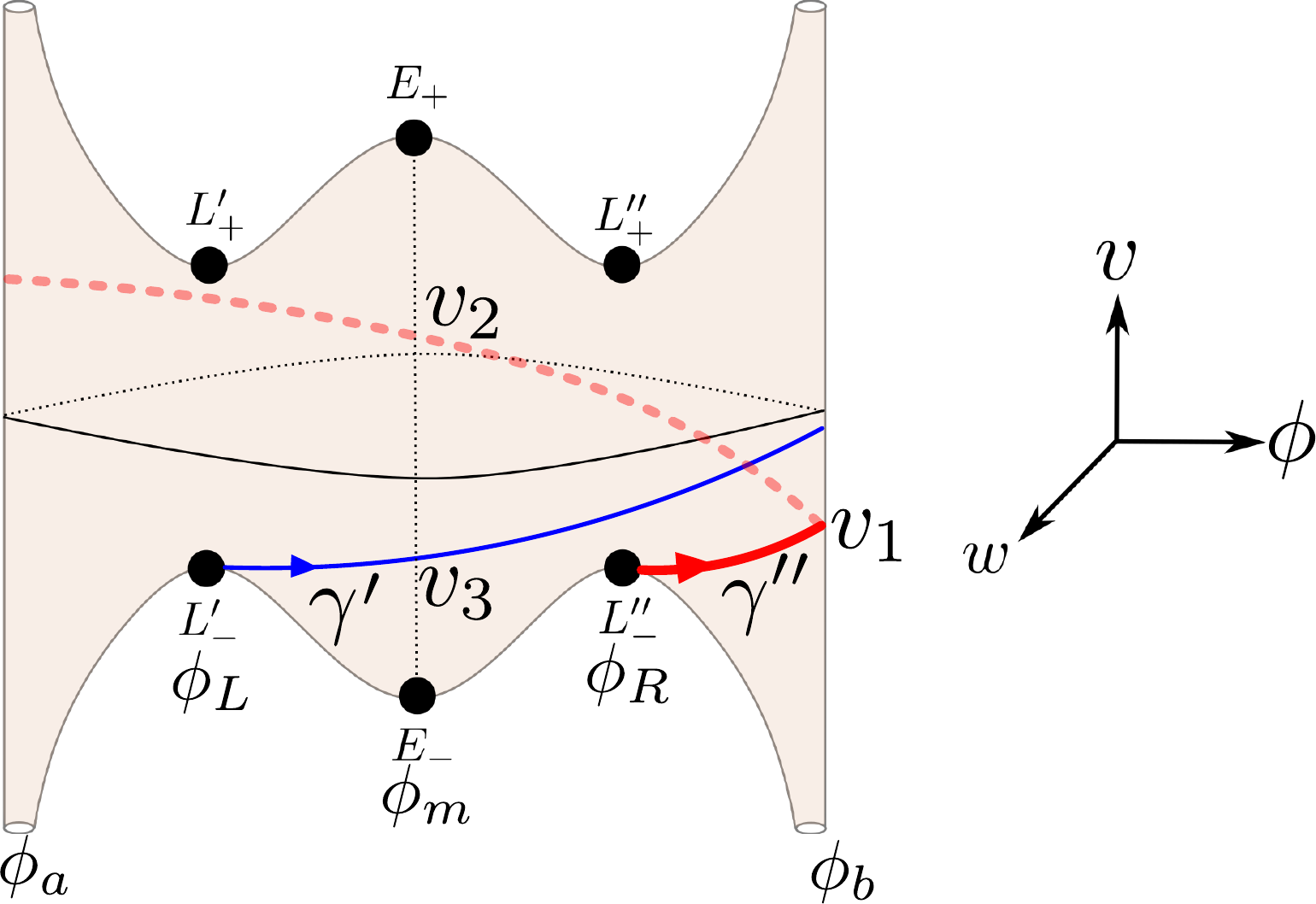}
\caption{The collision manifold and the two branches $\gamma',\gamma''$.}
\label{collision_manifold_devaney}
\end{figure}

The system~\ref{diffeqn_devaney} has exactly six equilibrium points: the Lagrange equilibria 
$L'_{\pm}=(0,\pm\sqrt{2V(\phi_{L})},\phi_{L},0)$, $L''_{\pm}=(0,\pm\sqrt{2V(\phi_{R})}, \phi_{R},0)$, and the Euler equilibria $E_{\pm}=(0,\pm\sqrt{2V(\phi_{m})},\phi_{m},0)$. We call these equilibrium points Lagrange/Euler equilibria because they share the same hyperbolic properties with those equilibria in the isosceles three-body problem. When restricted to the triple collision manifold, both $W^{u}(L'_{-})$ and $W^{u}(L''_{-})$ are one-dimensional, while $W^{u}(E_{-})$ is two-dimensional. We denote the branches of $W^{u}(L'_{-})$ and $W^{u}(L''_{-})$ that initially have $w\geq 0$ by $\gamma'$ and $\gamma''$ respectively. When restricted to the energy manifold, all equilibria are hyperbolic, and $dim(W^{s}(L''_{-}))=2$, $dim(W^{u}(L''_{-}))=1$, $dim(W^{u}(E_{-}))=2$. Once we have $W^{s,u}(L''_{-})$, we may find the stable and unstable manifolds of other equilibria by symmetries.

\subsection{New Coordinates \cite{moeckel2012,chen2013}}

The new coordinate system uses four variables: $r$,$v$,$\theta$, and $w$. The size variable $r\geq 0$ and the size velocity $v$ are defined the same as in Devaney's coordinates. That is, $r^2=\mathbf{q}^{T}A\mathbf{q}$ and $v=r^{\frac{1}{2}}\dot{r}$. We write $\mathbf{q}=(q_{1},q_{2})$ and $\bar{\mathbf{q}}=1/r\sqrt{A}\mathbf{q}$. The new shape variable $\theta$ is defined by $\mathbf{\bar{q}}=(c_{1}(\theta),c_{2}(\theta))$. In application, the choice of  $(c_{1}(\theta),c_{2}(\theta))$ depends on the range of $\mathbf{\bar{q}}$. If $\mathbf{\bar{q}}$ lies on the half unit circle; that is, if $q_{2}\in\mathbf{R}$, then we consider the stereographic projection of the segment $\{(0,\sin\theta),\theta\in \mathbf{R}\}$ from $(-1,0)$ to the half unit circle, as shown in figure~\ref{shape_variable}(a), and one obtains $(c_{1}(\theta),c_{2}(\theta))=(\frac{\cos^{2}\theta}{1+\sin^{2}\theta},\frac{2\sin\theta}{1+\sin^{2}\theta})$. 
If $\mathbf{\bar{q}}$ lies on the quarter unit circle in the first quadrant; that is, if $q_{2}\geq 0$, then we consider the parallel projection of the segment $\{(0,\sin\theta),\theta\in \mathbf{R}\}$ along the direction $(1,1)$ to the half unit circle, as shown in figure~\ref{shape_variable}(b), and one obtains $(c_{1}(\theta),c_{2}(\theta))=(\frac{\sqrt{\cos^2\theta+1}-\sin\theta}{2},\frac{\sqrt{\cos^2\theta+1}+\sin\theta}{2})$. Note that we allow the variable $\theta$ to vary from $-\infty$ to $\infty$; this gives a multiple cover of the half unit circle or quarter circle, with branched points at the endpoints of circles which correspond to $a-$collision or $b-$collision singularities. In both cases, $c'_{1}(\theta)^2+c'_{2}(\theta)^2=\frac{\cos^2\theta}{c(\theta)}$ for some analytic function $c(\theta)$, with $c(\theta)>0$ and $\frac{c'(\theta)}{\cos\theta\sin\theta}$ being analytic. Specifically,  $c(\theta)=(1+\sin^{2}\theta)^2 /4$ for the first case, and $c(\theta)=1+\cos^{2}\theta$ for the second case. Finally, the shape velocity $w$ is defined by $w=\dot{\theta}r^{3/2}\frac{\cos^2\theta}{c(\theta)}$.

\textbf{Remark.} The shape variable $\phi$ in Devaney's coordinates and the shape variable $\theta$ in the new coordinates are related by $(c_{1}(\theta),c_{2}(\theta)=(\cos\phi,\sin\phi)$. There is no simple expression that relates the $w$ variable in Devaney's coordinates with the $w$ variable in the new coordinates.

\begin{figure}[ht]
\centering
\subfloat[$q_{2}\in\mathbf{R}$]{{\includegraphics[width=0.4\textwidth]{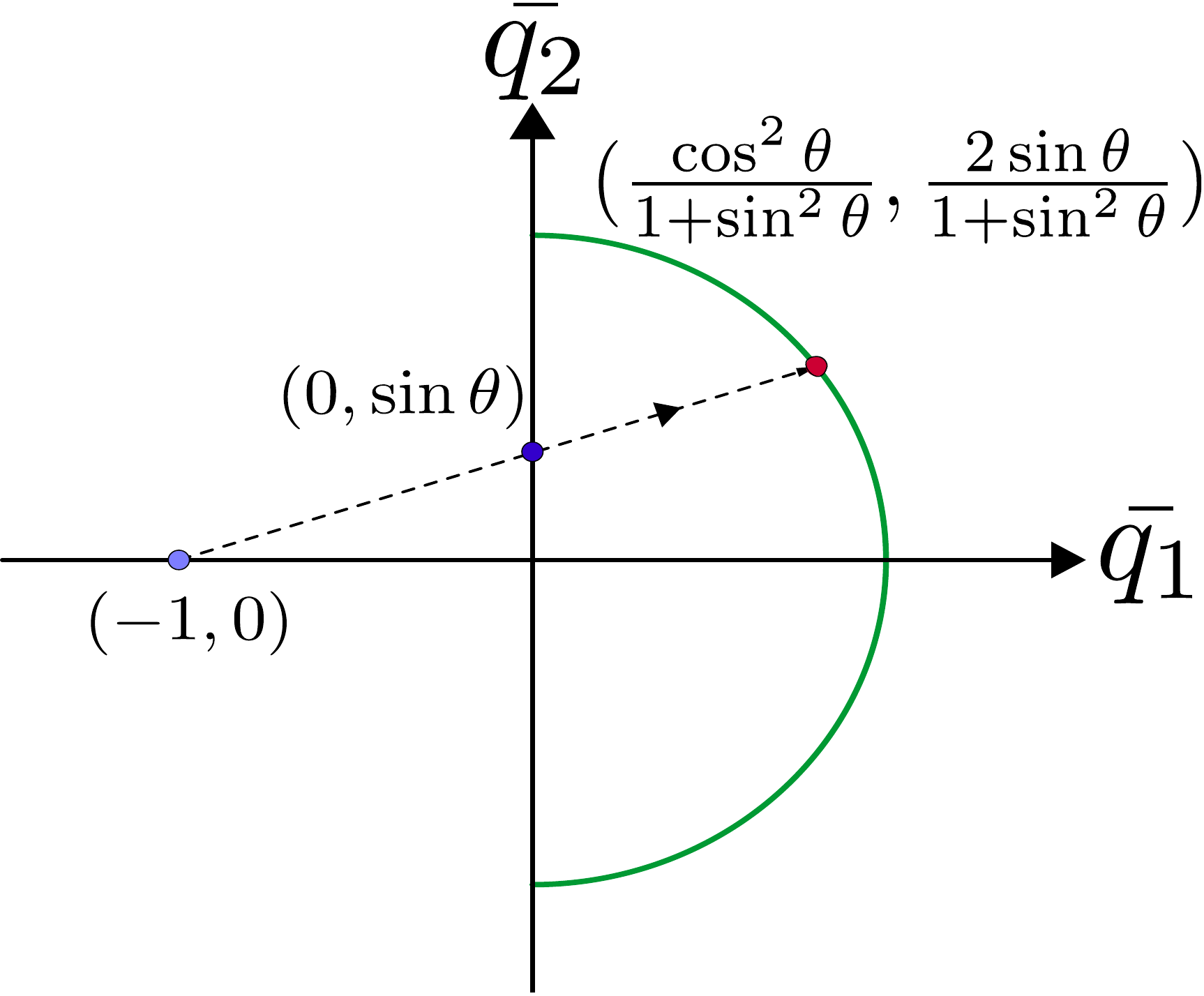}}}\hfill
\subfloat[$q_{2}\geq 0$]{{\includegraphics[width=0.4\textwidth]{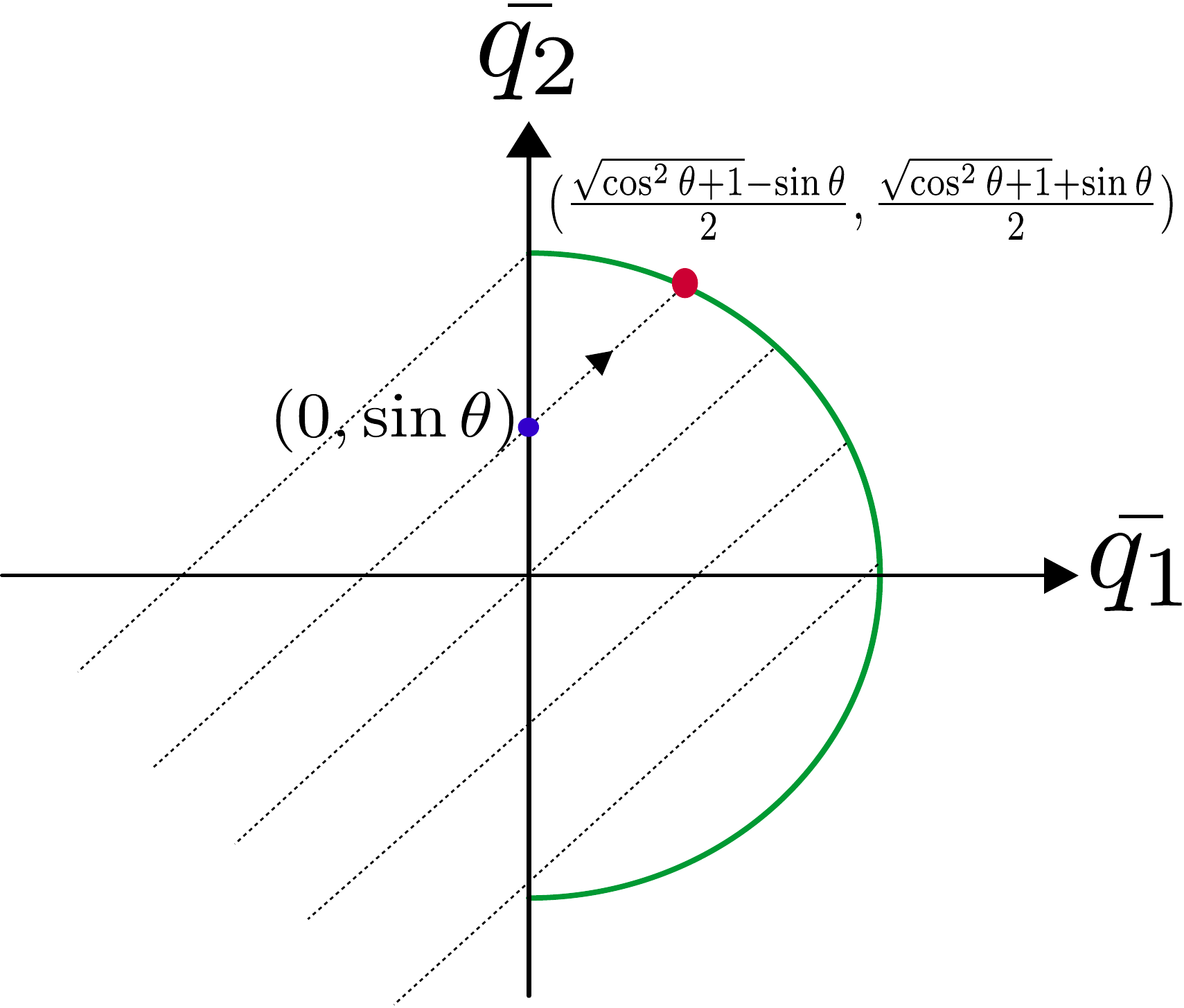}}}
\caption{The shape variable $\theta$.}
\label{shape_variable}
\end{figure}

In the new coordinates, the Lagrangian becomes
\begin{equation*}\label{Lagrangian_new}
L(r,\dot{r},\theta,\dot{\theta})=\frac{1}{2}{\dot{r}^2}+\frac{1}{2}r^2\dot{\theta}^2\frac{\cos^{2}\theta}{c(\theta)}+\frac{1}{r} V(\theta).
\end{equation*} 
The energy equation becomes
\[\frac{1}{2}v^2\cos^{2}\theta+\frac{1}{2}w^2 c(\theta)-\mathcal{W}(\theta)=rh\cos^2\theta,\]
where $\mathcal{W}(\theta):=\cos^2\theta V(\theta)$, $h$ is the energy of the system, and with an abuse of notation, $V(\phi)=V(\theta)$.

After a change of time scale by $\frac{dt}{ds}=r^{\frac{3}{2}}\cos^{2}\theta$,
the system of differential equations becomes
\begin{equation}
\begin{aligned}
\frac{dr}{ds}=&r v \cos^2\theta \\
\frac{dv}{ds}=&\frac{1}{2}v^2\cos^{2}\theta+w^2 c(\theta)-\mathcal{W}(\theta) \\
\frac{d\theta}{ds}=&w c(\theta) \\
\frac{dw}{ds}=&\mathcal{W}'(\theta)-\frac{1}{2}vw\cos^2\theta+\cos\theta\sin\theta(2r+v^2-\frac{1}{2}w^2\frac{c'(\theta)}{\sin\theta\cos\theta}).
\end{aligned}\label{newequation}
\end{equation}
The derivation of the system~(\ref{newequation}) can be found in the Appendix.

We remark that the singularities due to $a-$collision and $b-$collision have been regularized, since by our choice, $\mathcal{W}(\theta)$ and $\frac{c'(\theta)}{\cos\theta\sin\theta}$ appearing in equation~(\ref{newequation}) are both analytic. Moreover, the total collision that corresponds to $r=0$ has been blow-up to the collision manifold, which is an invariant set of the flow.

The \textit{energy manifold} (with $h=-1$) is defined as the set
\begin{equation}
\mathcal{P}_{1}=\{(r,v,\theta,w): r\geq 0,\frac{1}{2}v^2\cos^{2}\theta+\frac{1}{2}w^2 c(\theta)-\mathcal{W}(\theta)=-r\cos^2\theta.\},
\end{equation} Notice that in the energy equation, $c(\theta)>0$, and thus the variable $w$ can be solved as a two-valued function of $(r,v,\theta)$. The energy manifold can be visualized as two copies (one with $w\geq 0$ and another one with $w\neq 0$) of its projection to the $(r,v,\theta)$-space. See figure~\ref{energymanifold2}.
The \textit{collision manifold} is defined as the subset $P_{1}\cap \{r=0\}$, in which the flow is gradient-like with respect to the variable $v$, that is, $\frac{dv}{ds}\geq 0$.

\begin{figure}[ht]
 \begin{center}
  \includegraphics[scale=0.6]{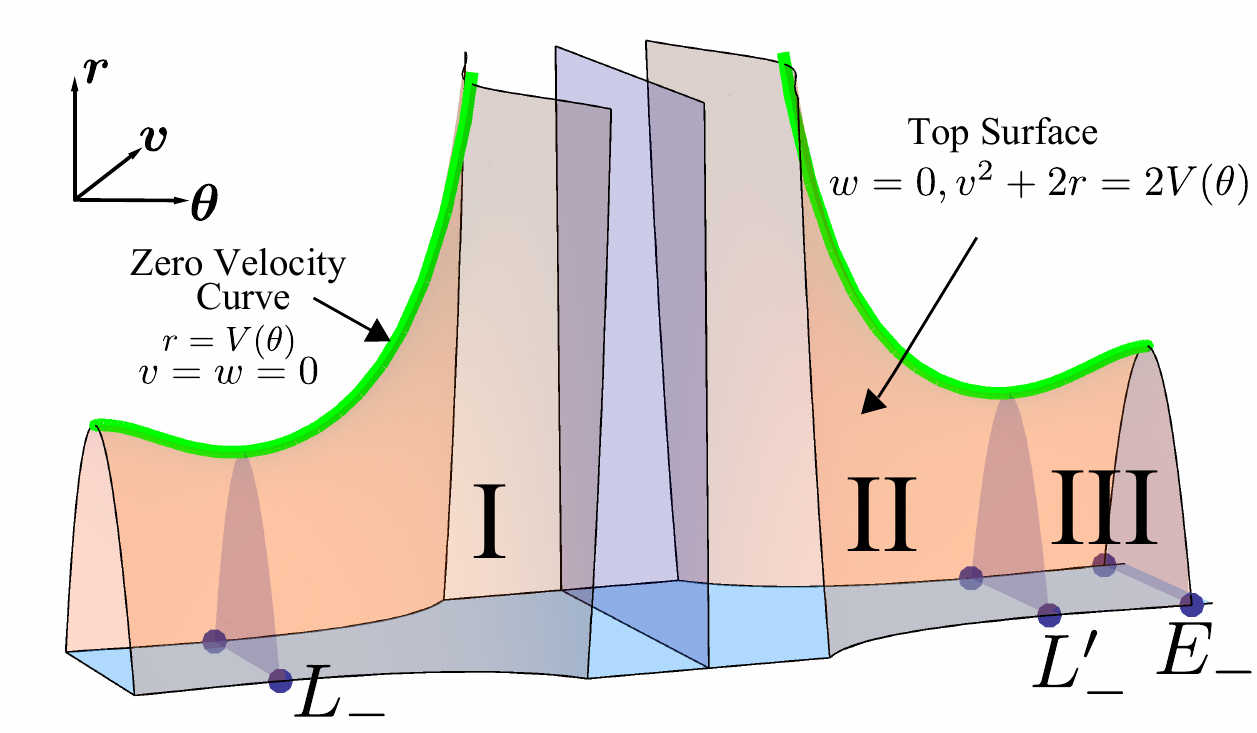}  
 \end{center}
  \caption{One copy of the projection of the energy manifold to the $(r,v,\theta)$-space. The top surface, the floor surface, and the zero velocity curve are given by $w=0$, $r=0$, and $v=w=0$ respectively.}
  \label{energymanifold2}
 \end{figure}

Now we summarize the flow on the energy manifold. The equilibrium points of the system are at the Lagrange equilibria $(r,\theta,v,w)=(0,\pm\theta^{*}+k\pi,\pm v^{*},0)$ and the Euler equilibria $(r,\theta,v,w)=(0,k\pi,\pm\sqrt{2V(0)},0)$, where $\theta^{*}$ is the unique critical point of $V(\theta)$ for $\theta\in (0,\pi/2)$, $k\in \emph{Z}$ and $v^{*}=\sqrt{2V(\theta^{*})}$.
Denote the equilibrium points that are relevant to our proofs by
\begin{equation*}
\begin{aligned}
&L_{\pm}=(0,\theta^{*}-\pi,\pm v^{*},0), \ \ L'_{\pm}=(0,-\theta^{*},\pm v^{*},0), \ \ E_{\pm}=(0,0,\pm \sqrt{2V(0)},0).
\end{aligned}
\end{equation*} 
The equilibrium points $L_{+}$ and $L_{-}$ are connected by a homothetic orbit; so are $L'_{+}$ and $L'_{-}$. Also $E_{+}$ and $E_{-}$ are connected by a homothetic orbit. When restricted to the collision manifold,  both $W^{s}(L_{-})$ and $W^{u}(L_{-})$ have dimension one, while $W^{s}(E_{+})$ has dimension two. When restricted to the energy manifold $P_{1}$, $W^{s}(L_{-})$ has dimension two, $W^{u}(L_{-})$ has dimension one, and $W^{s}(E_{+})$, which is contained in the collision manifold, has dimension two.

Next we divide a part of the energy manifold $P_{1}$ into several regions. Define
\begin{equation*}
\begin{aligned}
R_{I}=&P_{1}\cap\{\theta\in[\theta^{*}-\pi,-\pi/2],w\geq 0\}, &&Q_{I}=P_{1}\cap\{\theta\in[\theta^{*}-\pi,-\pi/2],w\leq 0\}, \\
R_{II}=&P_{1}\cap\{\theta\in[-\pi/2,-\theta^{*}],w\geq 0\}, &&Q_{II}=P_{1}\cap\{\theta\in[-\pi/2,-\theta^{*}],w\leq 0\}, \\
R_{III}=&P_{1}\cap\{\theta\in[-\theta^{*},0],w\geq 0\}.
\end{aligned}
\end{equation*}

We call the planes $\theta=\theta^{\ast}-\pi$ and $\theta=-\pi/2$ the left and right walls of $R_{I}$ respectively with similar definitions for the other regions.
Now we summarize the properties of the flow in each region. 

\begin{lemma} \label{lemma_invariant}
\begin{enumerate}
\item[(i)] Regions $R_{I}, R_{III}$ are \textit{flowing-rightward}. With the exception of the Lagrange homothetic orbits and orbits in the stable manifold of $E_{+}$, orbits cross theses regions as follows: every orbit beginning in the left wall of 
$R_{I}(R_{III})$ crosses region $R_{I}(R_{III})$ and exits at the right wall.

\item[(ii)]
Region $R_{II}$ is \textit{flowing-leftward in backward-time}. Except 
for the Lagrange homothetic orbits, any backward-time orbit beginning in the 
right wall of $R_{II}$ can be followed back to the left wall.
Furthermore, forward orbits beginning in the left wall of $R_{II}$
either leave $R_{II}$ through the right wall, 
leave $R_{II}$ through the surface $\{w=0\}$, or converge to
one of the Lagrange restpoints $L'_{\pm}$ in the right wall. 
In other words, forward orbits would not stay in $R_{II}$ forever unless 
they belong to $W^{s}(L'_{\pm})$. 

\item[(iii)] Region $Q_{I}$ is flowing-rightward in backward-time. Except for the Lagrange homothetic orbit, every backward-time orbit beginning in the left wall of $Q_{I}$ can be followed back to the right wall. Region $Q_{II}$ is flowing-leftward; except for the Lagrange homothetic orbit, every forward orbit beginning in region $Q_{II}$ crosses the region and exits at the left wall.

\end{enumerate}

\end{lemma}
\begin{proof}
The proof for the case when the system of differential equations is given by the isosceles problem can be found in~\cite{moeckel2012,chen2013}, where their proof does not involve the formula of $V(\theta)$. The two main ingredients of that proof are: (1) On the top surface of the energy manifold, $w=0$, and $\frac{dw}{ds}=\cos^2\theta V'(\theta)$, so that an orbit cannot leave regions $R_{I}$ or $R_{III}$ through the top surface where $V'(\theta)\geq 0$. (2) The variable $w$ has a uniform bound, since $w^2\leq 2\mathcal{W}(\theta)/c(\theta)$ and both $\mathcal{W}(\theta)$ and $c(\theta)$ are are bounded and non-zero. Let $\lambda:=\sqrt{2r+v^2}$, then from~(\ref{newequation}), $\lambda \frac{d\lambda}{ds}=\frac{1}{2}vw^{2}c(\theta)$, $\frac{d\theta}{ds}=wc(\theta)$. Therefore, when reparametrized by $\theta$, the quantity $|\frac{d\lambda}{d\theta}|\leq|\frac{1}{2}w|$ has a uniform bound. 

Our systems have both these ingredients. Therefore, the lemma is true for our systems as well.
\end{proof}

We now list additional notations that will be used throughout the paper.
\begin{equation*}
\begin{aligned}
&\mbox{(Eulerian plane)}: Euler(\bar{\theta},+)=\{(r,v,\bar{\theta},w): w\geq 0\}, \ \ \mbox{where} \ \ \bar{\theta}=k\pi. \\
&\mbox{(Eulerian plane)}: Euler(\bar{\theta},-)=\{(r,v,\bar{\theta},w): w\leq 0\}, \ \ \mbox{where} \ \ \bar{\theta}=k\pi. \\
&\mbox{(Partial-collision plane)}: Partial(\bar{\theta},+)=\{(r,v,\bar{\theta},w): w\geq 0\}, \ \ \mbox{where} \ \ \bar{\theta}=\pi/2+k\pi. \\
&\mbox{(Partial-collision plane)}: Partial(\bar{\theta},-)=\{(r,v,\bar{\theta},w): w\leq 0\}, \ \ \mbox{where} \ \ \bar{\theta}=\pi/2+k\pi. \\
&\mbox{(The line} \ \ \theta=\bar{\theta}): S(\bar{\theta})=\{(r,v,\theta,w)\in P_{1}, v=0,\theta=\bar{\theta}, w\geq 0\}. \\
&\mbox{(Zero velocity curve)}: \mathcal{Z}=\{(r,v,\theta,w)\in P_{1}, v=w=0\}.
\end{aligned}
\end{equation*}
When considering the projection to the $(r,\theta)$-plane, we will call the line $\theta=k\pi$ the \textit{Eulerian line}, and the line $\theta=\pi/2+k\pi$ the \textit{partial-collision line} for any $k\in \mathbf{Z}$.

Moreover, the system~(\ref{newequation}) has symmetries. Let
\begin{equation}
\begin{aligned}
   &R_{1}: (r,v,\theta,w)\rightarrow (r,-v,\theta,-w)  \\
   &R_{2}: (r,v,\theta,w)\rightarrow (r,-v,-\theta,w) \\
   &T_{1}: (r,v,\theta,w)\rightarrow (r,v,\theta+\pi,w) \\ 
   &\mbox{Fix}(R_{i})=\{(r,v,\theta,w): R_{i}(r,v,\theta,w)=(r,v,\theta,w)\}, \ \ i=1,2.
\end{aligned} \label{symmetries}
\end{equation}
Then $\varphi(t)$ be a solution to X if and only $R_{i}\varphi(-t)$ and $T_{1}\varphi(t)$ are solutions to X for $i=1,2$.

\section{Theorems on the existence of periodic orbits}

In~\cite{martinez2013}, two families of Schubart-like periodic orbits, called $\mathcal{Z}-$family and $\mathcal{B}-$family, have been found but not completely rigorously proved. Here the letter $\mathcal{Z}$ stands for the zero velocity curve and $\mathcal{B}$ stands for a partial collision: An orbit in the $\mathcal{Z}$-family has at least one point on the zero velocity curve; an orbit in the $\mathcal{B}$-family has at least one point belonging to the partial collision configuration. We remark that in the isosceles problem, $\mathcal{Z}-$family orbits are identical to  type 1 periodic brake orbits in~\cite{chen2013}; we will refer to this family as $\mathcal{Z}1$-family in this paper. 

In this section, we provide theorems about the existence of $\mathcal{B}-$family, $\mathcal{Z}1-$family, and three additional families of periodic orbits: less-symmetric $\mathcal{B}-$family, $\mathcal{ZB}-$family, and $\mathcal{Z}5$-family. We now describe these orbits by their projection on the $(\theta,r)$-plane. See table~\ref{orbits_table} for these orbits. We denote the period by $T$ and the orbit by $\varphi(t)=(r(t),v(t),\theta(t),w(t))$. First, in a quarter of period, an orbit of $\mathcal{B}-$family starts from $\theta=-\pi/2$ with $v=0$ and reaches an Eulerian line $\theta=0$ or $\theta=-\pi$ orthogonally (i.e., $v=0$) at $T/4$; before a quarter of period, the orbit may cross the line $\theta=-\pi/2$ arbitrary many times. The next quarter orbit is obtained by reflecting the first quarter orbit with respect to the line $\theta=\theta(T/4)$, and the second half orbit is obtained by reflecting the first half orbit with respect to the line $\theta=\theta(T/2)$. Next, we describe less-symmetric $\mathcal{B}-$family. An orbit in this family starts from $\theta=-\pi/2$ with $v=0$ and reaches $\theta=\pi/2$ orthogonally (i.e., $v=0$) at $T/2$; before a half period, the orbit crosses the line $\theta=-\pi/2$ at least once, then crosses the Eulerian line $\theta=0$, and then crosses the line $\theta=\pi/2$ at least one. The second half orbit is obtained by reflecting the first half orbit with respect to the line $\theta=\theta(T/2)$. Next, an orbit in $\mathcal{ZB}-$family starts from the zero velocity curve and reaches a partial-collision line orthogonally at $T/2$. In the half period, the orbit crosses a partial collision line at least once, then crosses the Eulerian line, and then crosses another partial collision line at least once. The next half period is obtained by reflecting the first half orbit with respect to the line $\theta=\theta(T/2)$. Finally, an orbit in $\mathcal{Z}1-$family or $\mathcal{Z}5$-family starts at a brake time and reaches another brake time again at $t=T/2$, then the orbit retraces its first half orbit and reaches a brake time again at $t=T$; while orbits in $\mathcal{Z}1-$family cross the Eulerian line orthogonally at $T/4$, orbits in $\mathcal{Z}5-$family do not orthogonally cross any Eulerian lines or partial collision lines.

It may not be obvious at this point, but it will become clearer soon that the invariant manifolds of $L_{-}$ and $L'_{-}$ play a crucial role in the theorems.

Recall that the unstable manifold of $L_{-}$ and $L'_{-}$ are both one-dimensional and lie on the collision manifold; we denote the branches of $W^{u}(L_{-})$ and $W^{u}(L'_{-})$ that initially have $w\geq 0$ by $\gamma$ and $\gamma'$ respectively. See figure~\ref{four_branches}(a). And whenever they are well-defined, we denote the $v-$coordinates of the intersections of $\gamma$ with $\theta=-\pi/2$ and with $\theta=0$ by $v_{1}$ and $v_{2}$ respectively, and that of the intersections of $\gamma'$ with $\theta=0$ by $v_{3}$. Furthermore, we denote the branches of $W^{u}(L_{-})$, $W^{u}(L'_{-})$ that initially have $w\leq 0$ by $\gamma_{-},\gamma'_{-}$ respectively. See figure~\ref{four_branches}(b). From symmetries,  $\gamma_{-}$ and $\gamma'_{-}$ can be obtained by first reflecting $\gamma$ and $\gamma'$ with respect to the line $\theta=-\pi/2$ and then changing the positive values of $w$ to negative values.

When restricted to the collision manifold, the stable manifold of $L'_{-}$ is one-dimensional, and since $R_{II}$ is flowing-leftward in backward-time, the branch that has $w\geq 0$ can be followed in backward-time to intersect $\theta=-\pi/2$ at, say $v=v_{0}$. When restricted to the energy manifold, the stable manifold of $L'_{-}$ becomes two-dimensional, where the extra one dimension comes from the Lagrange homothetic orbit that connects $L'_{+}$ to $L'_{-}$. 

\textbf{(A crucial surface, Roof I)} The ``quadrant" of the surface $W^{s}(L'_{-})$ that lies in $R_{II}$ will be crucial to our proofs later. We will refer to this quadrant surface as Roof I. See figure~\ref{beta_family_proof}(a) for this surface. One edge of Roof I is the Lagrange homothetic orbit connecting $L'_{+}$ to $L'_{-}$, and the other edge is the unstable branch of $L'_{-}$ that lies in the collision manifold. Since region $R_{II}$ is flowing-leftward in backward-time, Roof I can be followed to reach the left wall of $R_{II}$, intersecting with the wall $Partial(\pi/2,+)$ and forming a curve with two endpoints on the collision manifold. One endpoint arises from the stable branch of $L'_{+}$, and therefore it has $v=-v_{1}$; the other endpoint arises from the stable branch of $L'_{-}$ that lies on the collision manifold, and therefore it has $v=v_{0}$. 

\textbf{(Roof II)} Another surface that will also be crucial to our proofs is the ``quadrant" of the surface $W^{s}(L_{-})$ that lies in region $Q_{I}$. We will refer to the quadrant surface as Roof II. See figure~\ref{beta_family_proof}(b). From symmetries, this surface can be obtained by first reflecting Roof I with respect to the plane $\theta=-\pi/2$ and then changing positive values of the $w$-coordinate to negative values. 

We now are ready to prove the existence of periodic orbits.
\begin{theorem}\label{thm_type1} If $v_{1}<0, v_{2}>0, v_{3}<0$, then there exists a T-periodic brake orbit of the following types:
\begin{itemize}
\item[(i)]\textbf{($\mathcal{Z}1-$family periodic orbits).}
In a quarter of period, the orbit starts at a brake time, then crosses the partial collision line $\{\theta=-\pi/2\}$ $k$ times, and then hits the Eulerian line orthogonally at $t=T/4$, i.e., $(\theta(T/4),v(T/4))=(0,0)$. See table~\ref{orbits_table}(f).
\item[(ii)]\textbf{($\mathcal{Z}5-$family periodic orbits).}
Assume additionally, $v_{2}\neq -v_{3}$. Let $(i,j)$ be any pair of positive integers. In a half period, the orbit starts at a break time, then crosses the partial-collision line $\{\theta=-\pi/2\}$ i times, then crosses the Eulerian line $\{\theta=0\}$, then crosses the partial-collision line $\{\theta=\pi/2\}$ j times, and then reaches zero velocity. See table~\ref{orbits_table}(g).
\end{itemize}

\end{theorem}

\begin{proof}
The case when the system is given by the isosceles three-body problem has been proved in theorem 5.4 and theorem 6.1 of~\cite{chen2013}. The only three differences between those theorems and the theorem here are as follows. First, $\mathcal{Z}1-$family (respectively $\mathcal{Z}5-$family) periodic orbits here were called type 1 (respectively type 5) periodic brake orbits in ~\cite{chen2013}.
Second, the statement of those theorems requires a condition on the mass of the third body, specifically, $0<m_{3}<\epsilon_{2}\approx 2.661993$, which is used to ensure $v_{1}<0, v_{2}>0, v_{3}<0$, and additionally $v_{2}\neq -v_{3}$ (for type 5 orbits). Here we replace that condition on $m_{3}$ by the condition $v_{1}<0, v_{2}>0$ and $v_{3}<0$. The third difference is only a change of naming; in that theorem, $\{\theta=-\pi/2\}$ is called a binary collision line, while we call the same line a partial collision line. After these modifications, that proof can be applied verbatim to prove our theorem here.

\end{proof}

\begin{theorem}{\textbf{($\mathcal{B}-$family Schubart-like orbits with $k=0$).}}\label{thm_schubart}
If $v_{1}<0$ and $v_{3}<0$, then there exists a Schubart-like orbit with the following properties: In a quarter of period, the orbit starts at the partial-collision line $\theta=-\pi/2$, then hits the Eulerian line orthogonally at $t=T/4$, i.e., $(\theta(T/4),v(T/4))=(0,0)$. See table~\ref{orbits_table}(a).
\end{theorem}

\begin{proof}
Mart\'{i}nez's idea~\cite{martinez2012}, expressed in the new coordinates, is to show that part of the line segment $S(-\pi)$ (which represents Eulerian shapes) can be followed to reach $\theta=-\pi/2$ (which represents partial-collision shapes) and form a continuous curve, whose one endpoint has $v>0$ and the other endpoint has $v<0$. Therefore, there is an orbit that starts with $\theta=-\pi, v=0$ and hits $\theta=-\pi/2$ orthogonally (i.e. $v=0$), which corresponds to a Schurbart-like periodic orbit. In contrast, we will shoot from the line segment $S(-\pi/2)$ and target to hit  $\theta=0$ orthogonally. Figure~\ref{thm_existence1sproof}(a) illustrates the idea of the proof. (\textbf{Remark.} Shooting from $S(-\pi)$ works fine in the new coordinates. Here we choose to shoot from $S(-\pi/2)$ simply because the proof will require fewer new notations and facilitate the proofs of further theorems.)

We start from $S(-\pi/2)$, whose one endpoint, say $q_{0}$, is in the collision manifold, and the other endpoint is at infinity. Our goal now is to construct part of $S(-\pi/2)$ that can be followed across region $R_{II}$. Note that by lemma~\ref{lemma_invariant}(ii), region $R_{II}$ is not flowing-rightward --- not the entire $S(-\pi/2)$ can be followed across $R_{II}$ under the flow.  We consider Roof I surface, whose intersection with the left wall of $R_{II}$ forms an curve with one endpoint having $(r,v)=(0,v_{0})$ and the other endpoint having $(r,v)=(0,-v_{1})$. Since $v_{0}<0<-v_{1}$, $S(-\pi/2)$ and the curve just mentioned must intersect; we denote the first intersection by $q_{I}$. Let $X_{I}$ be the part of $S(-\pi/2)$ that is between $q_{0}$ and $q_{I}$, then $X_{I}$ can be followed across region $R_{II}$ to reach the left wall of $R_{III}$, since Roof I serves as a trapping surface that prevents the image of $X_{I}$ from leaving $R_{II}$ through the surface $\{w=0\}$. Furthermore, since region $R_{III}$ is flowing-rightward, $X_{I}$ can be followed further to cross region $R_{III}$, and then form an arc, called $X_{II}$, on the right wall of $R_{III}$. We now investigate the two endpoints of $X_{II}$. Denote the $v-$coordinate of the intersection of the orbit of $q_{0}$ with $\theta=0$ by $\widehat{v}$. Since the flow in the collision manifold is gradient-like, $\widehat{v}> 0$. One endpoint of $X_{II}$, which arises from the orbit of $q_{0}$, has $v=\widehat{v}$. As for the other endpoint, orbits that start near a neighborhood of $q_{I}\in W^{s}(L'_{-})$ will follow the orbit of $q_{I}$ entering a neighborhood of $L'_{-}$ and then follow the branch $\gamma'$, and therefore the other endpoint of $X_{II}$ is at $v=v_{3}<0$. Since $v_{3}<0<\widehat{v}$, there exists a point on $X_{II}$ that has $v=0$, which corresponds to the desired periodic orbit. This completes the proof.

\end{proof}

\begin{figure}[ht]
\centering
\subfloat[The part of $S(-\pi/2)$ below Roof I can be followed across region $R_{II}$.]{{\includegraphics[width=0.5\textwidth]{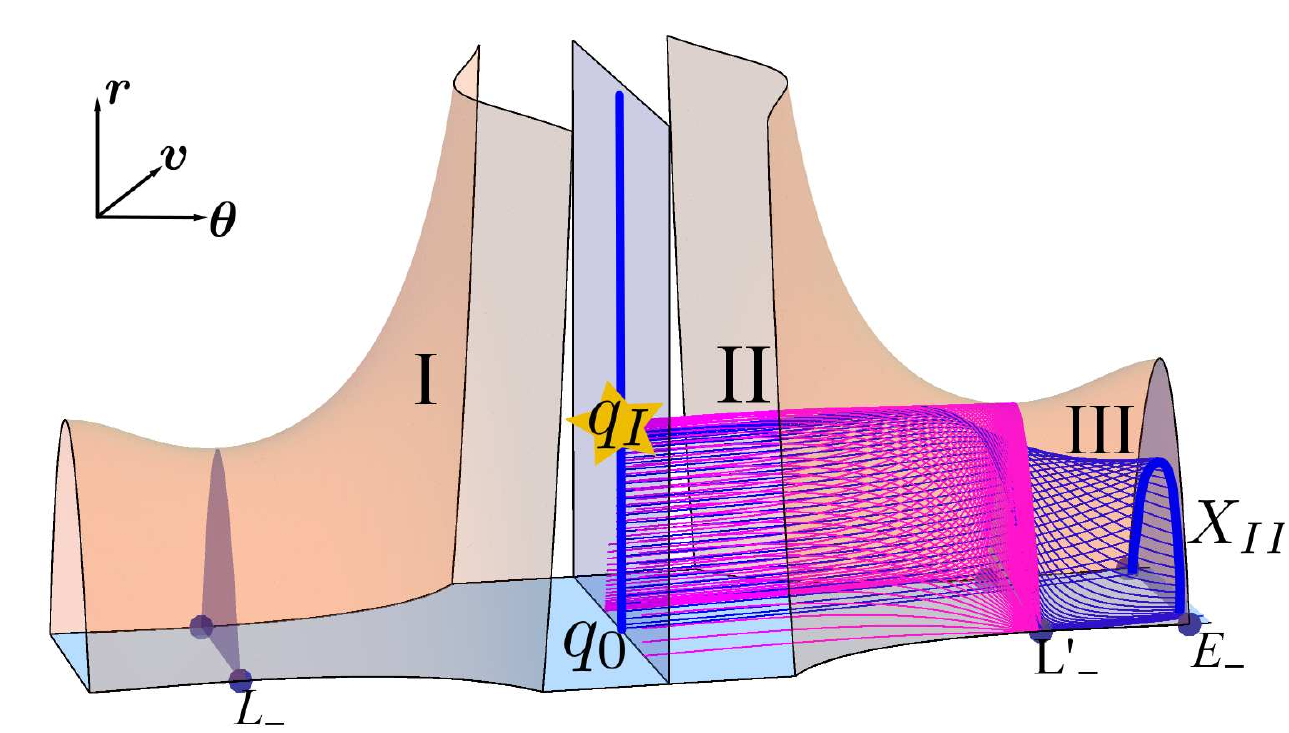}}}\hfill
\subfloat[The two branches $\gamma,\gamma'$ and the orbit starting from $q_{0}$]{{\includegraphics[width=0.5\textwidth]{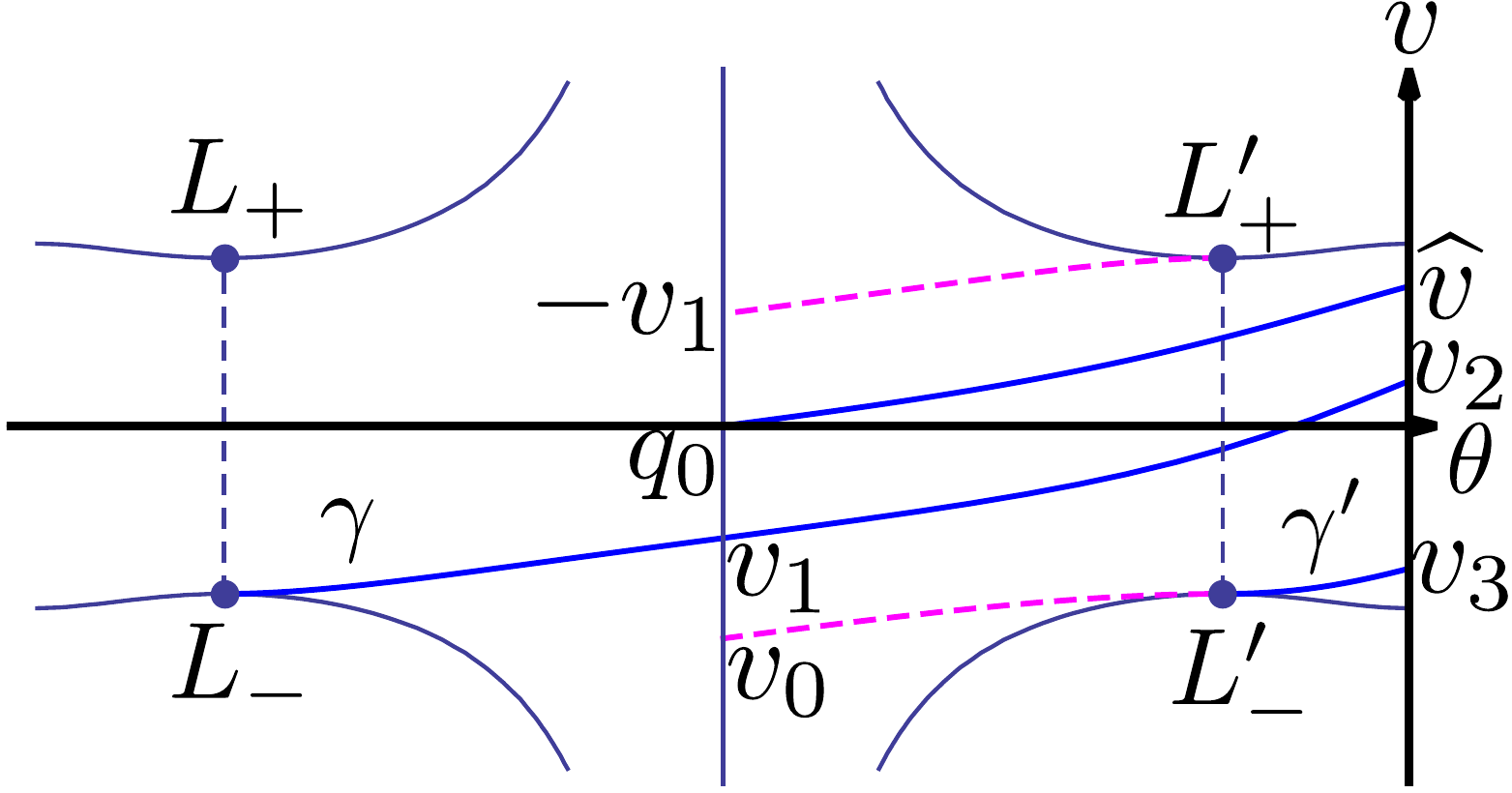}}}
\caption{Illustration for the proof of Theorem~\ref{thm_schubart}}
\label{thm_existence1sproof}
\end{figure}

\begin{theorem}{\textbf{($\mathcal{B}-$family Schubart-like orbits).}}\label{thm_schubart2}
If $v_{1}<0, v_{2}>0, v_{3}<0$, then for any $k\in N$, there exists a Schubart-like orbit with the following properties: In a quarter of period, the orbit starts at the partial-collision line $\theta=-\pi/2$, then crosses the partial collision line $\{\theta=-\pi\}$ $k$ times, and then hits one of the Eulerian lines $\theta=0$ (if $k$ is even) or $\theta=-\pi$ (if $k$ is odd) orthogonally at $t=T/2$. See table~\ref{orbits_table}(a)(b)(c).
\end{theorem}
\begin{proof}
The proof of this theorem is quite similar to that of theorem 5.4 of ~\cite{chen2013}.

The case $k=0$ has been proved in theorem~\ref{thm_schubart}. We start with the case $k=1$. Figure~\ref{beta_family_proof} illustrates the idea of the proof. Again we shoot from $S(-\pi/2)$. Recall that in the proof for the case $k=0$, we choose $X_{I}$ to be the part of $S(-\pi/2)$ that is below Roof I such that $X_{I}$ can be followed across $R_{II}$. Now for the case $k=1$, we construct $X'_{I}$ to be the part of $S(-\pi/2)$ that is above Roof I such that $X'_{I}$ can \textbf{not} be followed across $R_{II}$. Specifically, let $p_{I}$ be the last intersection of $S(-\pi/2)$ with the Roof I, and let $X'_{I}$ be the part of $X_{I}$ that is between $p_{I}$ and infinity. Then points on $X'_{I}$ will not cross $R_{II}$; instead, they leave $R_{II}$ through the top boundary surface $\{w=0\}$, then enter $Q_{II}$, and then reach the left wall of $Q_{II}$, forming an arc, say $X_{II}$. One of the endpoints of $X_{II}$ has $v=v_{1}$, which arise from $\gamma'_{-}$, and the other endpoint is at infinity. To construct the part of $X_{II}$ that can cross region $Q_{I}$, which is flowing-rightward in backward-time, we consider Roof II, whose intersection with $Partial(-\pi/2,-)$ forms an arc whose endpoints are at $v=v_{0}$ and at $v=-v_{1}$. Let $X'_{II}$ be the part of $X_{II}$ that is below the arc just mentioned, than $X'_{II}$ can be followed across region $Q_{II}$, since Roof II serves as a trapping surface here. Points on $X'_{II}$ can furthermore reach $Euler(-\pi,-)$, forming an arc, say $X_{III}$, whose endpoints are at $(r,v)=(0,v_{2})$ and $(r,v)=(0,v_{3}).$ This proves the case $k=1$.

If instead of considering $X'_{II}$, we consider $X''_{II}$ to be the part of $X_{II}$ that is above Roof II, then $X''_{II}$ will leave $Q_{I}$ through the top surface $\{w=0\}$, enter $R_{I}$, and reach the right wall of $R_{I}$, forming an image curve on $Partial(\pi/2,+)$, with the endpoints at $(r,v)=(0,v_{1})$ and infinity respectively. The situation now is quite similar to that when we have $S(-\pi/2)$ --- we obtain a curve that connects the collision manifold and infinity. By considering the relative position of the image curves with respect to Roof I or II, one may construct part of the image curves that travels between $\{w\geq 0\}$,$\{w\geq 0\}$ as many times as desired and then forms a curve on $Euler(0,+)$ or $Euler(-\pi,0)$, with the endpoints at $v=v_{2}>0$ and $v_{3}<0$. This proves the general case $k\in\mathbf{N}$.

\end{proof}

\begin{figure}[ht]
\centering
\subfloat[Points above Roof I will leave $R_{II}$ through the top surface $\{w=0\}$.]{{\includegraphics[width=0.5\textwidth]{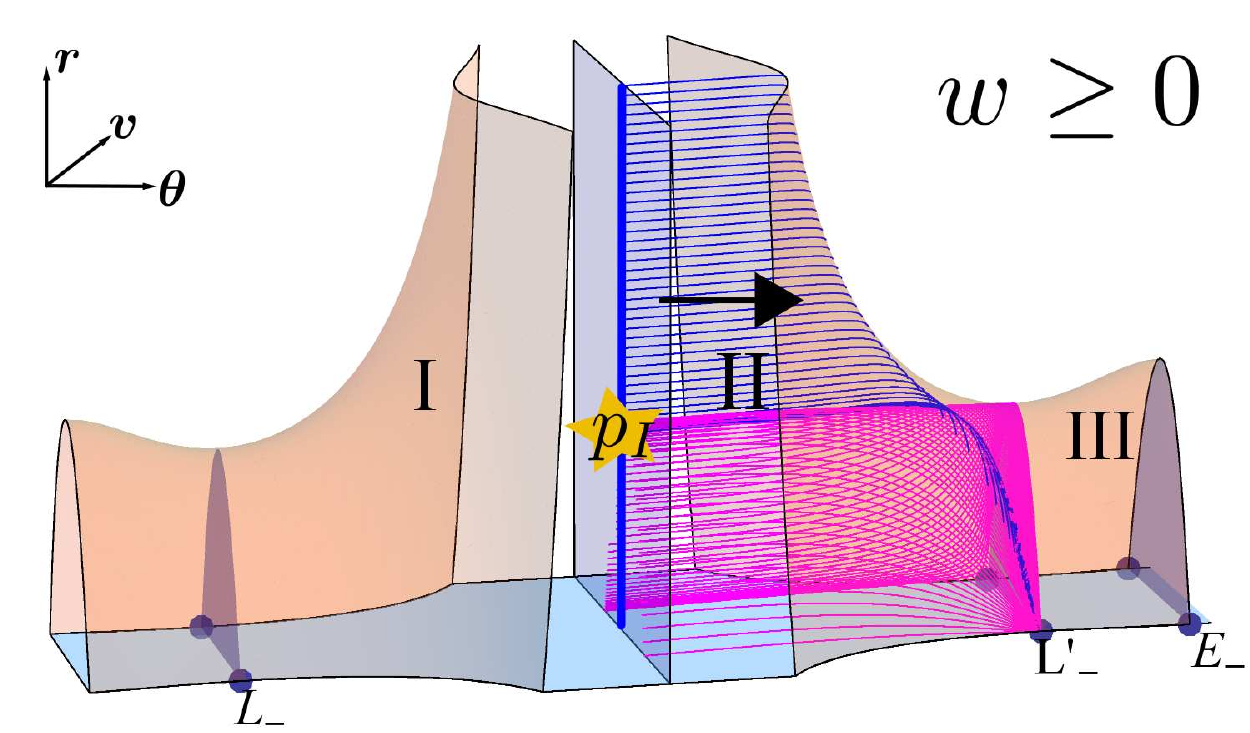}}}\hfill
\subfloat[Points below Roof II will cross $Q_{I}$ and then furthermore reach $Euler(-\pi,-)$.]{{\includegraphics[width=0.5\textwidth]{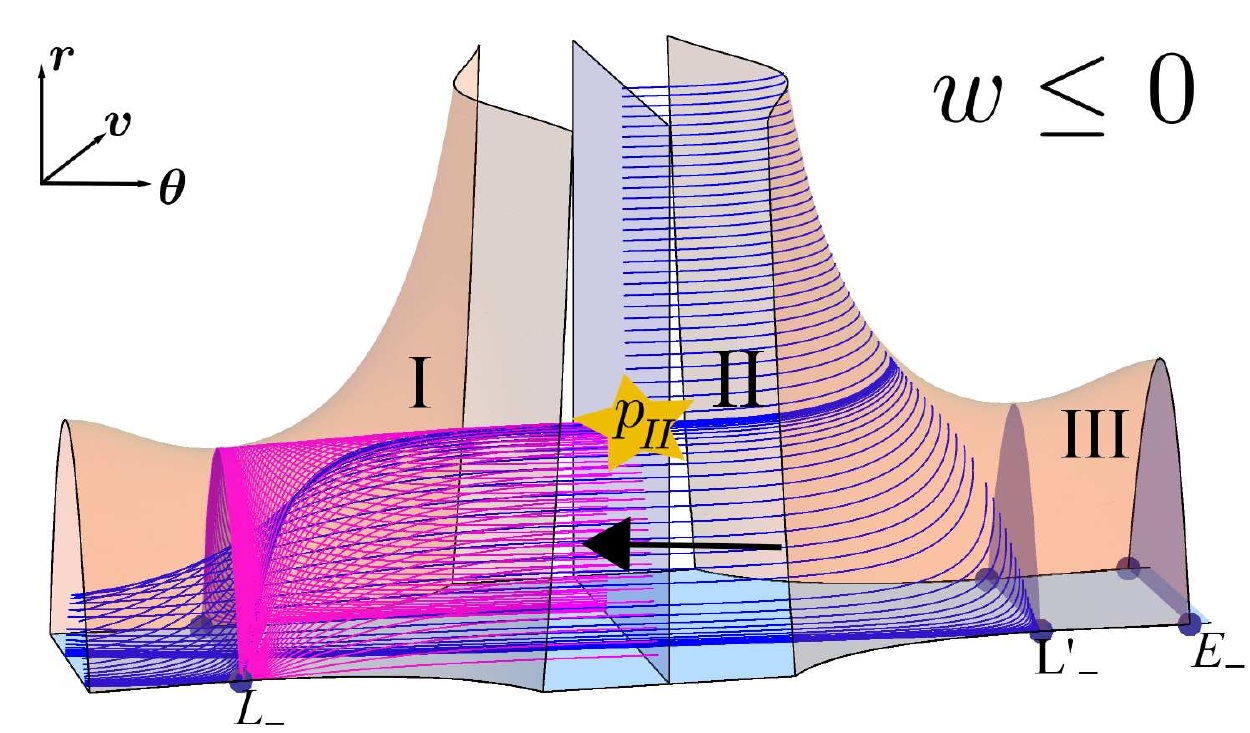}}}
\caption{This figure illustrates the proof of Theorem~\ref{thm_schubart3}}
\label{beta_family_proof}
\end{figure}

\begin{figure}[ht]
\centering
\subfloat[]{{\includegraphics[width=0.5\textwidth]{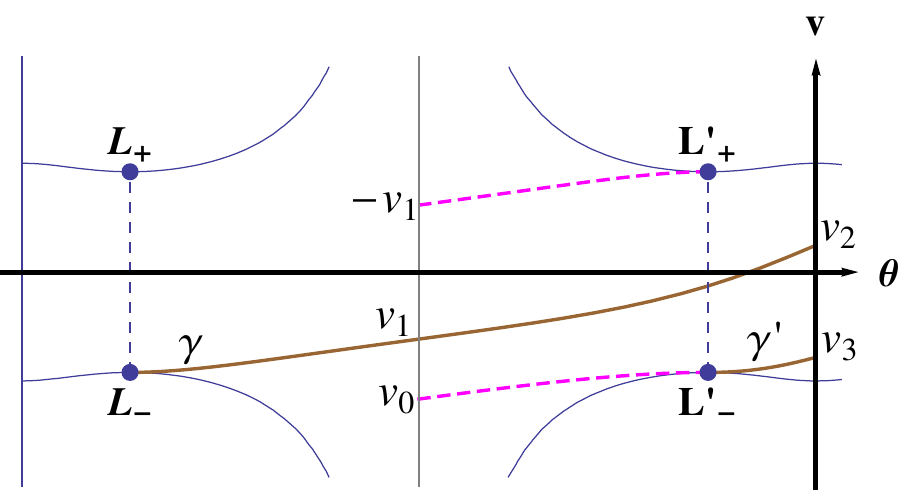}}}\hfill
\subfloat[]{{\includegraphics[width=0.5\textwidth]{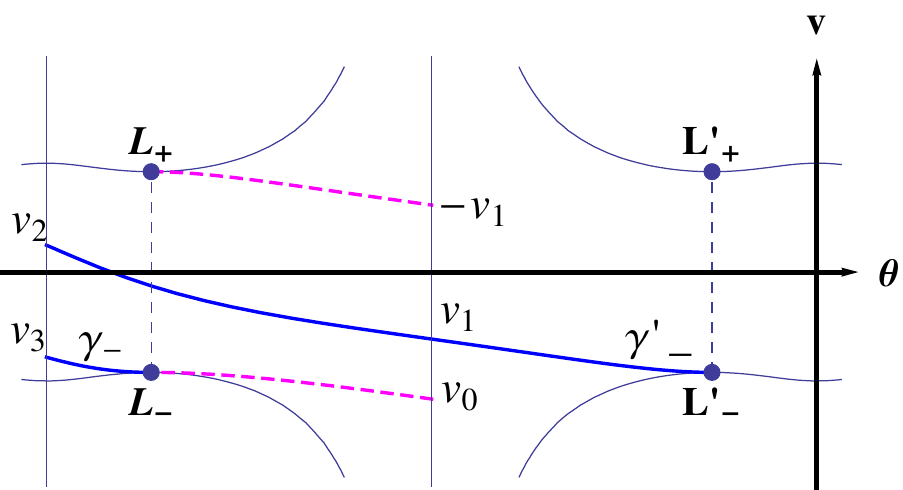}}}
\caption{The branches $\gamma$,$\gamma'$ lie in $w\geq 0$, while the branches $\gamma_{-}$,$\gamma'_{-}$ lie in $w\leq 0$.}
\label{four_branches}
\end{figure}

\begin{theorem}{\textbf{(Less-symmetric $\mathcal{B}-$family Schubart-like orbits).}}\label{thm_schubart3}
Assume $v_{1}<0, v_{2}>0, v_{3}<0$, and $v_{2}\neq -v_{3}$. For every pair of positive integers $(i,j)$, there exists a Schubart-like orbit with the following property: In a half period, the orbit starts with $i+1$ partial-collisions, then crosses the Eulerian line, followed by $j+1$ continuous partial-collisions. See table~\ref{orbits_table}(d).
\end{theorem}

\begin{proof}
The proof is quite similar to the proof of theorem 6.1 of~\cite{chen2013}.

Here we only prove the case when $i$ is even and $j$ is odd; the other cases can be proved similarly. In the proof of theorem~\ref{thm_schubart2}, for even $i$, we construct a subset of $S(-\pi/2)$, say $S_{i}$, that can be followed to cross $Partial(-\pi/2,\pm)$ $i$ times to reach $Euler(0,+)$, and form a curve, say $\Gamma_{i}$, having endpoints at $v=v_{2},v_{3}$ in the collision manifold. For odd $j$, we construct a subset of $S(-\pi/2)$, say $S_{j}$, that can be followed to cross $Partial(-\pi/2,\pm)$ $j$ times to reach $Euler(-\pi,-)$, and form a curve, said $\Gamma_{j}$, having endpoints at $v=v_{2},v_{3}$ in the collision manifold. Recall that the system has symmetries $R_{1},R_{2}$, and $T_{1}$ as defined in~(\ref{symmetries}). By symmetries and reversibility, $R_{1}T_{1}\Gamma_{j}$ can be followed to cross $Partial(\pi/2,\pm)$ $j-1$ times to form $R_{1}T_{1}S_{i}$. The curve $R_{1}T_{1}\Gamma_{j}$ lies on $Euler(0,+)$, with its endpoints at $(r,v)=(0,-v_{2})$ and $(0,-v_{3})$. Since $v_{2}>0, v_{3}<0$, and $v_{2}\neq -v_{3}$, the two curves $\Gamma_{i}$ and $R_{1}T_{1}\Gamma_{j}$ must intersect. An intersection point yields the desired orbit. 

\end{proof}

\begin{theorem}{\textbf{($\mathcal{ZB}-$family Schubart-like orbits).}}\label{thm_schubart4}
Assume $v_{1}<0, v_{2}>0, v_{3}<0$, and $v_{2}\neq -v_{3}$. For every pair of positive integers $(i,j)$, $i,j\geq 1$, there exists a periodic orbit with the following property: In a half period, the orbit starts at a brake time, then crosses the partial-collision line $\{\theta=-\pi/2\}$ i times, then crosses the Eulerian line, then crosses the partial-collision line $\{\theta=\pi/2\}$ $j$ times, and then hits the partial-collision line $\{\theta=\pi/2\}$ orthogonally. See table~\ref{orbits_table}(e).
\end{theorem}

\begin{proof}
Here we prove the case when both $i$ and $j$ are odd; the other cases can be proved similarly. The proof is almost the same as that of theorem~\ref{thm_schubart3}; only the curve $S_{i}$ is constructed differently. In the proof of type 1 periodic brake orbits, which is theorem 5.4 of ~\cite{chen2013}, we construct a part of the zero velocity curve $\mathcal{Z}$, say $S_{i}$, that can be followed to cross $Partial(-\pi/2,\pm)$ $i$ times to reach $Euler(0,+)$, and form a curve, say $\Gamma_{i}$, having endpoints at $v=v_{2},v_{3}$ in the collision manifold. Therefore, the curve $R_{1}T_{1}S_{j}$ constructed in~\ref{thm_schubart3}, which also lies on $Euler(0,+)$ with its endpoints at $v=-v_{2},-v_{3}$, must intersect $\Gamma_{i}$. An intersection point yields the desired orbit.
\end{proof}

We end this section with another family of periodic orbits in the equal-mass isosceles three-body problem. See table~\ref{orbits_table}(h)(i). These orbits are numerically founded by shooting from $S(-\pi/2)$ and targeting to hit $\theta=\pi/2$ or $\theta=-3\pi/2$ orthogonally but not  orthogonally hitting an Euler line. We are unable to prove their existence at this moment. The difficulty is not to prove the existence of orbits that hit two partial-collision lines orthogonally, but is to distinguish them from $\mathcal{B}-$family orbits.

\section{The Behaviours of $\gamma$ and $\gamma'$}
From the previous section, we can see that the two branches $\gamma$ and $\gamma'$ play a crucial role in the existence proofs of periodic orbits. Recall that we denote the $v$-coordinates of the intersections of $\gamma$ with $\theta=-\pi/2$, $\theta=0$ by $v_{1},v_{2}$ respectively, and that of $\gamma'$ with $\theta=0$ by $v_{3}$. See figure~\ref{four_branches}(a). Since regions $R_{I}$ and $R_{III}$ are flowing rightward, $v_{1}$ and $v_{3}$ are always well-defined. As for $v_{2}$, it is well-defined provided $v_{1}<0$; this can be seen by considering the stable branches of $L'$ and $L'_{-}$ which trap $\gamma$ in between, and hence $\gamma$ crosses region $II$.

While the new coordinates facilitate the proofs for the existence theorems of periodic orbits, it is, however, easier to estimate the branches $\gamma$ and $\gamma'$ in Devaney's coordinates. The branch $\gamma'$ in the new corresponds to $\gamma'$ in the Devaney's coordinates, and hence the identical notations. On the other hand, the translation of the branch $\gamma$ in the new coordinate, i.e., $T_{1}\gamma$, corresponds to $\gamma''$ in Devaney's coordinates. Therefore, studying $\gamma$,$\gamma'$ in the new coordinates is equivalent to studying $\gamma''$,$\gamma'$ in Devaney's coordinates.

We will provide sufficient conditions that guarantee $v_{1}<0, v_{3}<0$, along with a lower bound for $v_{1}$ that guarantees $v_{2}>0$.
In the case when $V(\phi)$ has exactly three critical points at $\phi=\phi_{L},\phi_{m},\phi_{R}$, Mart\'{i}nez's conditions are as follows~\cite{martinez2012,martinez2013}:
\begin{align*}
&\cos(\phi_{b}-\phi)\widehat{V}(\phi)-\sin(\phi_{b}-\phi)\widehat{V}'(\phi)>0, \ \ \phi\in [\phi_{R}, \phi_{b}]\tag{M1}\\
&3V(\phi_{R})-2V(\phi_{m})>0, \tag{M2} \\
&G(\phi):=\frac{1}{\phi_{R}-\phi_{m}}-\frac{\phi-\phi_{m}}{2}\sqrt{\frac{2(\phi_{R}-\phi)}{\phi_{R}-\phi_{m}}}+2\frac{V'(\phi)}{V(\phi_{m})}>0,\ \ \phi\in [\phi_{m},\phi_{b}] \tag{M3} \\
&\mbox{The orbit} \ \ \phi(t;P_{m})\ \ \mbox{runs up to} \ \ B_{b}^{+}(B_{a}^{+}) \ \ \mbox{for positive time},  \tag{M4}
\end{align*}
where $P_{m}=(r,v,\phi,w)=(0,0,\phi_{m},+)$ lies on the collision manifold, and $B_{a}^{+}$, $B_{b}^{+}$ represent two of the four holes of the collision manifold that have $v\geq 0$. Specifically, Mart\'{i}nez's existence proof of Schubart-like periodic orbits that have only one singularity in a half period requires conditions (M1,M2,M3), and the existence proof of Schubart-like periodic orbits that have many singularities requires an additional condition (M4). However, not every condition has been successfully verified. For the planar double-polygon problem, the condition (M3) fails, and the condition (M2) is not rigorously proved. (In ~\cite{martinez2012}, it is proved that (M2) is true for $n$ large enough and numerically verified only for $3\leq n\leq 50$.) In~\cite{martinez2013}, the condition (M4) is not proved but only supported by numerical evidence in all her three problems.
 
In contrast, in the same setting, our conditions are as follows:
\begin{align*}
&\phi_{a}=-\pi/2, \phi_{b}=\pi/2,W'(\phi)\leq 0, \ \ \phi\in [0, \frac{\pi}{2}), \tag{N1}\\
&V(\phi_{R})-(\sin^{2}\frac{\phi_{R}-\phi_{m}}{2})V(\phi_{m})>0, \tag{N2} \\
& \mbox{In addition to} \ \ (N1), W(\frac{\pi}{2})=\frac{S_{n}}{4}, \mbox{and} \ \  |\frac{W'(\phi)}{W(\phi)}|\leq\frac{4}{5} \ \ \mbox{for} \ \  \phi\in[\frac{\pi}{4},\frac{\pi}{2}),\tag{N3} \\
&\mbox{Replace} \ \ |\frac{W'(\phi)}{W(\phi)}|\leq\frac{4}{5}\ \ \mbox{in (N3) by the condition in Lemma 4.1(iv)}, \tag{N3'} \\
& v_{2}\neq -v_{3}, \ \ \mbox{or equivalently,}\\
& \ \ \gamma \ \ \mbox{is not a heteroclinic connection between} \ \ L_{-} \ \ \mbox{and} \ \ T_{1}(L_{+}), \tag{N4}
\end{align*}
where $S_{n}=\sum_{k=1}^{n}\csc\pi k/n$.
We remark that the conditions (M1) and (N1) are equivalent when $\phi_{a}=-\pi/2, \phi_{b}=\pi/2$; this is apparent from the proof for proposition 1 in~\cite{martinez2012}.
We also remark that, in our examples, the condition (N2) is looser than (M2), since we will show that $\phi_{R}-\phi_{m}\leq \pi/4$.

In the following lemmas, we will show that the condition (N1) implies $v_{1}<0$, that (N2) implies $v_{3}<0$, and that the condition (N3) or (N3') implies $v_{2}>0$. As a result, the conditions (N1,N2) ensure the existence of $\mathcal{B}-$family Schubart-like periodic orbits with $n=0$. The conditions (N1,N2,N3 or N3') ensure the existence of $\mathcal{B}-$family and $\mathcal{Z}1-$family periodic orbits. If moreover, $v_{2}\neq -v_{3}$, then there exist $\mathcal{ZB}-$family, less-symmetric $\mathcal{B}-$family, and $\mathcal{Z}5-$family periodic orbits.

\begin{lemma}\label{lemma_v1}
Assume $\phi_{a}=-\pi/2$, $\phi_{b}=\pi/2$, $\phi_{R}\in (0,\pi/4]$, $W(\pi/2)=S_{n}/4$. Then
\begin{itemize} 
\item[(i)] \textbf{(N1)} If $W'(\phi)\leq 0$ in $[0,\frac{\pi}{2})$, then $v_{1}<0$.
\item[(ii)]\textbf{(N3)} If moreover, $\alpha:=\frac{4}{5}\leq |\frac{W'(\phi)}{W(\phi)}|$ for $\phi\in [\frac{\pi}{4},\frac{\pi}{2})$, then $\frac{\beta}{2}\sqrt{S_{n}}<v_{1}<0$, where $\beta= -1.32$.
\item[(iii)] Furthermore, (ii) implies that $v_{2}>0$.
\item[(iv)]\textbf{(N3')} If the condition for (ii) does not hold, let $g_{3}(\phi)$ be the solution of~(\ref{secondlowerbound}). If $g_{3}(\pi/2)\geq \beta$, then $v_{2}>0$.
\end{itemize}
\end{lemma}

\begin{proof}
Following~\cite{martinez2012}, we introduce a new variable $g=\frac{v}{\sqrt{W(\phi)}}$. When restricted to the collision manifold, the differential equations become
\begin{equation}\label{equation_g}
\begin{aligned}
\dot{g}&=1-\frac{g^2}{2}\cos\phi-\frac{gw}{2}\frac{W'(\phi)}{W(\phi)}\\
\dot{\phi}&=w\\
\dot{w}&=-\frac{gw}{2}\cos\phi-\sin{\phi}(1-\cos\phi g^2)+\frac{W'(\phi)}{W(\phi)}(\cos\phi-\frac{w^2}{2}),
\end{aligned}
\end{equation}
and the collision manifold becomes
\begin{equation}\label{collision_g}
\frac{w^2}{2\cos\phi}-1=-\frac{1}{2}\cos\phi g^2.
\end{equation}
Note that the equation of the collision manifold~(\ref{collision_g}) is independent of $W(\phi)$, and the differential equation~(\ref{equation_g}) is not necessary gradient-like with respect to $g$. 

Using the equation for collision manifold, one finds that 
$\dot{g}=\frac{w^2}{2\cos\phi}-\frac{gw}{2}\frac{W'(\phi)}{W(\phi)}$, and hence
\begin{equation*}
\begin{aligned}
\frac{dg}{d\phi}&=\frac{w}{2\cos\phi}-\frac{g}{2}\frac{W'(\phi)}{W(\phi)} \\
                &=\pm\sqrt{\frac{1}{2\cos\phi}-\frac{g^2}{4}}-\frac{g}{2}\frac{W'(\phi)}{W(\phi)}, 
\end{aligned}
\end{equation*}                  
where we take the $+$ sign when $w\geq 0$ and the $-$ sign when $w<0$. 

Before reaching $\phi=\pi/2$, the unstable branch $\gamma''$ stays in $\{w\geq 0\}$, and it satisfies the differential equation
\begin{equation}\label{ODEforGamma}
\begin{aligned}
\frac{dg}{d\phi}&=\sqrt{\frac{1}{2\cos\phi}-\frac{g^2}{4}}-\frac{g}{2}\frac{W'(\phi)}{W(\phi)} \\
\lim_{\phi\rightarrow\phi_{R}}g(\phi)&=-\sqrt{2\sec\phi_{R}}.
\end{aligned}
\end{equation}

\begin{figure}[ht]
 \centering
\includegraphics[width=0.5\textwidth]{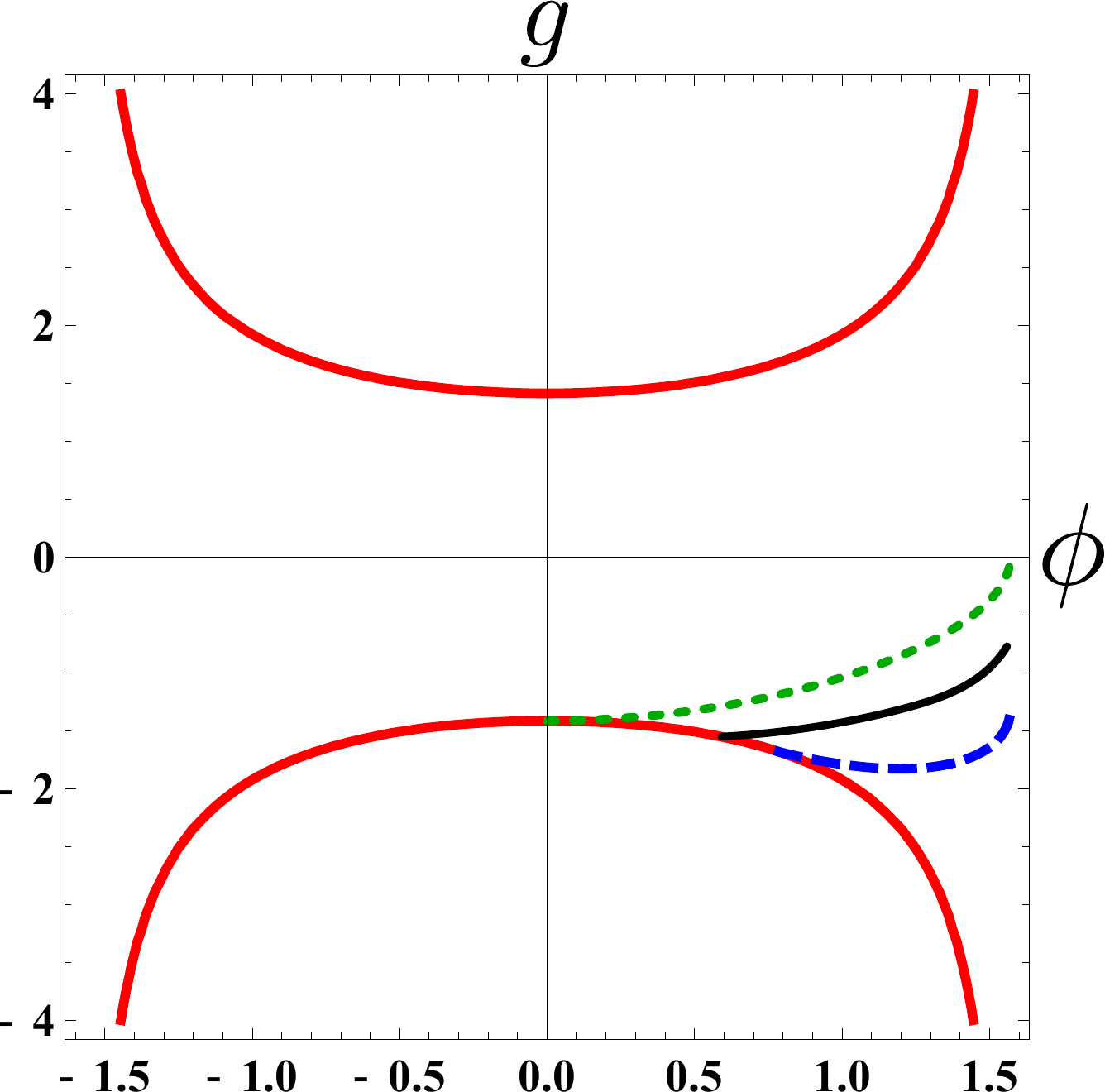}
\caption{The black solid curve is $\gamma''$. The green dotted curve and the blue dashed curve are the solutions to~\ref{solvableODE} and~\ref{nonautoODE} respectively.}
\label{compare_branches}
 \end{figure}

We study $\gamma''$ by comparing it with the solutions to the following two differential equations, see figure~\ref{compare_branches}.
\begin{equation}\label{solvableODE}
\begin{aligned}
\frac{dg_{1}}{d\phi}&=\sqrt{\frac{1}{2\cos\phi}-\frac{g_{1}^2}{4}} \\
g_{1}(0)&=-\sqrt{2}.
\end{aligned}
\end{equation} 

\begin{equation}\label{nonautoODE}
\begin{aligned}
\frac{dg_{2}}{d\phi}&=\sqrt{\frac{1}{2\cos\phi}-\frac{g_{2}^2}{4}}+\alpha\frac{g_{2}}{2}, \ \ \mbox{where} \ \ \alpha=\frac{4}{5} \\
g_{2}(\frac{\pi}{4})&=-\sqrt{2\sqrt{2}}.
\end{aligned}
\end{equation} 

For equation~(\ref{solvableODE}), one can find the solution explicitly; the solution is $g_{1}(\phi)=-\sqrt{2\cos\phi}$, and hence $g_{1}(\pi/2)=0$. The branch $\gamma''$ cannot cross the solution to~(\ref{solvableODE}) before it reaches $\phi=\pi/2$, since the slope $\frac{dg}{d\phi}$ of~(\ref{ODEforGamma}) is strictly less than that of ~(\ref{solvableODE}) when $g<0$ and $\frac{W'(\phi)}{W(\phi)}<0$. Therefore, $v_{1}<g_{1}(\pi/2)=0$. This proves (i).

Now we prove (ii). For equation~(\ref{nonautoODE}), after a numerical integration,  we find that $g_{2}(\pi/2):=\beta_{1}\approx -1.315705>\beta$. Therefore, provided that $|\frac{W'(\phi)}{W(\phi)}|\leq \alpha:=\frac{4}{5}$ for $\phi\in [\pi/4,\pi/2)$ and that $\frac{W'(\phi)}{W(\phi)}\leq 0$ for $\phi\in [0,\pi/2)$, the curve $\gamma''$ stays above the solution of ~(\ref{nonautoODE}) at least until it reaches $\phi=\pi/2$, since 
\begin{equation*}
\alpha\frac{g}{2}\leq -\frac{g}{2}\frac{W'(\phi)}{W(\phi)} \leq 0.
\end{equation*}
Therefore, the intersection of $\gamma''$ and $\phi=\pi/2$ has $\beta<g(\pi/2)<0$. Recovering the $v$ variable by $v=g\sqrt{W(\phi)}$, we have proved that $\frac{\beta}{2}\sqrt{S_{n}}<v_{1}<0$.

Next, we prove (iii). We follow $\gamma''$ from $\phi=\pi/2$ to $\phi=0$. Note that along this segment of $\gamma''$, we have $w\leq 0$. For the sake of contradiction, assume $v\leq 0$, which is equivalent to that $g\leq 0$. Then $\frac{dg}{d\phi}=-\sqrt{\frac{1}{2\cos\phi}-\frac{g^2}{4}}-\frac{g}{2}\frac{W'(\phi)}{W(\phi)}$ is negative, hence $0\geq g(\phi)=\beta_{1}>\beta=-1.32 $ along the part of $\gamma''$ where $\phi$ goes from $\phi=\pi/2$ to $\phi=0$.

We then have 
\begin{equation}\label{integral}
\begin{aligned}
g(0)&=g(\frac{\pi}{2})+\int_{\pi/2}^{0}\frac{dg}{d\phi}d\phi \\
    &=g(\frac{\pi}{2})+\int_{0}^{\frac{\pi}{2}}\sqrt{\frac{1}{2\cos\phi}-\frac{g^2}{4}}+\frac{g}{2}\frac{W'(\phi)}{W(\phi)}d\phi \\
    &\geq g(\frac{\pi}{2})+\int_{0}^{\frac{\pi}{2}}\sqrt{\frac{1}{2\cos\phi}-\frac{g^2}{4}}d\phi \\
    &\geq \beta+  \int_{0}^{\frac{\pi}{2}}\sqrt{\frac{1}{2\cos\phi}-\frac{\beta^2}{4}}d\phi \\
    &\approx -1.32+ 1.379875 >0.
\end{aligned}
\end{equation} 
This contradicts to our assumption that $g\leq 0$. So this implies that $\gamma''$ intersects $g=0$, i.e., $v=0$, before it reaches $\phi=0$. Therefore, $v_{2}>0$.

Finally, if the condition in (ii) does not hold, we compare $\gamma''$ with the solution of the initial value problem,
\begin{equation}\label{secondlowerbound}
\begin{aligned}
\frac{dg_{3}}{d\phi}&=\sqrt{\frac{1}{2\cos\phi}-\frac{g_{3}^2}{4}}-\frac{g_{3}}{2}\frac{W'(\phi)}{W(\phi)} \\
g_{3}(\pi/4)&=-\sqrt{2\sqrt{2}}.
\end{aligned}
\end{equation} 

In the interval when $\phi$ increases from $\pi/4$ to $\pi/2$, the solution of ~(\ref{secondlowerbound}) stays below $\gamma''$, since they satisfy the same differential equation and $g(\pi/4)\geq g_{3}(\pi/4)$. Therefore, if $g_{3}(\pi/2)\geq\beta$, then by using the same argument for (iii), one proves that $v_{2}>0$.

\end{proof}

\textbf{Remark on improving the estimation.} In the proof, the constant $\alpha=4/5$ is an upper bound for $|\frac{W'(\phi)}{W(\phi)}|$. One may choose a smaller upper bound, which will give a larger value of $\beta_{1}$, and one still obtains contradiction from equation (\ref{integral}) and therefore concludes $v_{2}>0$. On the other hand, if one choose a larger upper bound, for example $\alpha=1$, then one may not obtain contradiction from equation (\ref{integral}).

\begin{lemma}\label{lemma_v3}
\textbf{(N2)} If $V(\phi_{R})>(\sin^{2}\frac{\phi_{R}-\phi_{m}}{2})V(\phi_{m})$, then $v_{3}<0$. 
\end{lemma}

\begin{proof}

Along the branch $\gamma''$, from equation~(\ref{diffeqn_devaney}), we have 
\begin{equation}
\frac{dv}{d\phi}=\frac{1}{2}\sqrt{2V(\phi)-v^2}.
\end{equation}
If $v_{3}\geq 0$, this implies that $\gamma''$ reaches $v=0$ before it reaches $\phi=\phi_{m}$, which implies that along $\gamma''$, when the variable $v$ varies from $-v_{L}=-\sqrt{2V(\phi_{L})}=-v_{R}$ to $0$, the total variation of the variable $\phi$ is less then $\phi_{m}-\phi_{L}$, so
\begin{equation*}
\begin{aligned}
\phi_{R}-\phi_{m}=\phi_{m}-\phi_{L}&\geq\bigtriangleup\phi=\int_{-v_{L}}^{0}\frac{2}{\sqrt{2V(\phi)-v^2}}dv\\
&\geq \int_{-v_{L}}^{0}\frac{2}{\sqrt{2V(\phi_{m})-v^2}}dv\\
&=2\arcsin(\frac{v_{L}}{\sqrt{2V(\phi_{m})}})=2\arcsin(\sqrt{\frac{2V(\phi_{R})}{2V(\phi_{m})}}),
\end{aligned}
\end{equation*}
which makes a contradiction to the assumption $V(\phi_{R})-\sin^{2}\frac{\phi_{R}-\phi_{m}}{2}V(\phi_{m})>0$.

\end{proof}

\section{Three Applications}

\subsection{The $n$-pyramidal problem.}

The n-pyramidal problem consists of $n$ equal masses $m_{1}=m_{2}=\cdots=m_{n}=1$ along with an additional mass $m_{n+1}=\mu$. The $n$ equal masses always lie in some horizontal plane $z=z_{1}$ and equally spaced in a circle centered at the origin with radius $q_{1}$, forming a regular $n-$polygon, while $m_{n}$ moves up and down on the $z-$axis. We denote the signed distance between $m_{n+1}$ to the plane $z=z_{1}$ by $q_{2}$. See table~\ref{summary} for the configuration. Note that the planar isosceles problem is the special case of the $n-$pyramidal problem when $n=2$.

Following Mart\'{i}nez~\cite{martinez2012}, the Lagrangian is given by
\begin{equation}
L(q_{1},q_{2},\dot{q_{1}},\dot{q_{2}})= \sum_{k=1}^{n-1}\frac{1}{4q_{1}\sin l_{k}}+\frac{\mu}{\sqrt{q_{1}^2+q_{2}^2}}+\frac{1}{2}(\dot{q}_{1}^2+\frac{\mu}{n+\mu}\dot{q}_{2}^2),
\end{equation}
where $l_{k}=\pi k/n$, and $q_{1}\geq 0, q_{2}\in\mathbf{R}$. In Devaney's coordinates, the variables $r,\phi$ are defined by 
\begin{equation*}
r^{2}=q_{1}^2+\frac{\mu}{n+\mu}q_{2}^2, \ \ \ 
q_{1}=r\cos\phi,\ \  q_{2}=r\sqrt{\frac{n+\mu}{\mu}}\sin\phi,  \ \ r\geq 0, \ \ \ \phi\in (-\frac{\pi}{2},\frac{\pi}{2}),
\end{equation*}
and as a consequence,
\[V(\phi)=\frac{S_{n}}{4\cos\phi}+\frac{\mu}{\sqrt{1+(n/\mu)\sin^{2}\phi}}, \ \ 
\mbox{where} \ \  S_{n}=\sum_{k=1}^{n-1}\csc{l_{k}}, l_{k}=\pi k/n\]

Mart\'{i}nez has located the critical points of $V(\phi)$ in the following lemma.
\begin{lemma}~\cite{martinez2012}
\begin{enumerate}
\item If $2\leq n<473$, then $V(\phi)$ has three non-degenerate critical points: a maximum at $\phi=0$ and two minima at $\pm\phi_{R}$, where 
\[\tan^{2}\phi_{R}=\frac{\mu}{n+\mu}((\frac{4n}{S_{n}})^{2/3}-1).\]
\item If $n\geq 473$, then $V(\phi)$ has a unique non-degenerate critical point at $\phi=0$.
\end{enumerate}
\end{lemma}

In~\cite{martinez2012}, the conditions (M1,M2,M3) have been successfully verified, and therefore the existence of $\mathcal{B}-$family periodic orbits with $k=0$ is proved. We remark that the conditions (M1,M2) imply our conditions (N1,N2), which also ensure the existence of $\mathcal{B}-$family periodic orbits with $k=0$.

To prove the existence of other families of periodic orbits, we now are left to verify the condition (N3), which ensures $v_{2}>0$. However, in the case $n=2$, that is, in the isosceles problem, the behavior of $\gamma$ and $\gamma'$ with respect to the mass ratio has be carefully analyzed~\cite{simo1981,simo1987,chen2013}. It is shown that $v_{2}>0$ if and only if the mass ratio $\mu$ satisfies $0<\mu<\epsilon_{2}\approx 2.661993$. In other words, the condition (N3) does not hold for general $\mu$. \textbf{From now on, we restrict our study to the case $\mu=1$.}   We will verify the condition (N3) for $n\geq 4$.

We first show that $\phi_{R}\in (0,\pi/4]$. From the previous lemma, 
\begin{equation*}
\begin{aligned}
\tan^{2}\phi_{R}&=\frac{1}{n+1}((\frac{4n}{S_{n}})^{2/3}-1)\leq \frac{1}{n+1}((\frac{4n}{n-1})^{2/3}-1)\leq \frac{1}{n+1}(8^{2/3}-1)\leq 1.
\end{aligned}
\end{equation*}
Therefore, $\phi_{R}\in (0,\pi/4]$ for $n\geq 2$.

Second, we study the function $\frac{W'(\phi)}{W(\phi)}$. See figure~\ref{RatioPyramidal}(a) for its graph. We have $W(\phi)=\frac{S_{n}}{4}+\frac{\cos\phi}{\sqrt{1+n\sin^2\phi}}>0$. By the monotonicity of $\sin\theta$ and $\cos\theta$, it is apparently that $W'(\phi)\leq 0$ and $\frac{W'(\phi)}{W(\phi)}\leq 0$ in $[0,\pi/2)$.

Third, we show that for $n\geq 4$, $|\frac{W'(\phi)}{W(\phi)}|\leq\frac{4}{5}$ for $\phi\in [\pi/4,\pi/2)$. We compute $W'(\phi)= -\frac{(n+1)\sin\phi}{(1+n\sin^{2}\phi)^{3/2}}$ and $W''(\phi)= -\frac{(n+1)\cos\phi}{(1+n\sin^{2}\phi)^{5/2}}(1-2n\sin^{2}\phi)$, so in the interval $(0,\pi/2)$, the critical point of $W'(\phi)$ is at $\phi_{\ast}=\arcsin(\sqrt{\frac{1}{2n}})<\frac{\pi}{4}$ and $W'(\phi)\leq 0$ is increasing in $(\phi_{\ast},\pi/2)$. Therefore, $|W'(\phi)|\leq W'(\pi/4)$ in $[\pi/4,\pi/2)$. 

On the other hand, clearly $W(\phi)\geq\frac{S_{n}}{4}$. Therefore, 
\begin{equation*}
\begin{aligned}
|\frac{W'(\phi)}{W(\phi)}|&\leq |W'(\frac{\pi}{4})|\frac{4}{S_{n}}=\frac{2(n+1)}{(2+n)^{3/2}}\frac{4}{S_{n}}<\frac{4}{5} \ \ \ \ \forall\phi\in [\frac{\pi}{4},\frac{\pi}{2}), n\geq 4.
\end{aligned}
\end{equation*}
This verifies the condition (N3) for $n\geq 4$. For $n=2,3$, the condition (N3) does not hold, so we verify the condition (N3') in Lemma~\ref{lemma_v1}(iv) instead. Let $g_{3}(\phi)$ be the solution to~(\ref{secondlowerbound}). After a numerical integration with Mathematica, we found that $g_{3}(\pi/2)\approx$ $-1.2328676, -0.9930229$ $\geq \beta$ for $n=2,3$ respectively. This implies that $v_{2}>0$.

As a remark, the condition (N2) can be easily verified as follows. We have $V(0)=\frac{S_{n}}{4}+1<\frac{S_{n}}{4}+S_{n}=\frac{5}{4}S_{n}$ and 
$V(\phi_{R})\geq \frac{S_{n}}{4}$. Therefore, $V(\phi_{R})\geq V(0)/5\geq \sin^{2}(\pi/8)V(0)$.

\begin{figure}[ht]
\centering
\subfloat[The $n-$pyramidal problem]
{\includegraphics[width=0.5\textwidth]{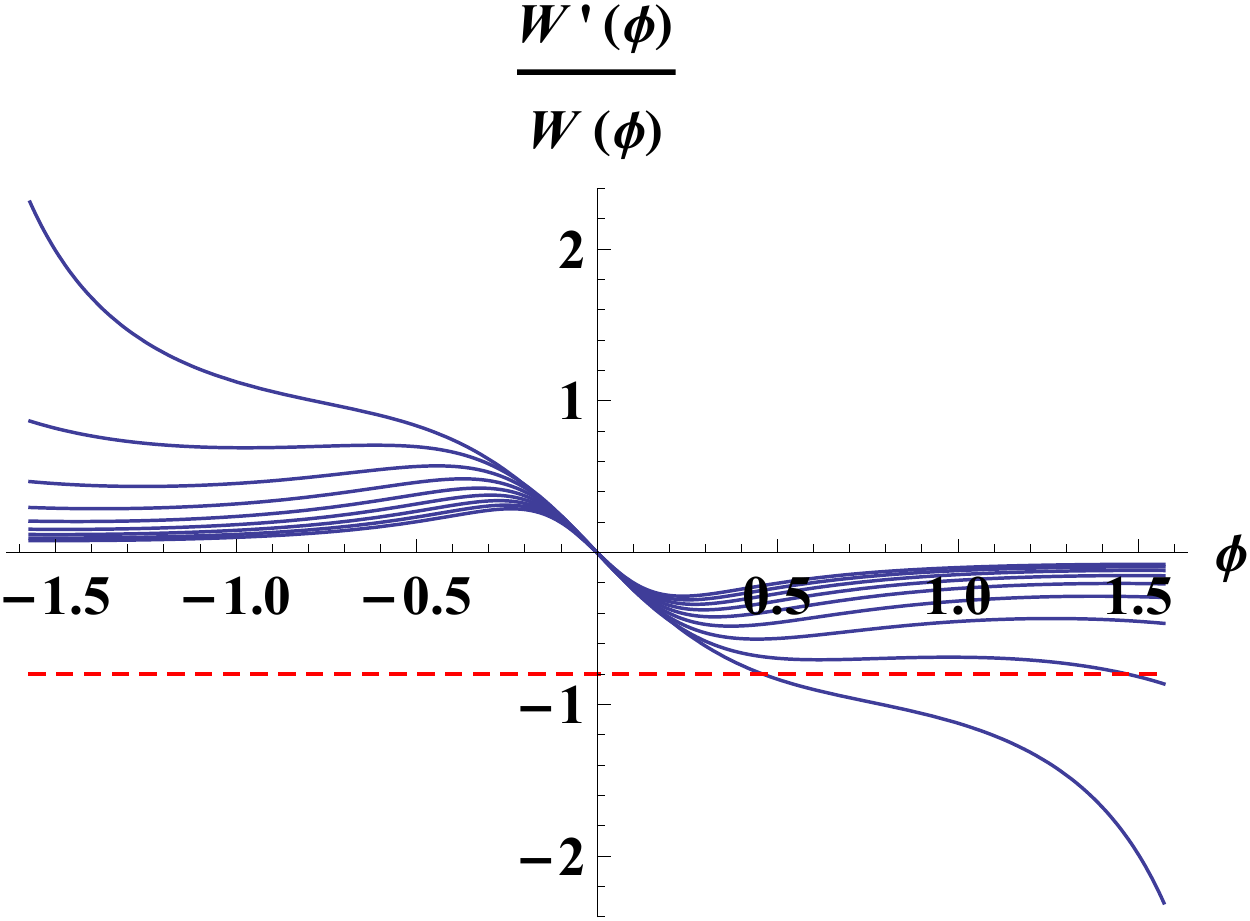}}\hfill
\subfloat[The spatial double-polygon problem]{\includegraphics[width=0.5\textwidth]{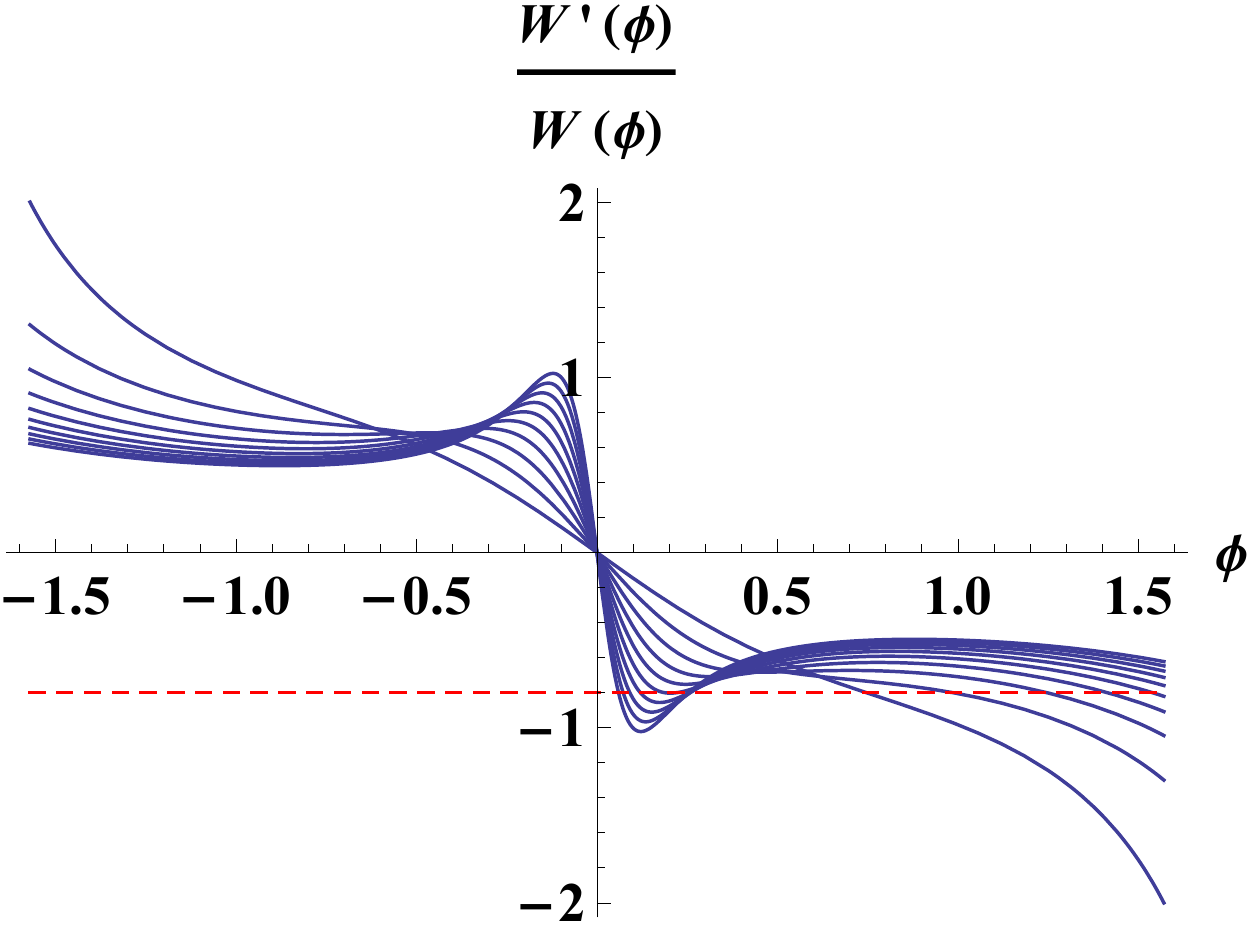}}
\caption{Graph of $\frac{W'(\phi)}{W(\phi)}$ for the equal-mass $n-$pyramidal problem and the spatial double-polygon problem with $2\leq n\leq 10$. As $n$ increases, $\frac{W'(\pi/2)}{W(\pi/2)}$ increases. The horizontal dash line is $\frac{W'(\phi)}{W(\phi)}=-\frac{4}{5}$}
\label{RatioPyramidal}
\end{figure}

Finally, we formally state the conclusion.
\begin{theorem}
In the planar isosceles three-body problem, let $m_{1}=m_{2}=1$. For any $m_{3}$ in an open interval including $m_{3}=1$, in addition to the six types of periodic brake orbits (including $\mathcal{Z}1-$famly and $\mathcal{Z}5-$family) proved in~\cite{chen2013}, there exist $\mathcal{B}-$family, less-symmetric $\mathcal{B}-$family, and $\mathcal{ZB}-$family periodic orbits.

For any $2\leq n<473$ with any positive mass $\mu$ in the $n-$pyramidal problem, there exists a  Schubart-like orbit in the $\mathcal{B}-$family with $n=0$.

For any $2\leq n<473$, in the equal-mass $n-$pyramidal problem, there exist $\mathcal{B}-$family and $\mathcal{Z}1-$family periodic orbits. If moreover, the hypothesis $v_{2}\neq -v_{3}$ is true, then there exist $\mathcal{Z}5-$family, less-symmetric $\mathcal{B}-$family, and $\mathcal{ZB}-$family periodic orbits.
\end{theorem}

\subsection{The spatial double-polygon problem.}

The spatial double-polygon problem consists of $2n$ equal masses, $n\geq 2$. The configurations form two twisted regular $n-$gons of the same size in two different non-fixed horizontal planes $z=\pm q_{2}$, centered on the $z-$axis. We denote the distance between a vertex to the $z-$axis by $q_{1}$. When projected to the $xy-$plane, the two $n-$gons are different by a rotation angle of $2\pi/n$. See figure~\ref{double_gons}.

\begin{figure}[ht]
\centering
\subfloat[The projection of the $2n$ bodies on the $xy$-plane.]{\includegraphics[width=0.3\textwidth]{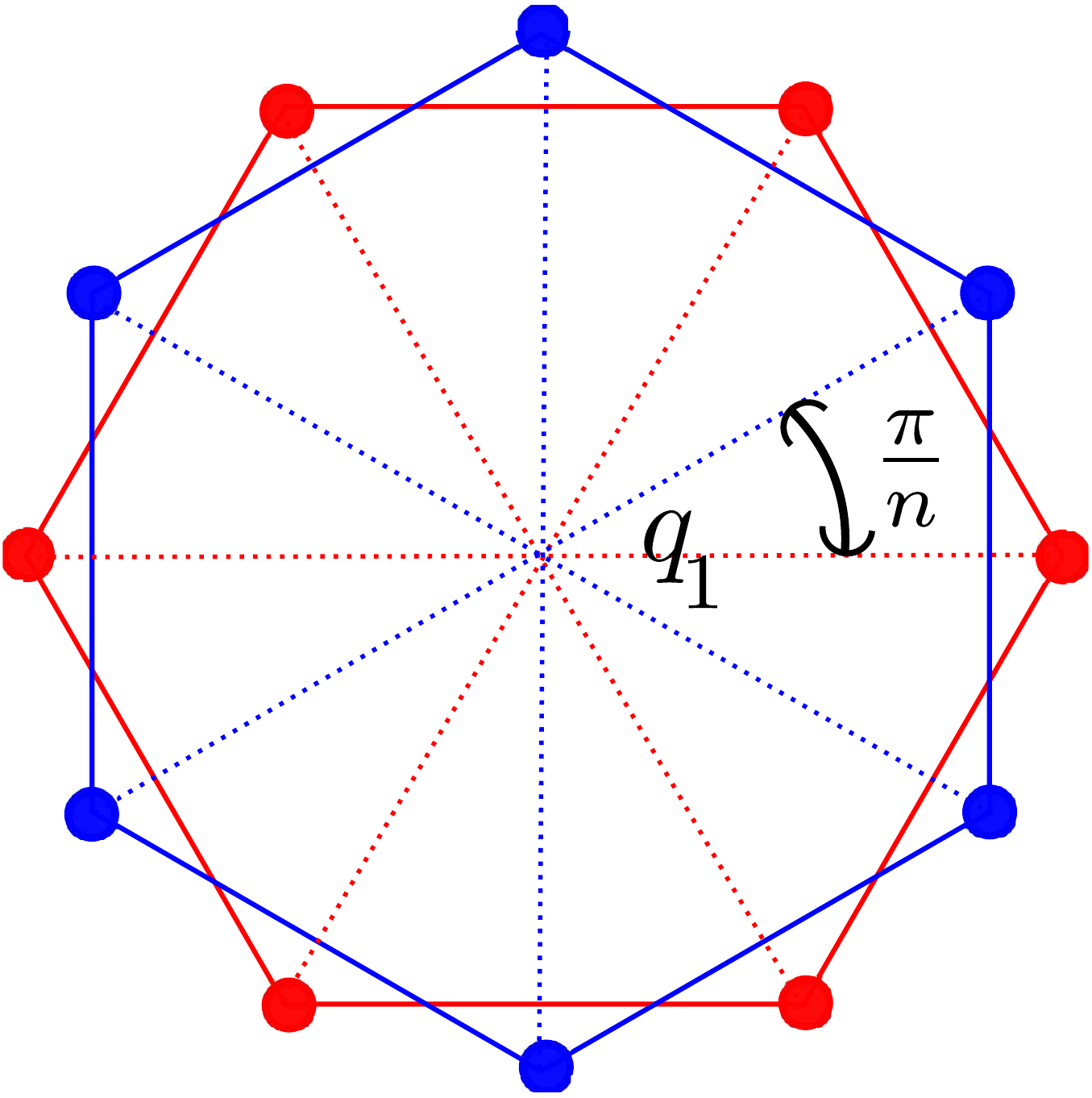}}\hfill
\subfloat[The configuration of the $2n$ bodies.]{\includegraphics[width=0.3\textwidth]{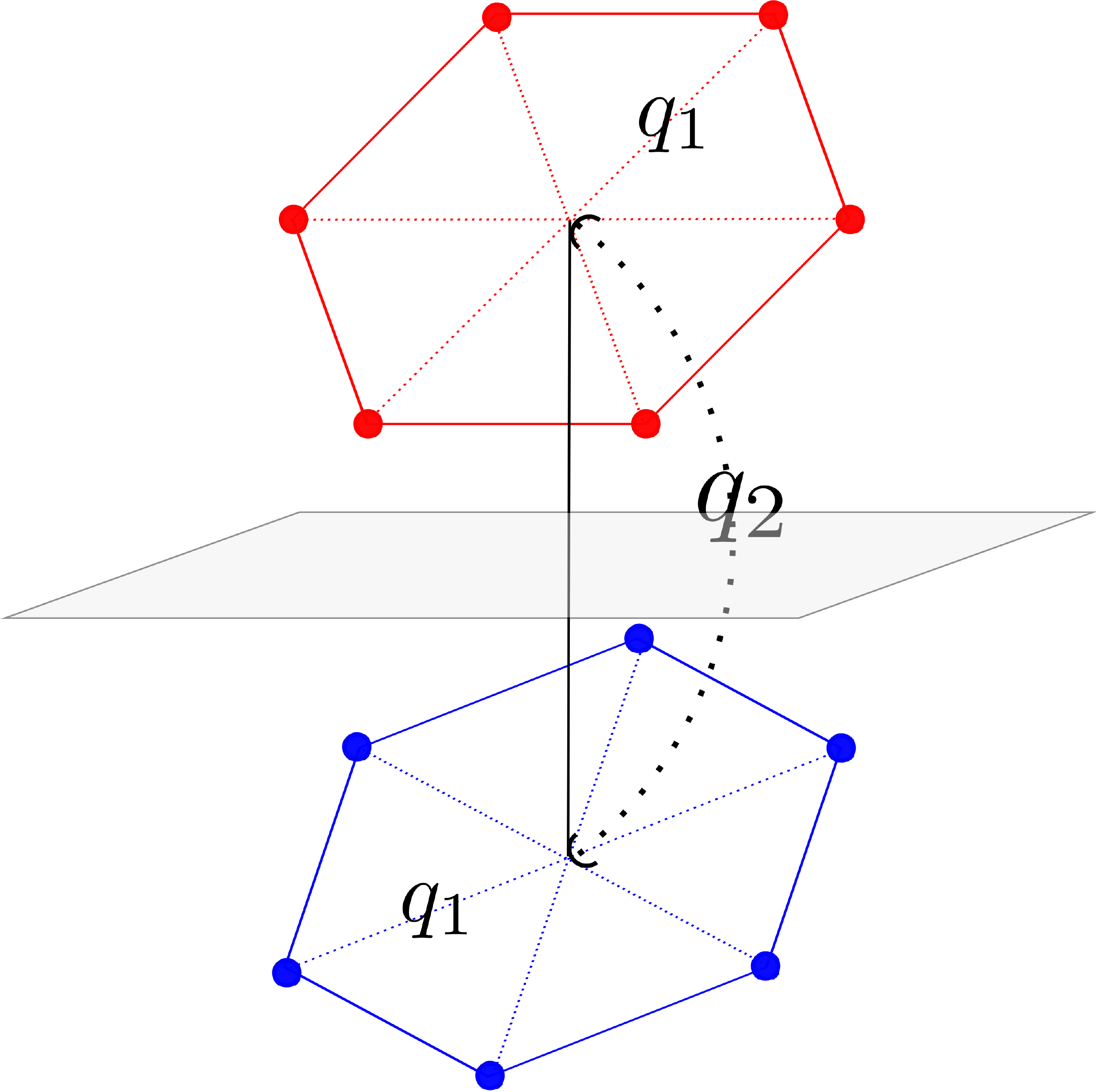}}
\caption{}
\label{double_gons}
\end{figure}

The Lagrangian is given by
\begin{equation*}
L(q_{1},q_{2},\dot{q_{1}},\dot{q_{2}})=\frac{1}{2}(\sum_{k=1}^{n-1}\frac{1}{2q_{1}\sin l_{k}}+\sum_{k=1}^{n}\frac{1}{\sqrt{(2q_{1}\sin\frac{l_{2k-1}}{2})^2+q_{2}^2}})+\frac{1}{2}(\dot{q_{1}}^2+\frac{\dot{q_{2}}^2}{4}),
\end{equation*}
where $l_{k}=\pi k/n$.

In Devaney's coordinates, write $(q_{1},q_{2})=r(\cos\phi,2\sin\phi)$.
Then
\begin{equation}
\begin{aligned}
V(\phi)&=\frac{1}{2}(\sum_{k=1}^{n-1}\frac{1}{2\cos\phi\sin l_{k}}+\sum_{k=1}^{n}\frac{1}{\sqrt{(2\cos\phi\sin\frac{l_{2k-1}}{2})^2+4\sin^{2}\phi}}) \\   
 &=\frac{S_{n}}{4}\frac{1}{\cos\phi}+\frac{1}{4}\sum_{k=1}^{n}\frac{1}{\sigma_{k}},   
\end{aligned}
\label{potentialV2}
\end{equation}
where $S_{n}=\sum_{k=1}^{n-1}\csc l_{k}$, $\sigma_{k}=(1-c_{k}^2\cos^2\phi)^{1/2}$, and $c_{k}=\cos\frac{\pi}{2n}(2k-1)$.

\begin{lemma}
$V(\phi)$ has exactly three critical points in $(-\pi/2,\pi/2)$, all of which are non-degenerate. They are at $\phi=-\phi_{R},0,\phi_{R}$, where $\phi_{R}\in (0,\pi/4)$.
\end{lemma}

\begin{proof}
\begin{equation}
\begin{aligned}
V'(\phi)&=\frac{S_{n}}{4}\sec\phi\tan\phi-\frac{1}{4}\sum_{k=1}^{n}\frac{c_{k}^2}{\sigma_{k}^3}\cos\phi\sin\phi \\
&=\frac{S_{n}}{4}\sin\phi\cos\phi(h(\phi)-g_{n}(\phi)),
\end{aligned}
\end{equation}
where $h(\phi)=\sec^{3}\phi$ and $g_{n}(\phi)=\frac{1}{S_{n}}\sum_{k=1}^{n}\frac{c_{k}^2}{\sigma_{k}^3}$.

Obviously $\phi=0$ is a critical point of $V(\phi)$. To show that $V(\phi)$ has exactly one critical point $\phi_{R}$ in $(0,\pi/2)$ and that $\phi_{R}\in (0,\pi/4)$, we will show that $g_{n}(\phi)$(resp. $h(\phi)$) is a decreasing (resp. increasing) function of $\phi\in (0,\pi/2)$, that $h(0)<g(0)$, and that $h(\pi/4)>g(\pi/4)$.

Since $\cos\phi$ is strictly decreasing in $(0,\pi/2)$, the increasing or decreasing properties of $h(\phi)$ and $g_{n}(\phi)$ are obvious. Now we show that $h(0)<g(0)$. Note that $h(0)=1$, and that to show  $g(0)>1$ is equivalent to show that $\sum_{k=1}^{n}\frac{c_{k}^2}{(1-c_{k}^2)^{3/2}}>S_{n}$. Actually, the left-hand side of this inequality is much greater than the right-hand side. It is straightforward to verify the case when $n=1,2$. For $n\geq 3$, by using the two relations: $\cot^{2}\frac{\pi}{2n}\geq n$ for $n\geq 3$ and $\sin\frac{\pi}{2n}<\sin\frac{\pi}{n}(k-1)$ for $k=2,\cdots, n$, a very rough estimate will prove this inequality as follows:
\begin{equation*}
\begin{aligned}
\sum_{k=1}^{n}\frac{c_{k}^2}{(1-c_{k}^2)^{3/2}}\geq& \frac{c_{1}^2}{(1-c_{1}^2)^{3/2}}=(\cot^{2}\frac{\pi}{2n})\frac{1}{\sin\frac{\pi}{2n}}\geq \frac{n}{\sin\frac{\pi}{2n}}>\sum_{k=1}^{n-1}\frac{1}{\sin\frac{\pi}{n}(k)}=S_{n}.\\
\end{aligned}
\end{equation*}

Now we show that $h(\pi/4)>g(\pi/4)$. We have $h(\pi/4)=2\sqrt{2}$. 
\begin{equation*}
g_{n}(\frac{\pi}{4})=\frac{1}{S_{n}}\sum_{k=1}^{n}\frac{c_{k}^2}{(1-\frac{1}{2}c_{k}^2)^{3/2}}\leq\frac{1}{S_{n}}\sum_{k=1}^{n}\frac{c_{k}^2}{(1/2)^{3/2}}=\frac{2\sqrt{2}}{S_{n}}\sum_{k=1}^{n}c_{k}^2=\frac{2\sqrt{2}}{S_{n}} \frac{n}{2}\leq 2\sqrt{2}, 
\end{equation*}
where we use the fact $S_{n}=\sum_{k=1}^{n-1}\frac{1}{\sin\frac{\pi}{n}(k)}\geq n-1$ at the last step.

Finally we show non-degeneracy. Since
\begin{equation*}
V''(\phi)= \frac{S_{n}}{4}(\cos 2\phi(h(\phi)-g_{n}(\phi))+\frac{1}{2}\sin 2\phi(h'(\phi)-g_{n}'(\phi))),
\end{equation*}
$V''(0)=\frac{S_{n}}{4}(h(0)-g(0))<0$, and
$V''(\phi_{\ast})=\frac{S_{n}}{8}\sin 2\phi_{\ast}(h'(\phi_{\ast})-g_{n}'(\phi_{\ast})))>0$.

\end{proof}

Now we study the function $\frac{W'(\phi)}{W(\phi)}$. See figure~\ref{RatioPyramidal}(b) for its graph.

 \begin{lemma} $ $
 
 \begin{enumerate}
 \item[(i)] \textbf{Condition (N1)} $\frac{W'(\phi)}{W(\phi)}\leq 0$ for $\phi\in [0,\pi/2)$.
 \item[(ii)] \textbf{Condition (N2)} Let $n\geq 10$. Then $W(\frac{\pi}{2})=\frac{S_{n}}{4}$ and $|\frac{W'(\phi)}{W(\phi)}|\leq\frac{4}{5}$ for $\phi\in [\pi/4,\pi/2)$.
\end{enumerate}
\end{lemma}

\begin{proof}
Since $W(\phi)=\frac{S_{n}}{4}+\frac{\cos\phi}{4}\sum_{k=1}^{n}\frac{1}{\sigma_{k}}$, we have 
$4W'(\phi)=-\sin\phi\sum_{k=1}^{n}\frac{1}{\sigma_{k}^3}$. Clearly, $W(0)=\pi/4$, $W'(\phi)\leq 0$ and $W(\phi)>0$ in $[0,\pi/2)$.

We then show that $4|W'(\phi)|\leq \delta n/\pi$ for $\phi\in [\pi/4,\pi/2)$, where $\delta=3.83$. We write
\begin{align*}
4|W'(\phi)|&=\sum_{k=1}^{n}\frac{\sin\phi}{(1-c_{k}^2\cos^{2}\phi)^{3/2}}, c_{k}=\cos\frac{\pi(2k-1)}{2n},\\
f_{\phi}(x)=f(x;\phi)&:=\frac{\sin\phi}{(1-\cos^{2}x\cos^{2}\phi)^{3/2}},
\end{align*}
then $f_{\phi}(x)$, as a function of $x$, is Riemann integrable over the interval $[0,\pi/2]$. Since $f_{\phi}(x)$ is decreasing in $[0,\pi/2]$ and increasing in $[\pi/2,\pi]$, we have
\begin{equation*}
\frac{\pi}{n}\sum_{k=1}^{n}f_{\phi}(\frac{2k-1}{2n}\pi)+\frac{\pi}{n}f_{\phi}(\frac{\pi}{2})\leq \int_{0}^{\pi}f_{\phi}(x) dx := h(\phi),
\end{equation*}
where the expression on the left-hand side equals the lower Riemann sum of the integral.

To find the maximum of $h(\phi)$, we write
\begin{align*}
 f(x;\phi)&=-\frac{d}{d\phi}\frac{\cos\phi}{\sqrt{1-\cos^{2}x\cos^{2}\phi}}. \\
 \frac{1}{2}h(\phi)&=\int_{0}^{\frac{\pi}{2}}f(x;\phi) dx = -\frac{d}{d\phi}\int_{0}^{\frac{\pi}{2}}\frac{\cos\phi}{\sqrt{1-\cos^{2}x\cos^{2}\phi}}dx \\
 &= -\frac{d}{d\phi} \cos\phi K(\cos^{2}\phi) = \frac{1}{\sin\phi} E(\cos^{2}\phi) \\
 &= E(-\cot^{2}\phi),
\end{align*}
where the elliptic integrals are defined by
\[K(m)=\int_{0}^{\frac{\pi}{2}}\frac{1}{\sqrt{1-m\cos^{2}\theta}}d\theta, \ \  E(m)=\int_{0}^{\frac{\pi}{2}}\sqrt{1-m\sin^{2}\theta}d\theta,\]
and we use the fact that \[\frac{dK(m^2)}{dm}=\frac{E(m^2)}{m(1-m)}-\frac{K(m^2)}{m}.\]
Therefore, the maximum of $h(\phi)$ in $[\pi/4,\pi/2)$ is at $h(\pi/4)=2E(-1)\approx 3.82<\delta$.

This implies that  $4|W'(\phi)|\leq \delta n/\pi$. On the other hand, $4W(\phi)\geq S_{n}$. Therefore, 
\begin{equation}\label{estimate}
|\frac{W'(\phi)}{W(\phi)}|\leq \frac{\delta}{\pi} \frac{n}{S_{n}}<\frac{4}{5}, 
\end{equation}
provided $\frac{S_{n}}{n}>\frac{5\delta}{4\pi}> 1.524$. The numerical estimate of the sequence $S_{n}/n$ is included in Appendix.

\end{proof}

We still need to prove that $v_{2}>0$ in the case $2\leq n\leq 9$, when the condition (N3) is either untrue or not verified. First, the case $n=2$ is the so called tetrahedral 4-body problem, in which the flow on the collision manifold has been studied in ~\cite{vidal1999}, where its theorem 1 implies that $v_{2}>0$. As for the cases left, we have verified the condition (N3') stated in lemma~\ref{lemma_v1}(iv). We compute the value of $g_{3}(\pi/2)$, where $g_{3}$ is the solution of~(\ref{secondlowerbound}). The result is summarized in the following table. One sees that $g_{3}(\pi/2)\geq \beta_{1}=-1.32$ for $3\leq n\leq 9$. This implies that $v_{2}>0$.
\begin{table}[h!]
\centering
\resizebox{\columnwidth}{!}{%
\begin{tabular}{|c|c|c|c|c|c|c|c|c|c|}
\hline 
 $n$ & 2 & 3 & 4 & 5 & 6 & 7 & 8 & 9 \\ 
\hline 
 $g_{3}(\pi/2)$ & -1.41124 & -1.28340 & -1.21070 & -1.16294 & -1.12866 & -1.10259 & -1.08191 & -1.06499 \\
\hline 
\end{tabular} }
\end{table}

\begin{lemma}
\textbf{(N3)} Let $n\geq 2$. Then $V(\phi_{R})\geq \sin^{2}\frac{\phi_{R}-\phi_{m}}{2} V(\phi_{m})$.
\end{lemma}

\begin{proof}
Since $\phi_{R}\in (0,\pi/4)$, $\phi_{m}=0$ and $\sin^{2}\frac{\phi_{R}-\phi_{m}}{2}\sin^{2}\leq\frac{\pi}{8}\leq\frac{1}{4}$, it is sufficient to show that $4V(\phi_{R})\geq V(0)$.

Clearly, from~(\ref{potentialV2}), $4V(\phi_{R})\geq S_{n}$. 
When $n=2$, $4V(\phi_{R})\geq S_{n}=1$ and $V(0)=\frac{1+2\sqrt{2}}{4}<1$, so $4V(\phi_{R})\geq V(0)$. We then consider the case when $n\geq 3$.

Write 
\begin{equation*}
\begin{aligned}
4V(0)=&S_{n}+\sum_{k=1}^{n}\frac{1}{\sin\frac{\pi}{2n}(2k-1)}=\sum_{k=1}^{n-1} s_{k}+ \sum_{k=1}^{n} a_{k},
\end{aligned}
\end{equation*}
where $s_{k}=\csc\frac{\pi}{n}k\geq 0$ and $a_{k}=\csc\frac{\pi}{2n}(2k-1)\geq 0$.

It is sufficient to show that $a_{k}\leq 2 s_{k}$ for $k=1,\cdots, n-2$ and that $a_{n-1}+a_{n}< 3 s_{n-1}$. Since if these two conditions are true, then $\sum_{k=1}^{n} a_{k}< 3 \sum_{k=1}^{n-1} s_{k}= 3 S_{n}$, and therefore $4V(0)< 4 S_{n}$. This implies that $4V(\phi_{R})\geq V(0)$. 

Now we show that $a_{k}\leq 2 s_{k}$ for $k=1,\cdots, n-2$. 
\begin{equation*}
\begin{aligned}
0\leq\frac{a_{k}}{s_{k}}=&\frac{\sin\frac{2k}{2n}\pi}{\sin\frac{2k-1}{2n}\pi}=\frac{\sin\frac{2k-1}{2n}\pi\cos\frac{\pi}{2n}+\cos\frac{2k-1}{2n}\pi\sin\frac{\pi}{2n}}{\sin\frac{2k-1}{2n}\pi} \\
&=\cos\frac{\pi}{2n}+\cos\frac{(2k-1)\pi}{2n}\frac{\sin\frac{\pi}{2n}}{\sin\frac{2k-1}{2n}\pi}\leq 2.
\end{aligned}
\end{equation*}

Finally, we show that $a_{n-1}+a_{n}<3s_{n-1}$. Since $\frac{a_{n}}{s_{n-1}}=2\cos\frac{\pi}{2n}\leq 2$ and $\frac{a_{n-1}}{s_{n-1}}=\sin\frac{2\pi}{2n}\csc\frac{3\pi}{2n}<1$ when $n\geq 3$,
therefore, $a_{n-1}+a_{n}<3s_{n-1}$.

\end{proof}

Finally, we formally state the conclusion.
\begin{theorem}
In the equal-mass spatial double-polygon problem, when $n\geq 2$, there 
exist $\mathcal{B}-$family and $\mathcal{Z}1-$family periodic orbits. If moreover, the hypothesis $v_{2}\neq -v_{3}$ is true, then there exist $\mathcal{Z}5-$family, less-symmetric $\mathcal{B}-$family, and $\mathcal{ZB}-$family periodic orbits as well.
\end{theorem}

\subsection{The planar double-polygon problem}

The configurations of the planar double-polygon problem consist of two regular $n-$gons centered both at the origin and different by a rotation of angel $\frac{2\pi}{2n}$. We denote the distance between any vertex on the two polygons to the origin by $q_{1}$ and $q_{2}$ respectively. See table~\ref{summary}.

The Lagrangian of this system is 
\begin{equation}
L(q_{1},q_{2},\dot{q_{1}},\dot{q_{2}})=\frac{S_{n}}{4}(\frac{1}{q_{1}}+\frac{1}{q_{2}})+  \sum_{k=1}^{n}\frac{1}{r_{k}}+\frac{1}{2}(\dot{q_{1}}^2+\dot{q_{2}}^2),
\end{equation}
where $r_{k}^2=q_{1}^2+q_{2}^2-2q_{1}q_{2}\cos l_{2k-1}$.

Note that $q_{1}$ and $q_{2}$ are non-negative. So we define the size variable by $r^2=q_{1}^2+q_{2}^2$ and the shape variable $\theta\in (-\infty,\infty)$ by 
\begin{equation*}
q_{1}=r\frac{\sqrt{\cos^2\theta+1}-\sin\theta}{2}, \ \ q_{2}=r\frac{\sqrt{\cos^2\theta+1}+\sin\theta}{2}.
\end{equation*}

Then 
\begin{equation*}
L(r,\dot{r},\theta,\dot{\theta})=\frac{1}{2}\dot{r}^2+\frac{1}{2}r^2\dot{\theta}^2\frac{\cos^2\theta}{1+\cos^2\theta}+\frac{1}{r}V(\theta),
\end{equation*}
where $V(\theta)=\frac{S_{n}}{2}\frac{\sqrt{\cos^2\theta+1}}{\cos^2\theta}+\sum_{k=1}^{n}\frac{1}{\sqrt{1-\cos^2\theta\cos l_{2k}}}.$

In Devaney's coordinates, clearly $\phi_{a}=0, \phi_{b}=\pi/2$, and Mart\'{i}nez~\cite{martinez2012} has shown that $V(\phi)$ has a unique critical point at $\phi=0$ if $n=2$. Moreover, if $n\geq 3$, then $V(\phi)$ has three non-degenerate critical points $\phi_{L}<\phi_{m}<\phi_{R}$, $\phi_{m}=\pi/4$, and $\phi_{R}\in (\pi/4,\arctan(2))$.

To prove the existence of Shubart-like orbits, Mart\'{i}nez has successfully verified the conditions (M1,M2); however, the condition (M3) fails. Nonetheless, the conditions (M1,M2) imply our conditions (N1,N2), which are sufficient to guarantee the existence of Schubart-like periodic orbits.

To furthermore prove the existence of other periodic orbits, we need to show that $v_{2}>0$. However, for $3\leq n\leq 20$, we numerically study the two important branches $\gamma$, $\gamma'$ as shown in figure~\ref{double_planar}, and find that $v_{2}<0$ and $v_{3}<0$. So our shooting arguments do not work here. Nonetheless, the figure suggests that $\gamma$ and $\gamma'$ can reach $\theta=\pi/2$ and that their intersections with $\theta=\pi/2$ have $v_{4}>0$ and $v_{5}<0$ respectively. This suggests the existence of another type of periodic orbits, namely type 2 periodic brake orbits, by Theorem 5.5 of~\cite{chen2013}. At this point we are unable to prove that $v_{4}>0, v_{5}<0$ rigorously; we leave it for further investigation.

\begin{figure}[ht]
\centering
\subfloat{{\includegraphics[width=0.5\textwidth]{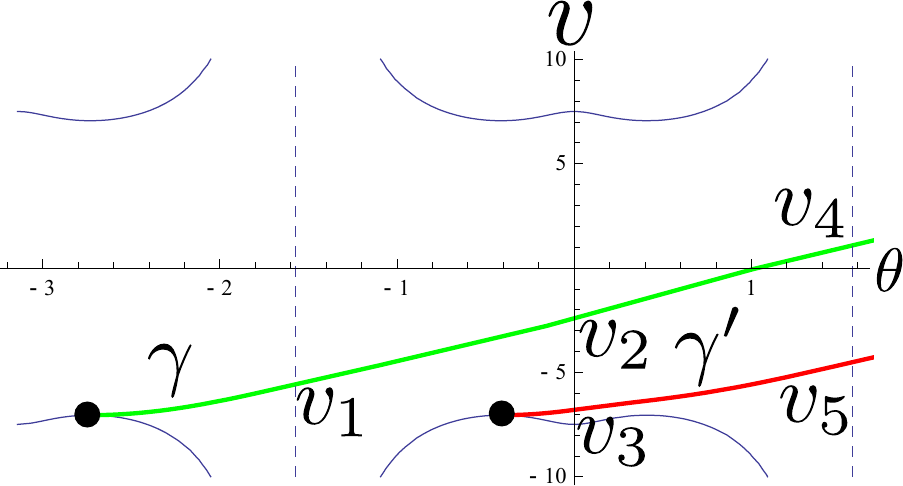}}}
\caption{The branches $\gamma,\gamma'$ for the planar double-polygon problem with $n=10$.}
\label{double_planar}
\end{figure} 
 
 Finally, we formally state the conclusion.
\begin{theorem}
In the equal-mass planar double-polygon problem, when $n\geq 3$, there exists a Shubart-like orbit in $\mathcal{B}-$family with $n=0$.
\end{theorem}

\section{Appendix}\label{appendix}
The appendix includes the properties of the series $S_{n}/n$ and the derivation of the differential equation~(\ref{newequation}). The derivation is only presented here for the referees' convenience and will not appear in the published version.

\subsection{The series $S_{n}/n$}

We recall that $S_{n}=\sum_{k=1}^{n-1}\csc\frac{\pi}{n}k$.
The series $S_{n}$ has been carefully analyzed in~\cite{moeckel1995}, where the authors provided an asymptotic expansion of $S_{n}/4$ for $n$ large:
\begin{equation}\label{asymptotic}
\frac{S_{n}}{4}\approx \frac{n}{2\pi}(\gamma+\log\frac{2\pi}{n})-\frac{\pi}{144n}+\frac{7\pi^{3}}{86400n^{3}}-\frac{31\pi^{5}}{7620480n^{5}}:=\tilde{A_{n}},
\end{equation}
where $\gamma\approx 0.5772156649$ is the Euler-Mascheroni constant, and the approximation has a relative error less than $10^{-6}$ for $n\geq 47$. Letting $F(n):=\frac{\tilde{A_{n}}}{n}$, then $F(n)$ is an increasing function. Mart\'{i}nez~\cite{martinez2012} has numerically computed some values of $\frac{S_{n}}{4n}$ and $F(n)$, which provide a strong evidence that $S_{n}/n$ should be an increasing sequence.

In this paper, we have assumed the fact that $\frac{S_{n}}{n}>1.524$ for $n\geq 10$ in~(\ref{estimate}).

\subsection{Derivation of the differential equation~(\ref{newequation})}

We recall that $v=r^{\frac{1}{2}}\dot{r}$, $w=\dot{\theta}r^{3/2}\frac{\cos^2\theta}{c(\theta)}$, $\frac{dt}{ds}=r^{\frac{3}{2}}\cos^{2}\theta$, $\mathcal{W}(\theta):=\cos^2\theta V(\theta)$,
\begin{equation*}
L(r,\dot{r},\theta,\dot{\theta})=\frac{1}{2}{\dot{r}^2}+\frac{1}{2}r^2\dot{\theta}^2\frac{\cos^{2}\theta}{c(\theta)}+\frac{1}{r} V(\theta),
\end{equation*} 
and the energy equation is
\[\frac{1}{2}v^2\cos^{2}\theta+\frac{1}{2}w^2 c(\theta)-\mathcal{W}(\theta)=-r\cos^2\theta.\]

First, $\frac{dr}{ds}=\frac{dr}{dt}\frac{dt}{ds}=r^{-\frac{1}{2}}v(r^{\frac{3}{2}}\cos^{2}\theta)=rv\cos^{2}\theta$.

Second, from the Euler-Lagrange equation $\frac{d}{dt}\frac{\delta L}{\delta \dot{r}}=\frac{\delta L}{\delta r}$, we have $\ddot{r}=r\dot{\theta}^2\frac{\cos^{2}\theta}{c(\theta)}-\frac{V(\theta)}{r^2}$. Therefore, 
\begin{align*}
\frac{dv}{ds}=&\dot{v}r^{\frac{3}{2}}\cos^{2}\theta=(\frac{1}{2}r^{-\frac{1}{2}}\dot{r}^2+r^{\frac{1}{2}}\ddot{r})r^{\frac{3}{2}}\cos^{2}\theta \\
    =&\frac{1}{2}v^2\cos^{2}\theta+w^2 c(\theta)-\mathcal{W}(\theta).
\end{align*}

Third, $\frac{d\theta}{ds}=wc(\theta)$ is straightforward from the definition of $w$.

Finally, from the Euler-Lagrange equation $\frac{d}{dt}\frac{\delta L}{\delta \dot{\theta}}=\frac{\delta L}{\delta\theta}$, we have $\frac{d}{dt}(r^{\frac{1}{2}}w)=\frac{\delta L}{\delta\theta}$, and hence
\[r^{\frac{1}{2}}\dot{w}=-\frac{1}{2}r^{-\frac{1}{2}}\dot{r}w+\frac{1}{2}r^{2}\dot{\theta}^2(\frac{d}{d\theta}\frac{\cos^{2}\theta}{c(\theta)})+\frac{1}{r}V'(\theta).\]
Therefore,
\begin{align*}
\frac{dw}{ds}&=\dot{w}r^{\frac{3}{2}}\cos^{2}\theta=-\frac{1}{2}vw\cos^{2}\theta+\frac{1}{2}r^{3}\dot{\theta}^2\cos^{2}\theta(\frac{d}{d\theta}\frac{\cos^{2}\theta}{c(\theta)})+\cos^{2}\theta V'(\theta) \\
&=\mathcal{W}'(\theta)-\frac{1}{2}vw\cos^{2}\theta+\frac{\sin\theta}{\cos\theta}[2\mathcal{W}(\theta)+\frac{1}{2}r^{3}\dot{\theta}^2\frac{\cos^{3}\theta}{\sin\theta}(\frac{d}{d\theta}\frac{\cos^{2}\theta}{c(\theta)})]\\
&=\mathcal{W}'(\theta)-\frac{1}{2}vw\cos^{2}\theta+\frac{\sin\theta}{\cos\theta}[2\mathcal{W}(\theta)-w^{2}c(\theta)-\frac{1}{2}w^2\frac{\cos\theta}{\sin\theta}c'(\theta)]\\
&=\mathcal{W}'(\theta)-\frac{1}{2}vw\cos^2\theta+\cos\theta\sin\theta(2r+v^2-\frac{1}{2}w^2\frac{c'(\theta)}{\sin\theta\cos\theta}),
\end{align*}
where we use the energy equation in the last line.

\begin{table}[h!]
\centering
\resizebox{\columnwidth}{!}{%
\begin{tabular}{|m{4.5cm}|m{4.5cm}|m{4.5cm}|}
\hline
\includegraphics[width=4cm]{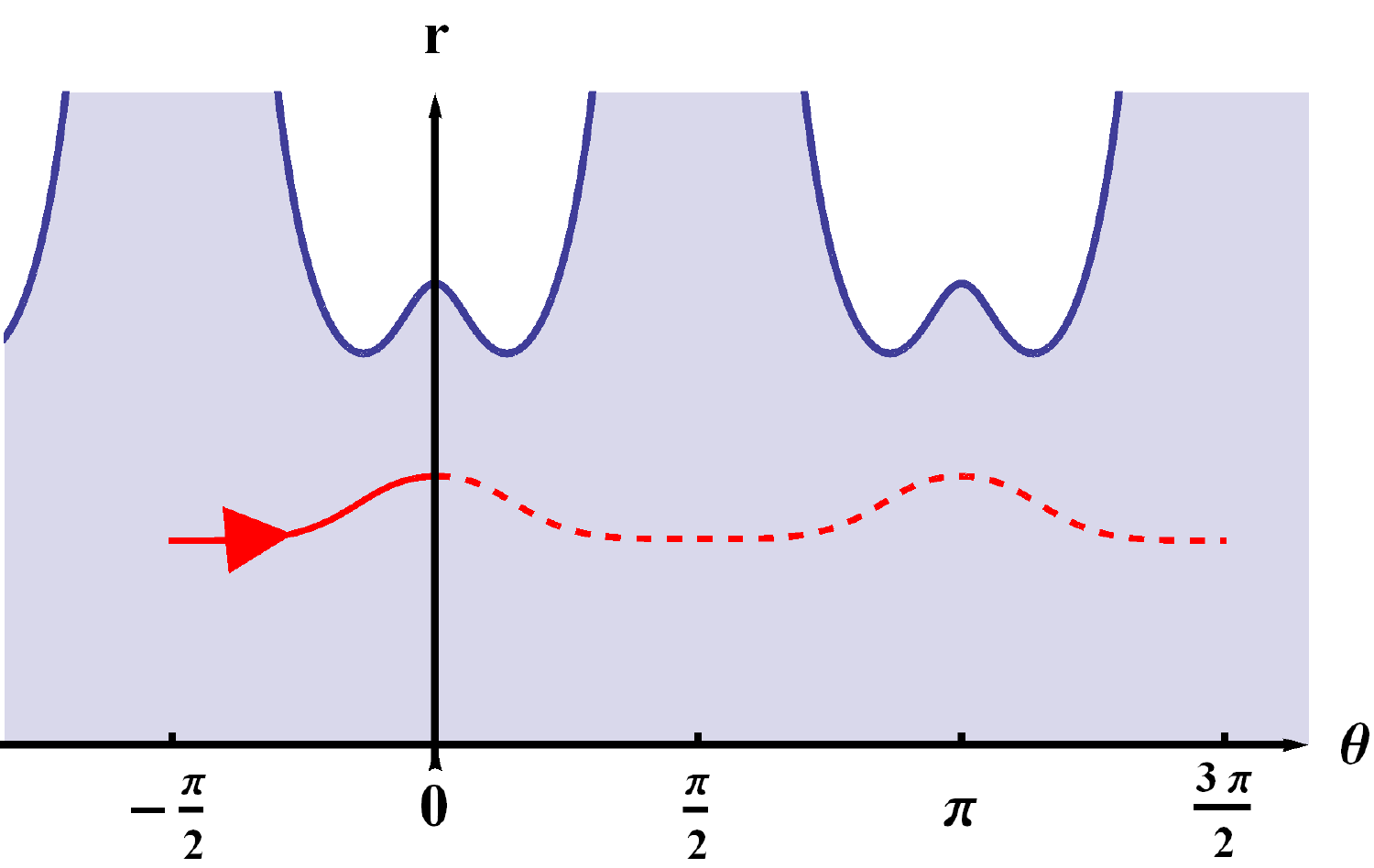} & \includegraphics[width=4cm]{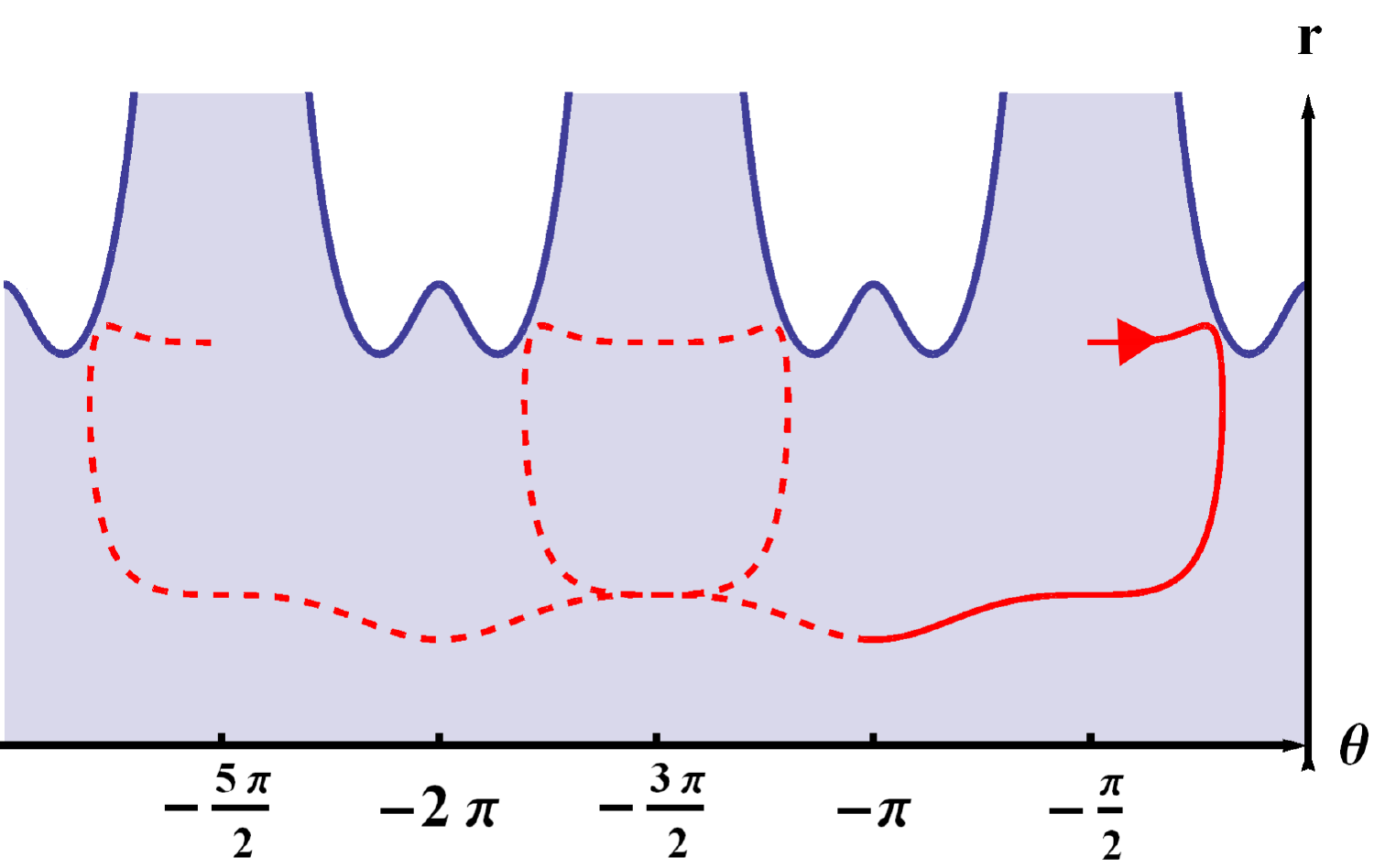} & 
\includegraphics[width=4cm]{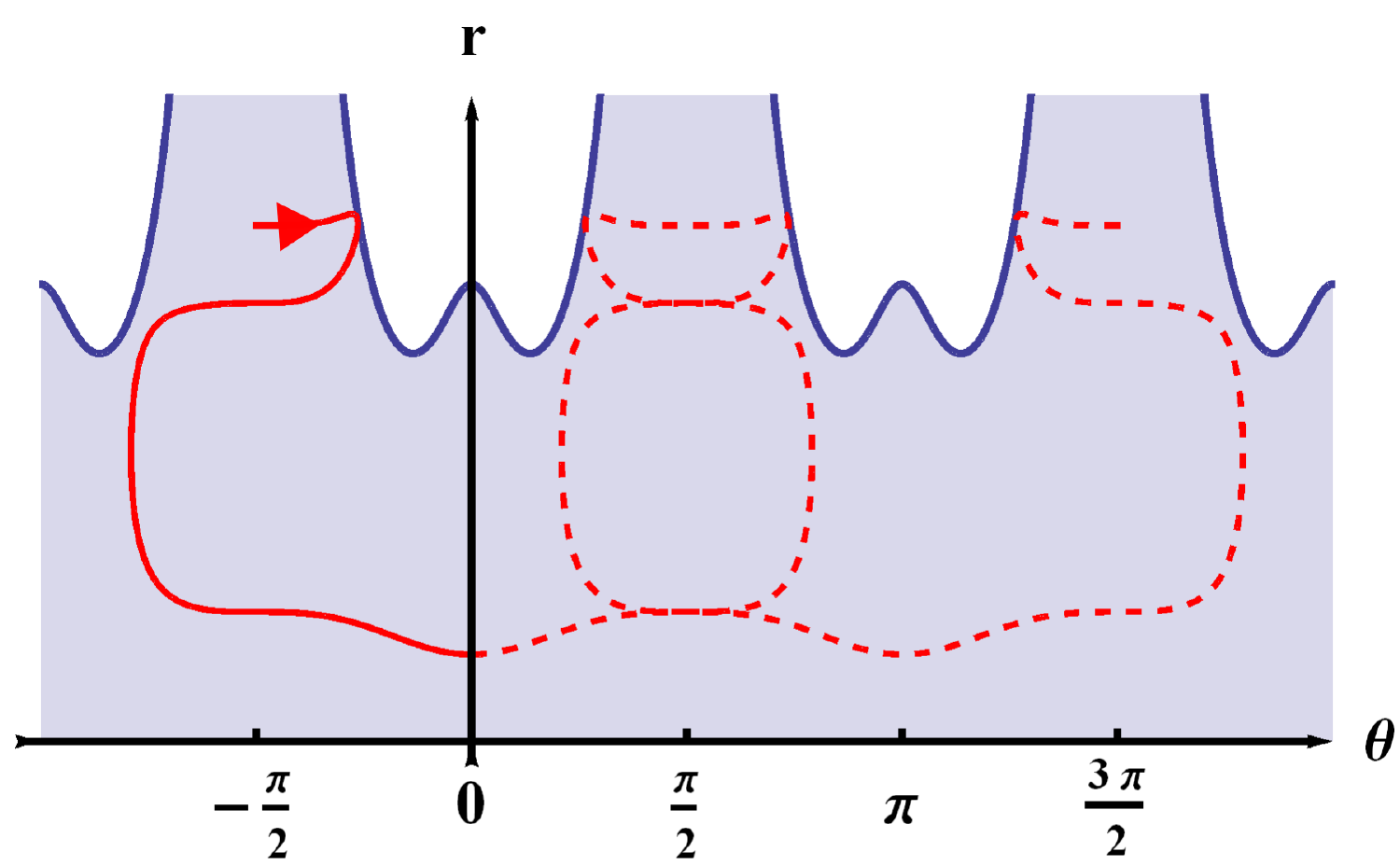} \\

\includegraphics[width=4cm]{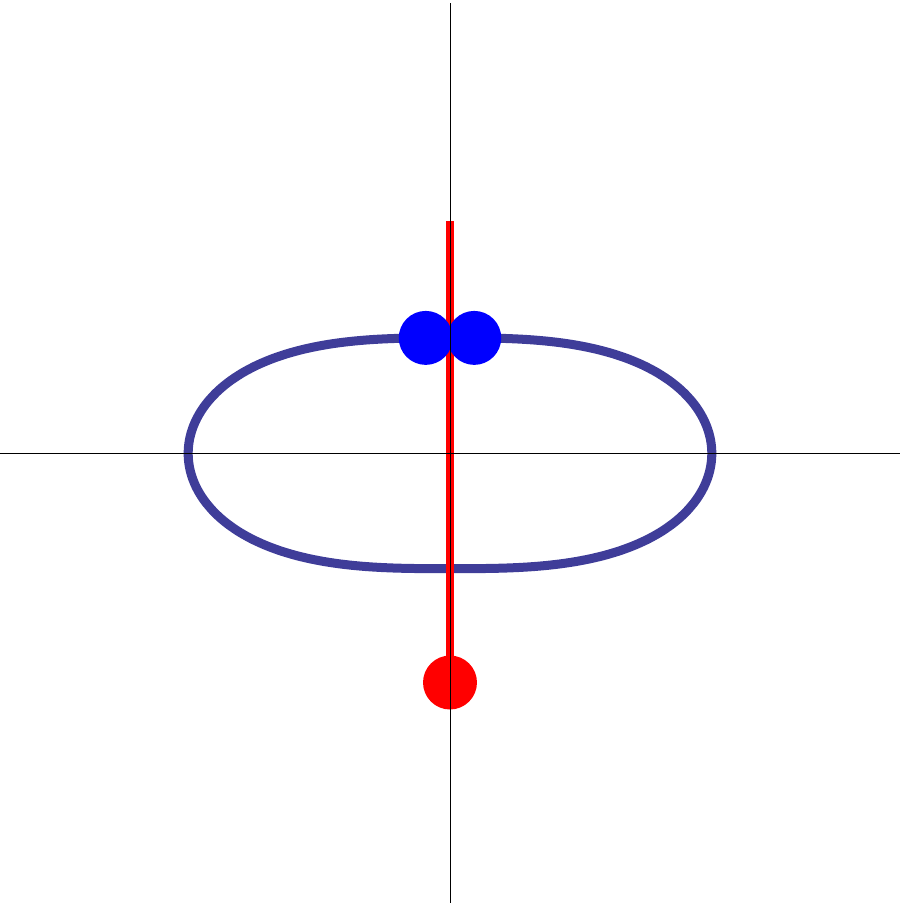} &
\includegraphics[width=4cm]{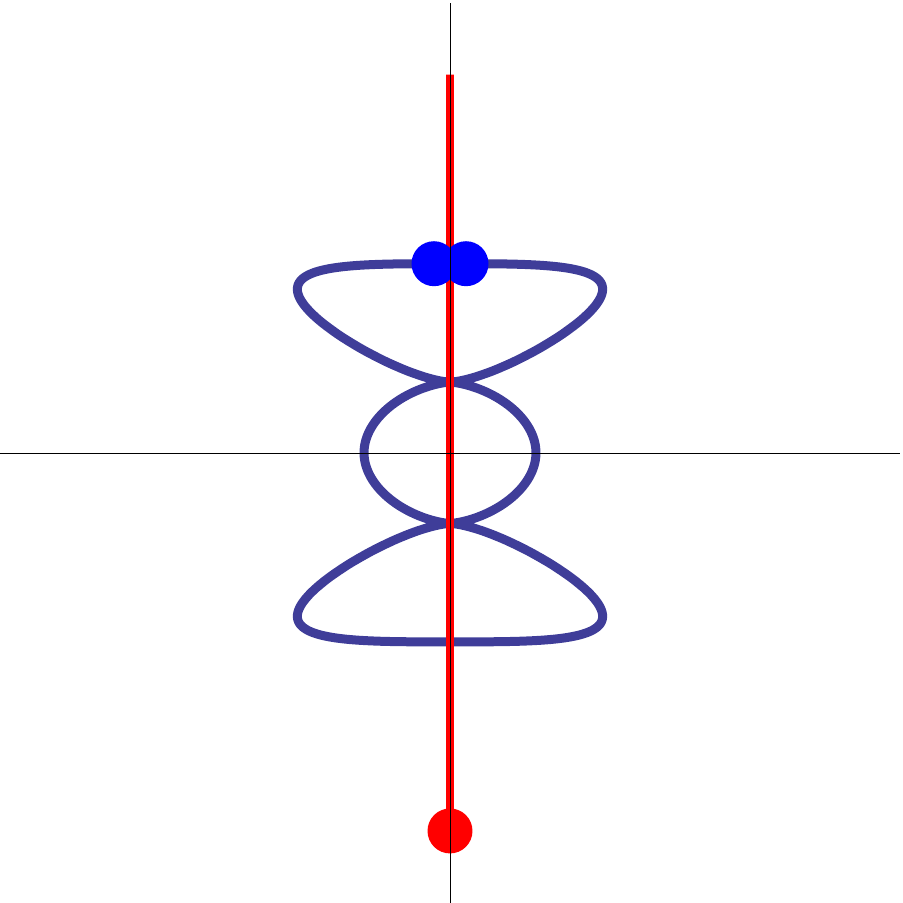}&
\includegraphics[width=4cm]{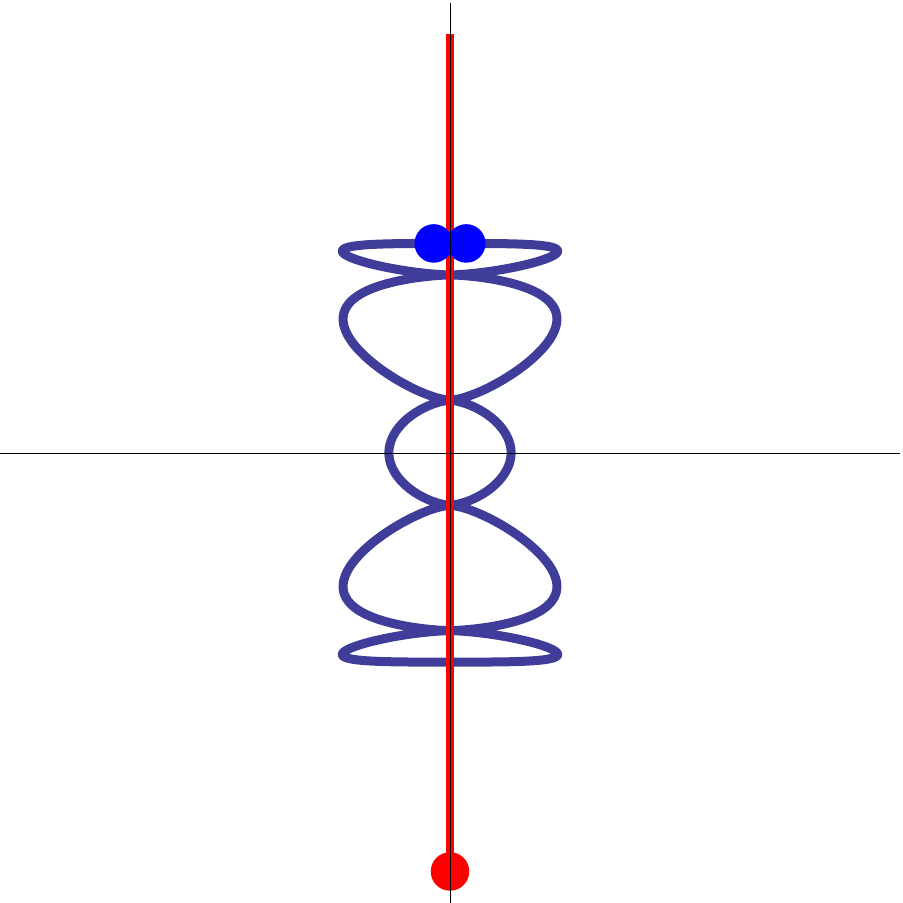}\\

(a) $\mathcal{B}-$family, $k=0$ & (b) $\mathcal{B}-$family, $k=1$ & (c) $\mathcal{B}-$family, $k=2$\\
\hline 
\includegraphics[width=4cm]{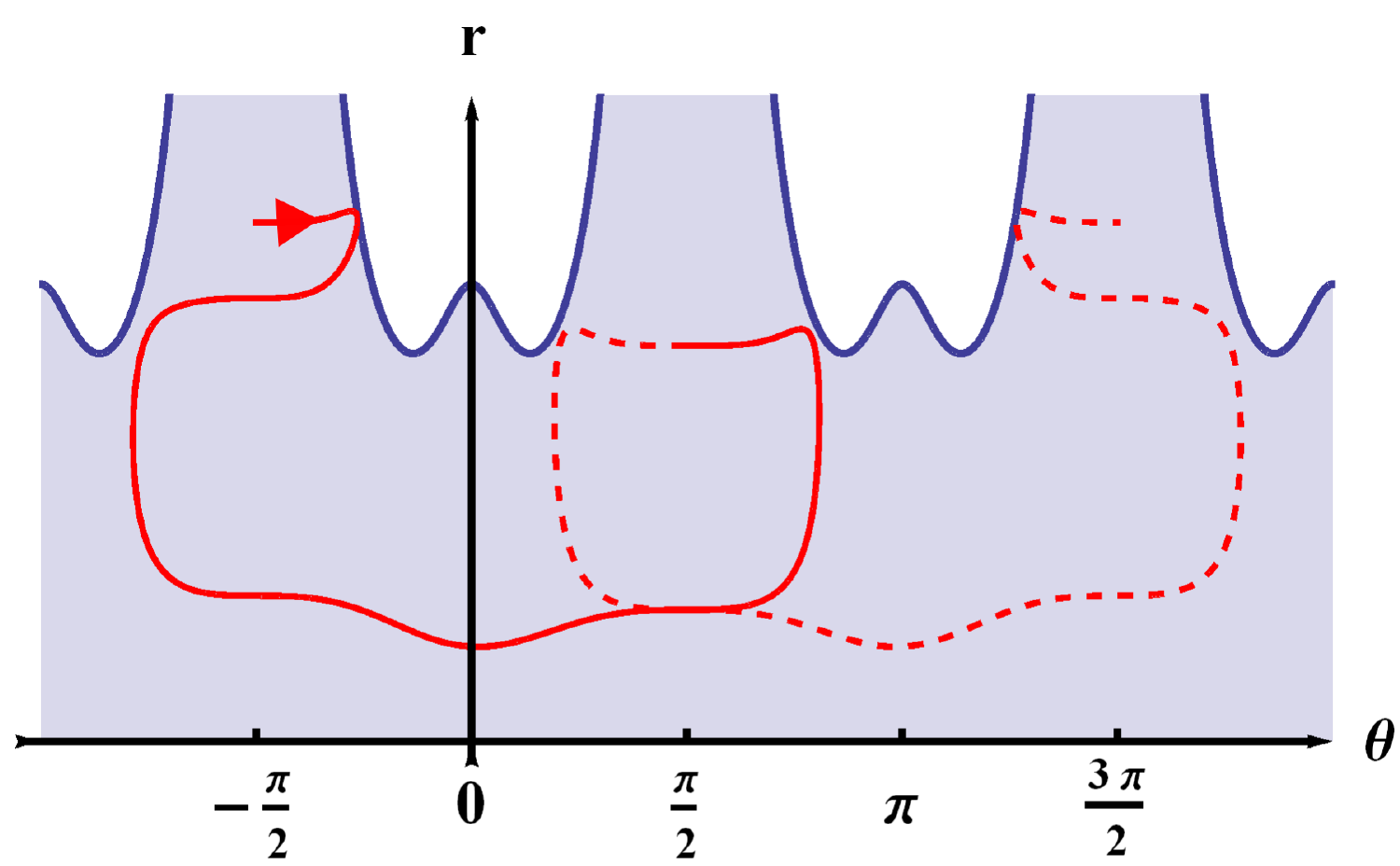} & \includegraphics[width=4cm]{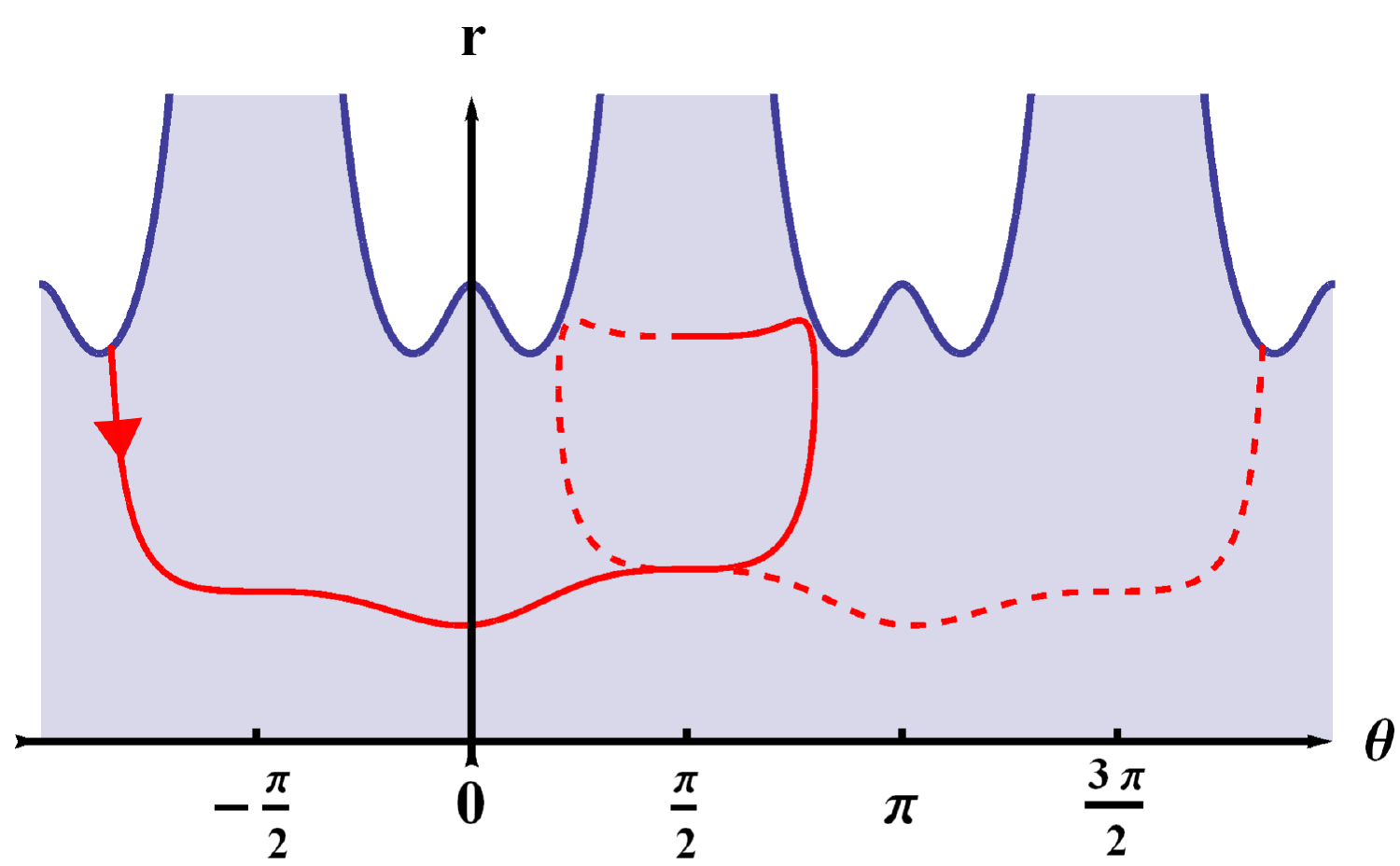} & 
\includegraphics[width=4cm]{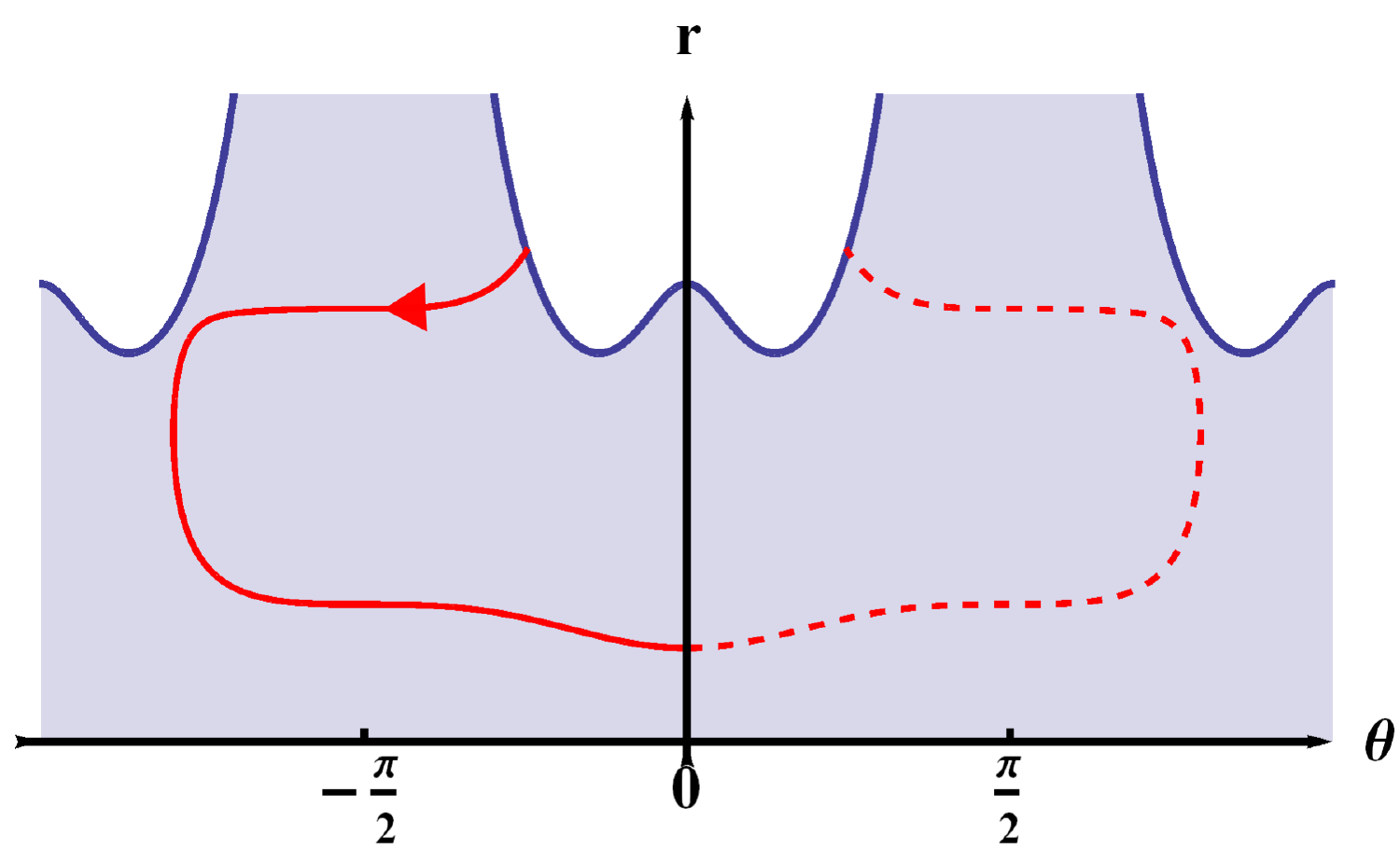} \\

\includegraphics[width=4cm]{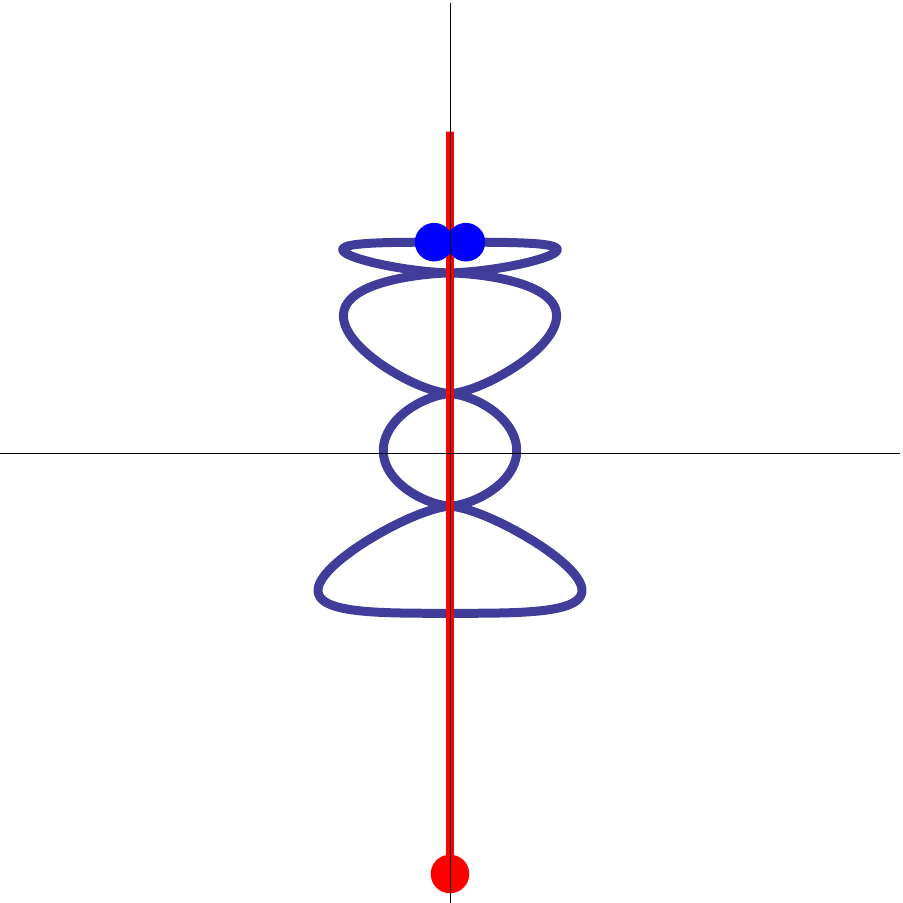} & \includegraphics[width=4cm]{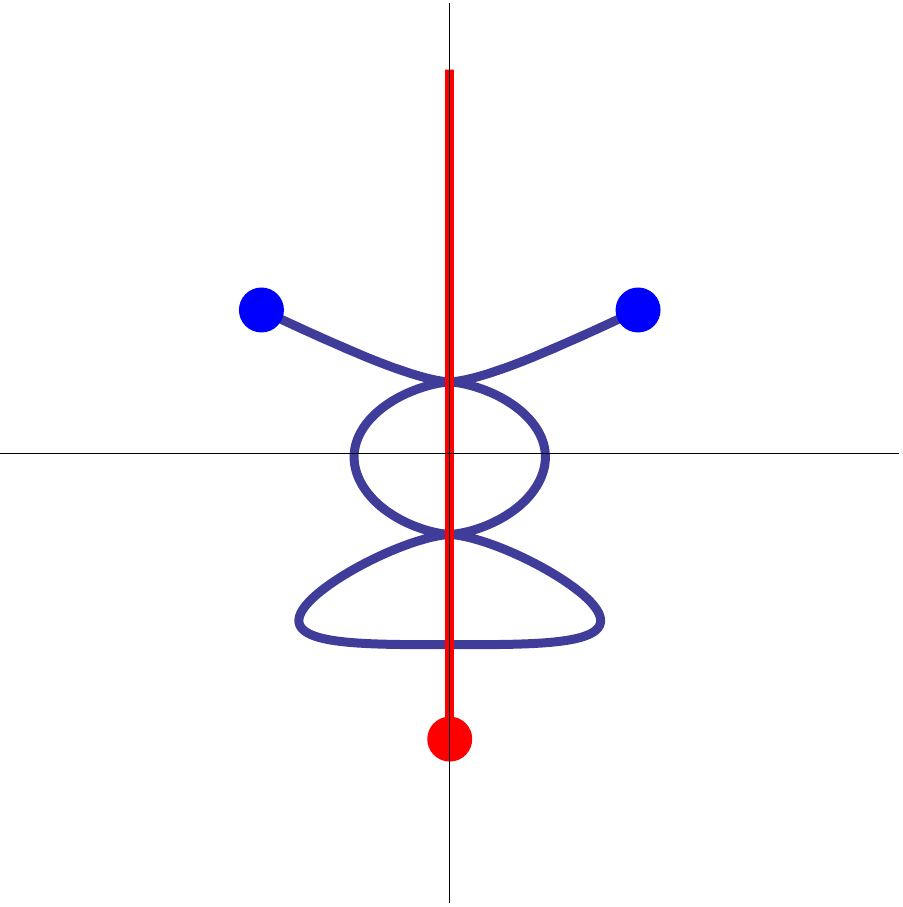} & 
\includegraphics[width=4cm]{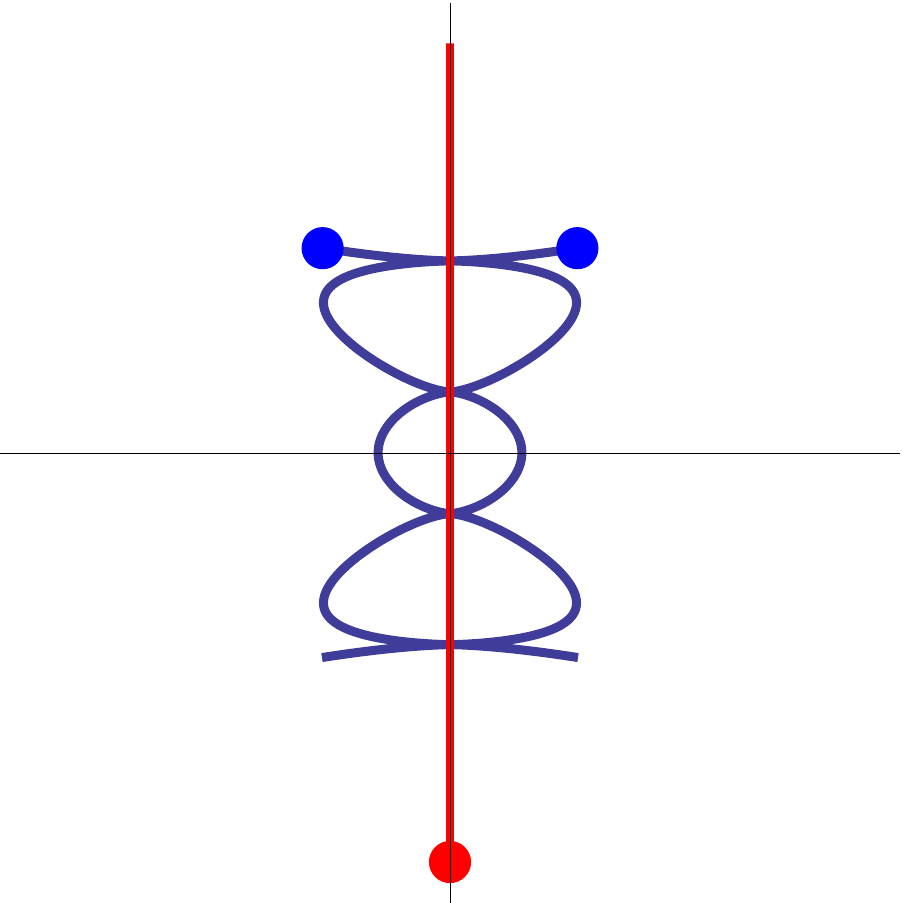} \\

(d) Less-symmetric $\mathcal{B}-$family   \newline \ \ with $i=2,j=1$ & (e) $\mathcal{ZB}-$family \newline \ \  with $i=1,j=1$ & (f) $\mathcal{Z}1-$family with $k=2$\\
\hline

\includegraphics[width=4cm]{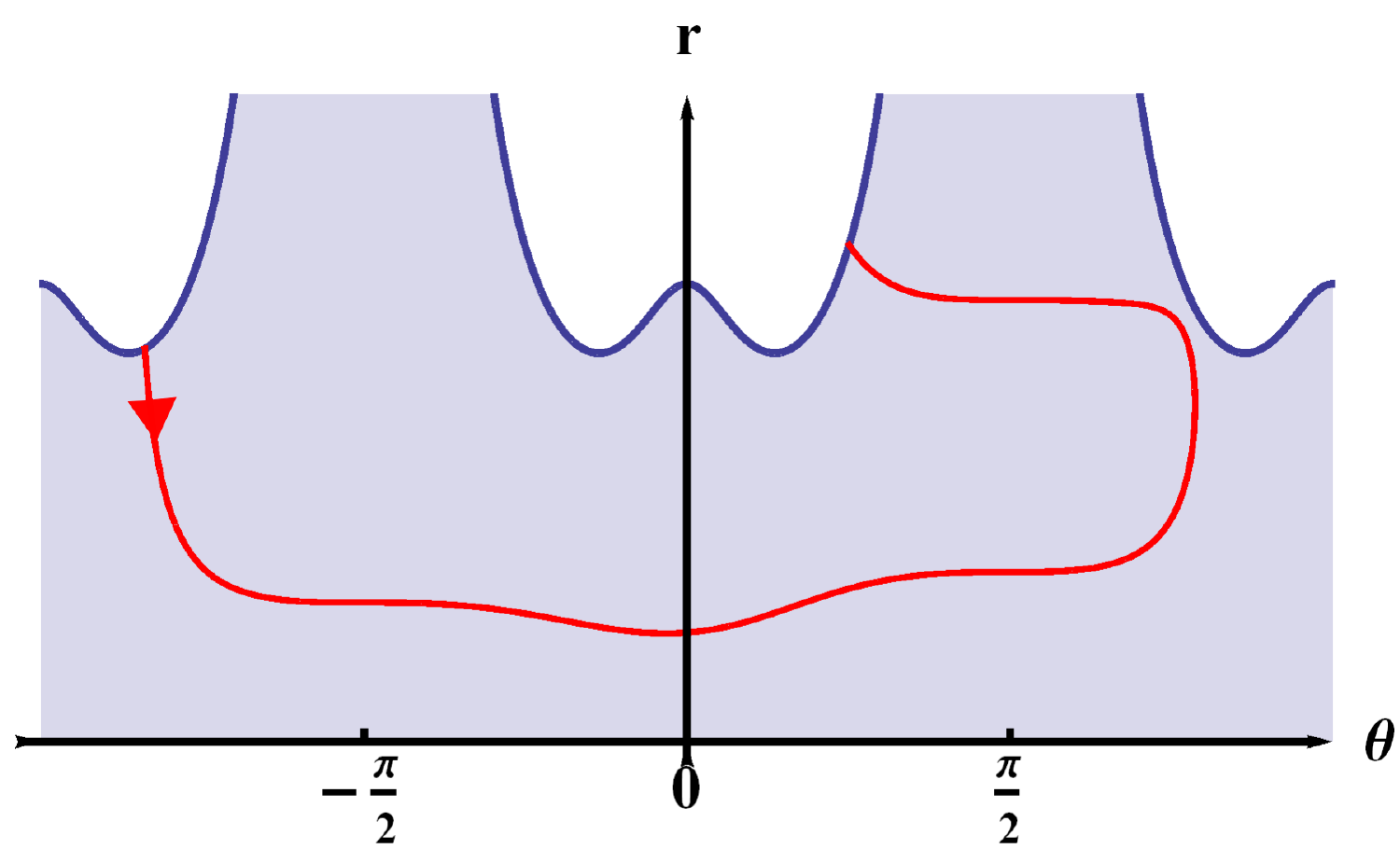} & \includegraphics[width=4cm]{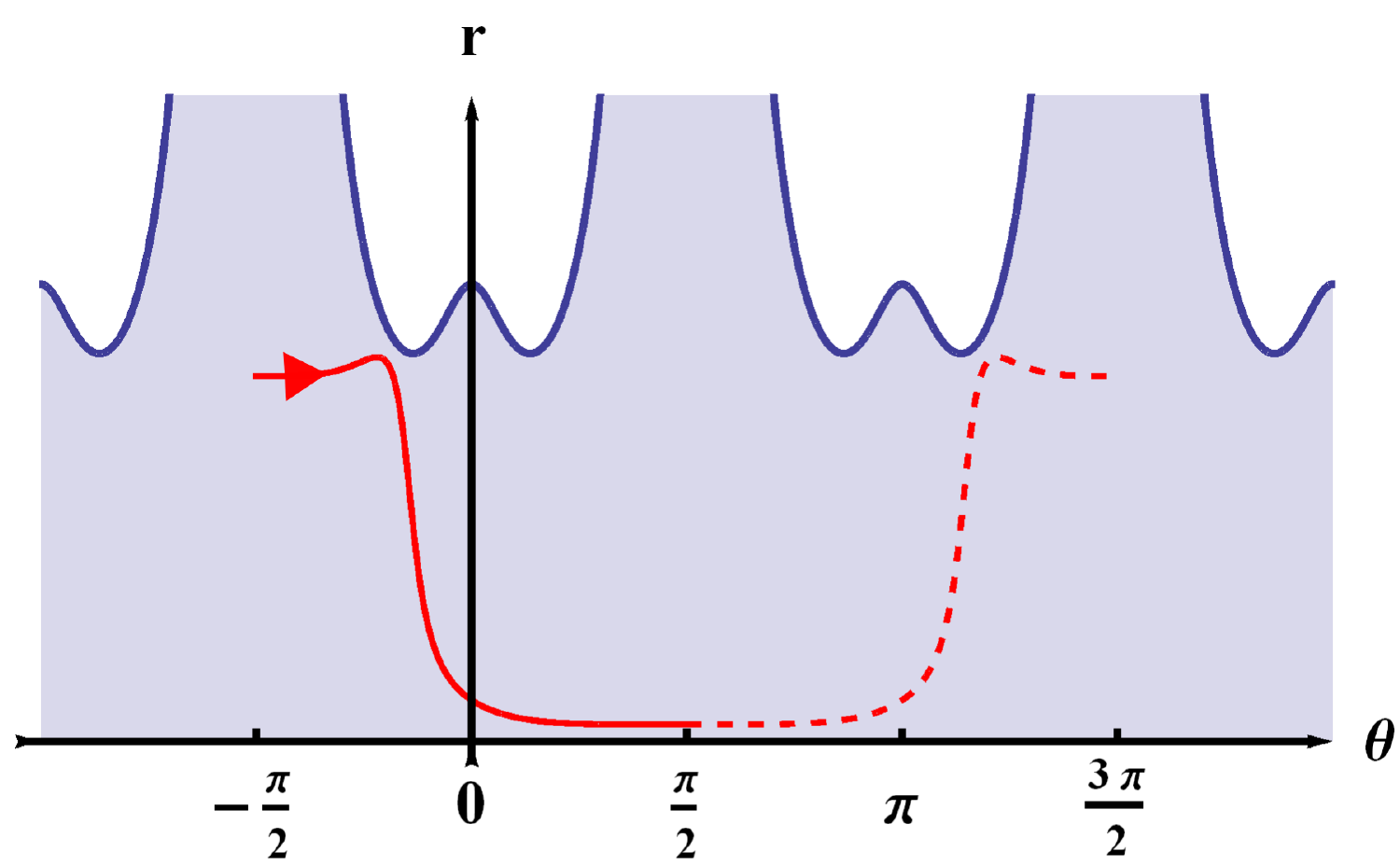} & 
\includegraphics[width=4cm]{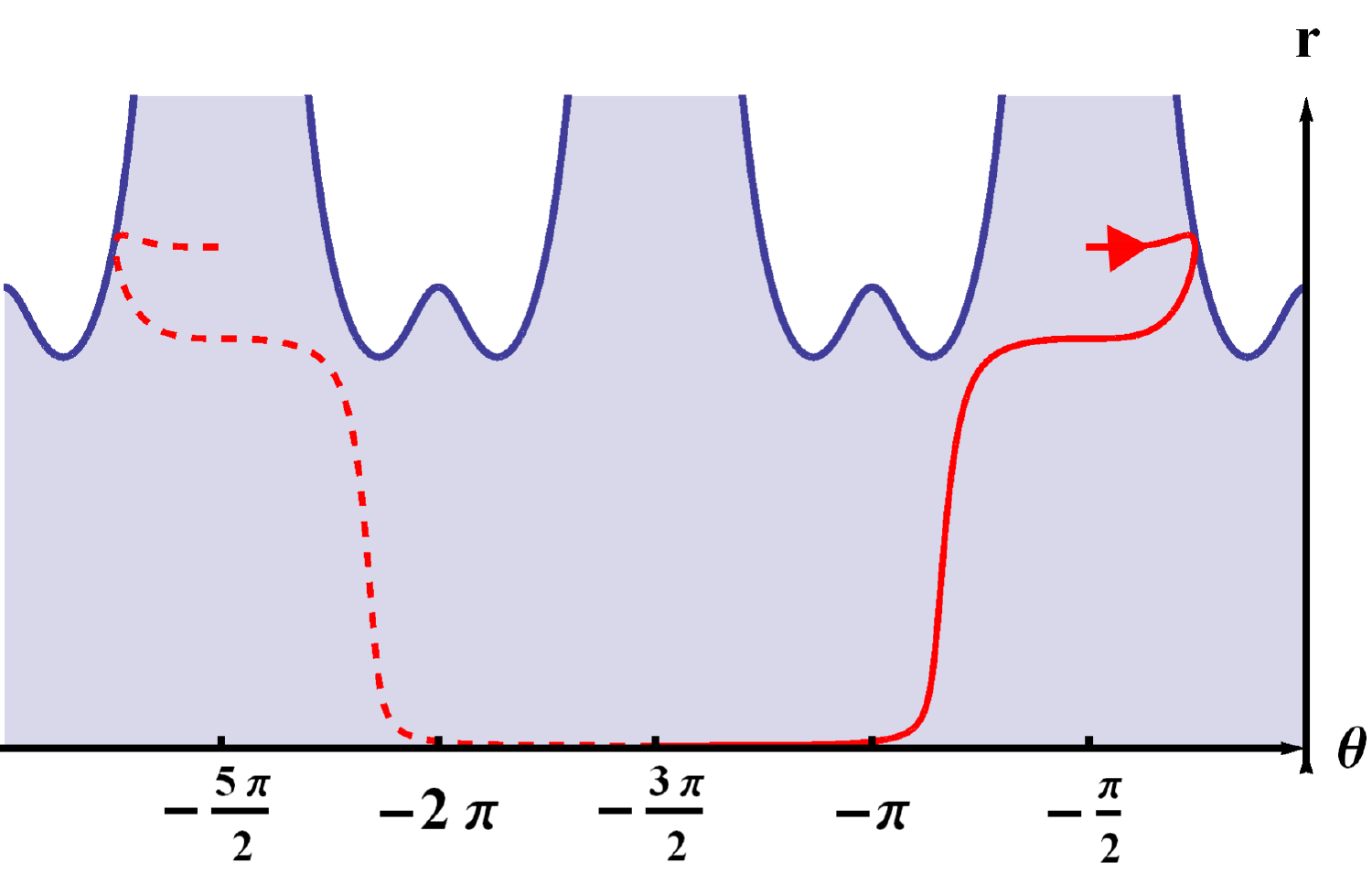} \\

\includegraphics[width=4cm]{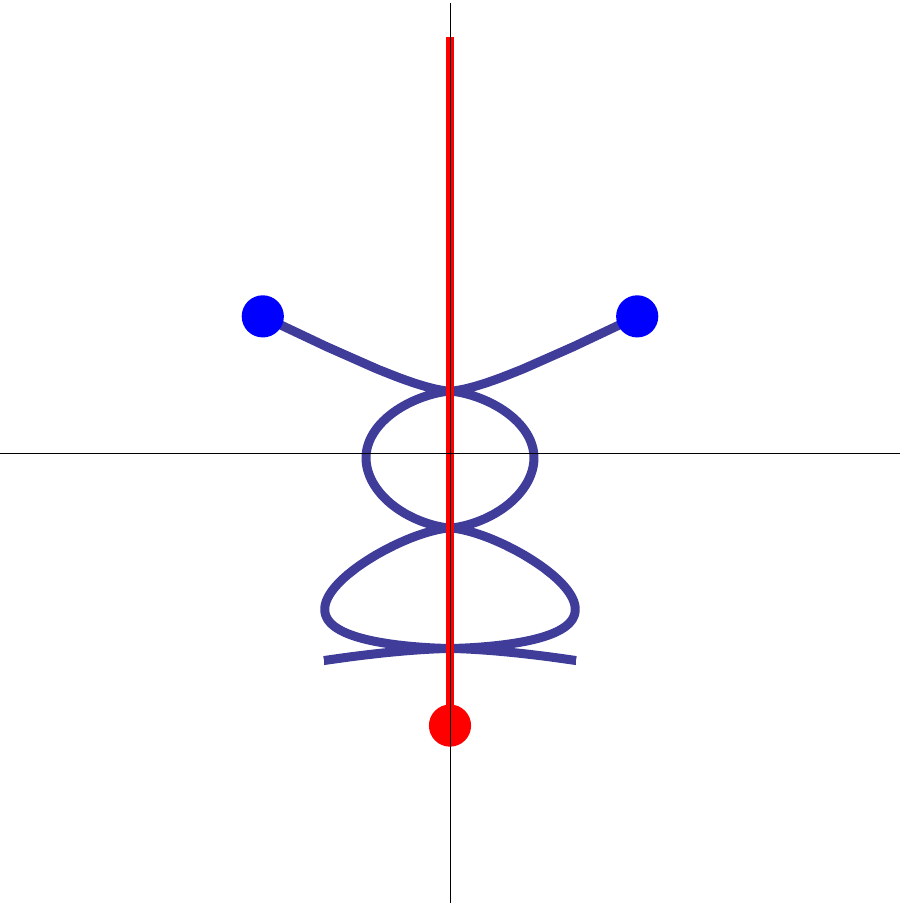} & \includegraphics[width=4cm]{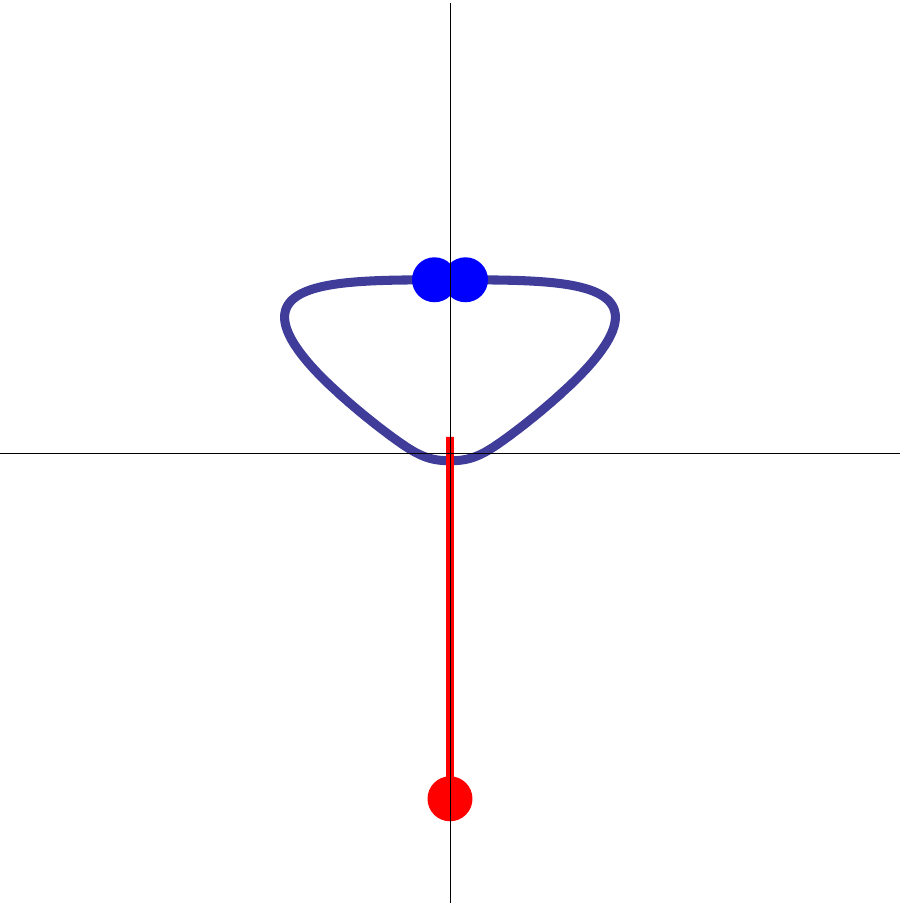} & 
\includegraphics[width=4cm]{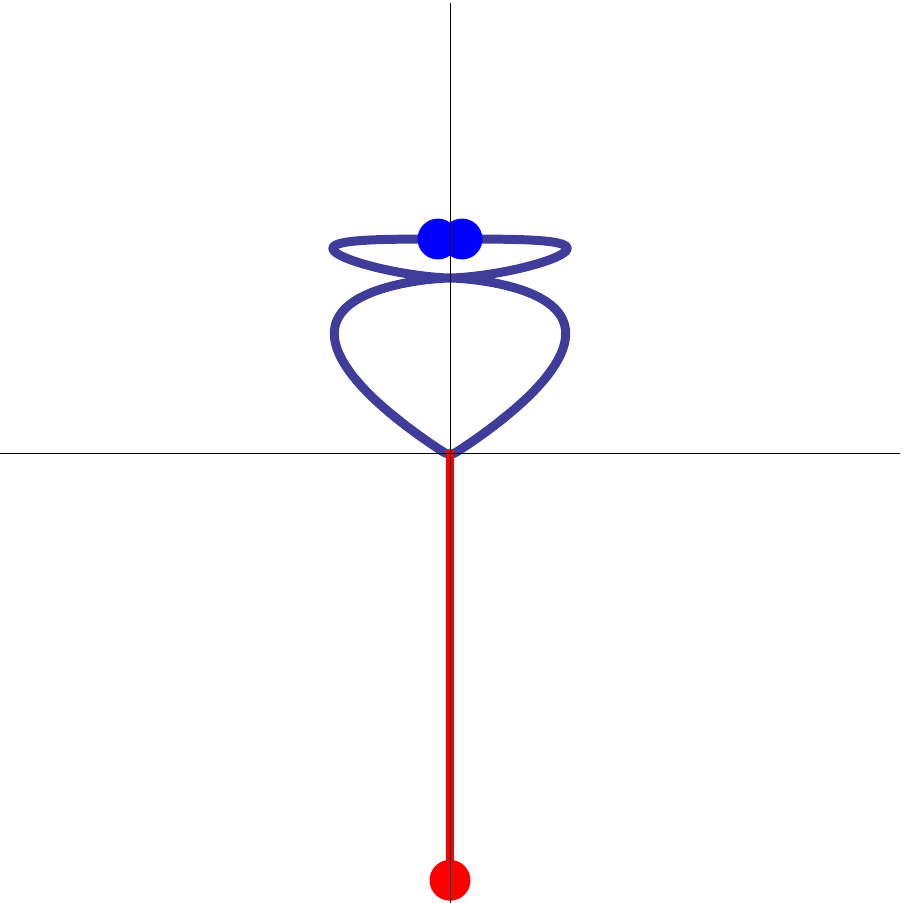} 
\\

(g) $\mathcal{Z}5-$family with \centering $i=1,j=2$&
(h) Unproved orbit &
(i) Unproved orbit\\
\hline
\end{tabular}}
\smallskip
\caption{Periodic orbits and their projection in the $(\theta,r)-$plane. For the projection in the $(\theta,r)-$plane, the fundamental domain of each orbit is plotted in solid curve; one may obtain the full orbit by symmetries.}
\label{orbits_table}
\end{table}

\renewcommand{\arraystretch}{1.5}
\begin{table}[h]
\resizebox{\columnwidth}{!}{%
\centering
 \begin{tabular}{|m{4cm}|m{5cm}|m{5cm}|m{5cm}|}
\hline 

The problem \newline(all equal-mass) & $n$-pyramidal problem & Spatial double-polygon \newline problem & Planar double-polygon \newline problem \\
\hline
Number of masses & $n+1$ & $2n$ & $2n$ \\
\hline
The configuration & {\includegraphics[width=5cm]{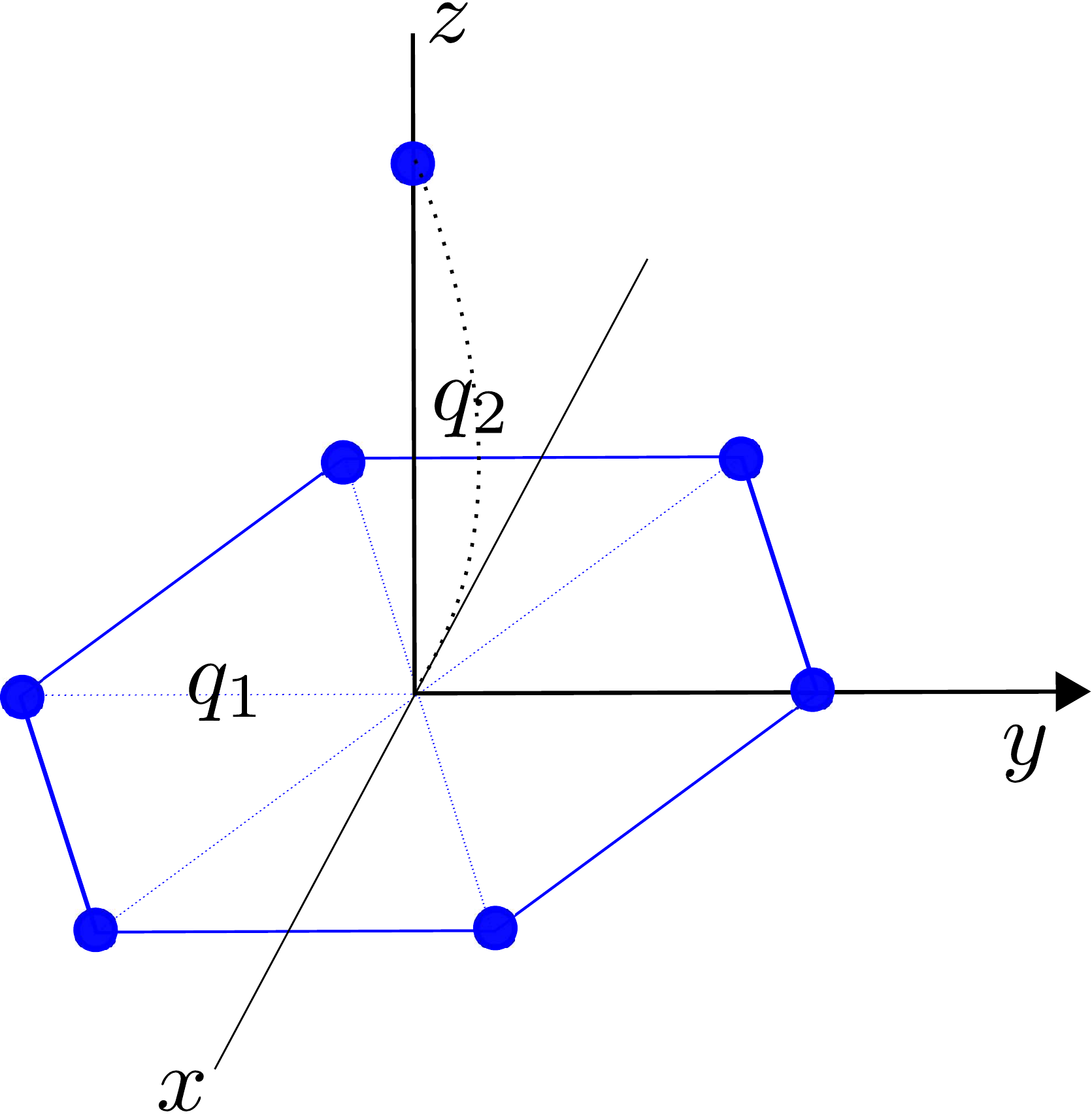}} & \includegraphics[width=5cm]{graph/spatial_gons.pdf} & {\includegraphics[width=5cm]{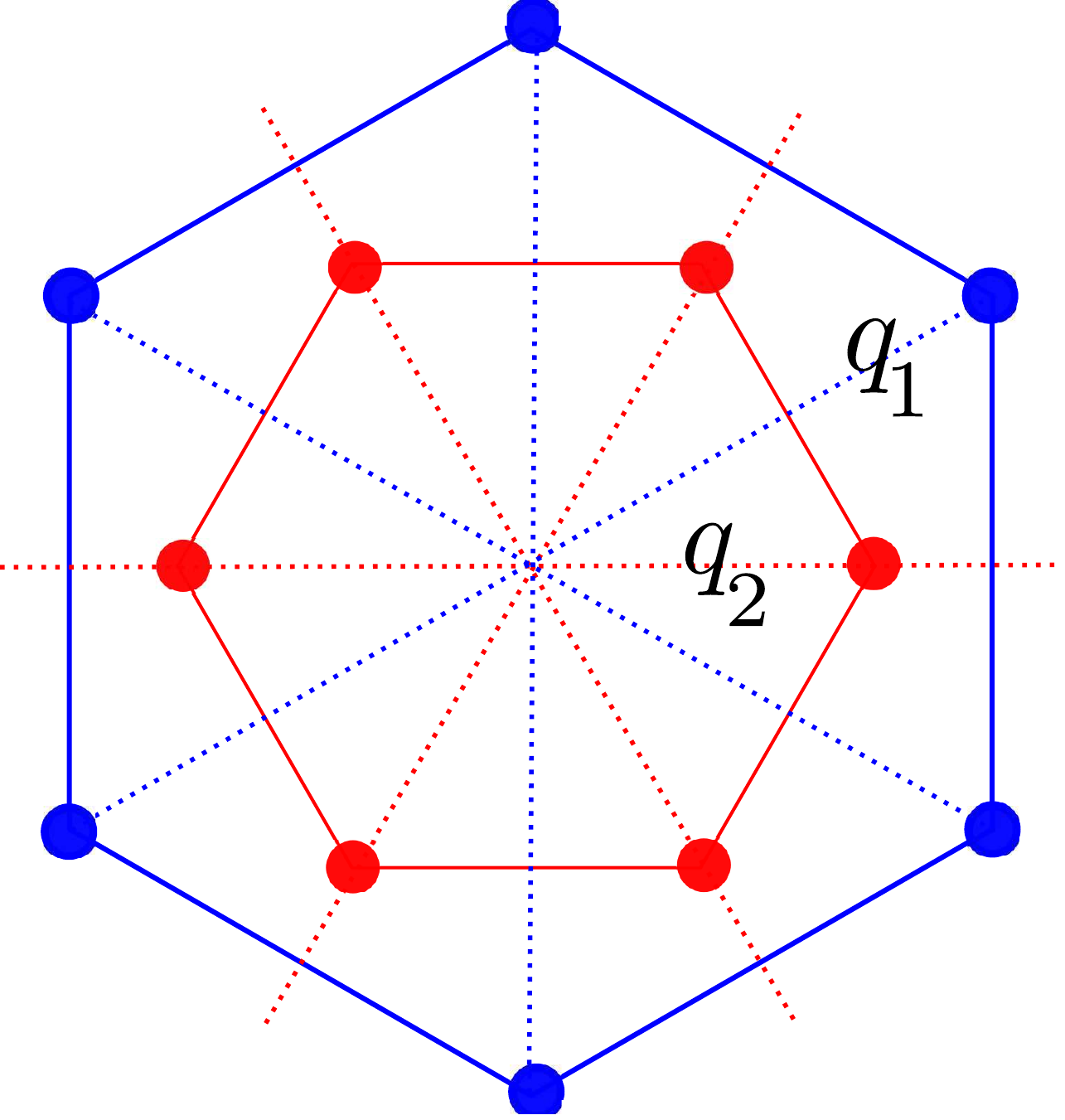}} \\ 
\hline  
$V(\phi)$ has \newline three critical points & when $2\leq n<473$ & when $n\geq 2$ & when $n\geq 3$ \\ 
\hline 
$v_{1}<0$ & true & true & true \\
\hline
$v_{2}>0$ & true & true & not true in general \\
\hline
$v_{3}<0$ & true & true & true \\
\hline
Mart\'{i}nez (M1) & proved in~\cite{martinez2012} & not verified; new problem & proved in ~\cite{martinez2012}\\
\hline
Mart\'{i}nez (M2) & proved in~\cite{martinez2012} & not verified; new problem & proved in ~\cite{martinez2012}\\
\hline
Mart\'{i}nez (M3) & verified in~\cite{martinez2012} & not verified; new problem &  (M3) fails \\
\hline
Mart\'{i}nez (M4) & computed in~\cite{martinez2013} only for some $n$ & not verified; new problem & computed in~\cite{martinez2013} only for some $n$ \\
\hline
Condition (N1) & holds & holds & not applicable \\ 
\hline 
Condition (N2) & holds, since (M2) holds & holds & holds, since (M2) holds \\ 
\hline 
Condition (N3) & proved true for $n\geq 4$& numerically true for $n\geq 7$ \newline proved true for $n\geq 10$  & not applicable \\ 
\hline
Condition (N3') & numerically verified for $n=2,3$ & case $n=2$ is treated in~\cite{vidal1999} \newline numerically verified here for $3\leq n\leq 9$ & not applicable \\ 
\hline
Orbits proved \newline to exist & $\mathcal{B}-$family \newline $\mathcal{Z}1-$family & $\mathcal{B}-$family \newline $\mathcal{Z}1-$family & $\mathcal{B}-$family with $k=0$ \\
\hline
Additional families if\newline provided (N4): $v_{2}\neq -v_{3}$ & Less-symmetric $\mathcal{B}-$family\newline $\mathcal{Z}5-$family \newline $\mathcal{ZB}-$family &  Less-symmetric $\mathcal{B}-$family\newline $\mathcal{Z}5-$family \newline $\mathcal{ZB}-$family & none is proved \\
\hline
Other cases \newline (unequal-mass): &
\multicolumn{3}{m{15cm}|}{
(1) Isosceles three-body problem. $m_{1}=m_{2}=1, m_{3}=\mu$, for $\epsilon_{1}\approx 0.379 <\mu<\epsilon_{2}\approx 2.662$, there exist Type 1 to 6 periodic brake orbits, $\mathcal{B}-$family, $\mathcal{ZB}-$family, and less-symmetric $\mathcal{B}-$family Schubart-like periodic orbits.\newline 
(2) n-pyramidal problem. $m_{1}=m_{2}\cdots=m_{n}=1, m_{n+1}=\mu$, for any $\mu<0$, there exists a $\mathcal{B}-$family Schubart-like orbit with $k=0$. }\\
\hline
\end{tabular}}
\caption{Summary of Results}
\label{summary}
\end{table}

\section*{Acknowledgments} This research is supported by NSF grant DMS-1208908.


\medskip
Received xxxx 20xx; revised xxxx 20xx.
\medskip

\end{document}